\documentclass[10pt]{article}
\RequirePackage[left=28mm,right=28mm,top=35mm,bottom=30mm]{geometry}
\usepackage[utf8]{inputenc}
\usepackage[english]{babel}
\usepackage{amsmath}
\usepackage{amssymb}
\usepackage{amsthm}
\usepackage[shortlabels]{enumitem} 
\usepackage[square,sort,comma,numbers]{natbib}
\usepackage{afterpage}
\usepackage{graphicx} 
\usepackage{tikz-cd} 
\usepackage{stmaryrd} 
\usepackage{url}
\usepackage{hyperref}
\usepackage{multicol}
\usepackage{xcolor}
\usepackage[toc,page, title, titletoc]{appendix}

\newcommand{\projr}{\!\operatorname{proj-}\!\!}  
\newcommand{\End}{\operatorname{End}}
\newcommand{\Hom}{\operatorname{Hom}}
\newcommand{\gldim}{\operatorname{gldim}}
\newcommand{\findim}{\operatorname{findim}}
\newcommand{\Ext}{\operatorname{Ext}}
\newcommand{\Tor}{\operatorname{Tor}}
\newcommand{\add}{\!\operatorname{add}}
\newcommand{\pdim}{\operatorname{pdim}}
\newcommand{\m}{\!\operatorname{-mod}} 
\newcommand{\rmod}{\operatorname{mod-}\!\!}
\newcommand{\proj}{\!\operatorname{-proj}}
\newcommand{\St}{\Delta}
\newcommand{\Cs}{\nabla}
\renewcommand{\L}{\Lambda}
\renewcommand{\l}{\lambda}
\newcommand{\pri}{\mathfrak{p}} 
\newcommand{\id}{\operatorname{id}}
\newcommand{\mi}{\mathfrak{m}} 
\newcommand{\rad}{\operatorname{rad}}
\newcommand{\Spec}{\operatorname{Spec}}
\newcommand{\rank}{\operatorname{rank}} 
\renewcommand{\top}{\operatorname{top}}

\newcommand{\MaxSpec}{\operatorname{MaxSpec}} 
\newcommand{\Stsim}{\tilde{\St}}
\newcommand{\im}{\!\operatorname{im}} 
\newcommand{\Ann}{\operatorname{Ann}}
\newcommand{\coker}{\operatorname{coker}}
\newcommand{\sumSt}{\underset{\l\in\L}{\bigoplus}}
\newcommand{\soc}{\operatorname{soc}}
\newcommand{\characteristic}{\operatorname{char}}

\newtheorem{numberingthm}{Theorem}[subsection] 
\theoremstyle{definition}
\newtheorem{Def}[numberingthm]{Definition}
\newtheorem{example}[numberingthm]{Example}

\theoremstyle{plain}
\newtheorem{Prop}[numberingthm]{Proposition}
\newtheorem{Theorem}[numberingthm]{Theorem}
\newtheorem{Cor}[numberingthm]{Corollary}
\newtheorem{Lemma}[numberingthm]{Lemma}
\newtheorem{Remark}[numberingthm]{Remark}

\newenvironment{Example}
{\pushQED{\qed}\example}
{\popQED\endexample}

\theoremstyle{remark}
\newtheorem{Observation}[numberingthm]{Observation}
\newtheorem{Notation}[numberingthm]{Notation}

\theoremstyle{empty}
\newtheorem*{thmintroduction}{Theorem}
\providecommand{\keywords}[1]
{\scriptsize
	\textbf{\textit{Keywords:}} #1 \normalsize \hfill
}
\providecommand{\msc}[1]
{\scriptsize
	\textbf{\textit{2020 Mathematics Subject Classification:}} #1 \normalsize \hfill
}

\title{Cellular Noetherian algebras with finite global dimension are split quasi-hereditary} 
\author{Tiago Cruz} 
\date{}

\newcommand{\Address}{{
		\bigskip
		\footnotesize
		
		TIAGO CRUZ,\par \textsc{Institute of Algebra and Number Theory}\par \textsc{University of Stuttgart,}\par \textsc{Pfaffenwaldring 57, 70569 Stuttgart, Germany,}\par\nopagebreak
		\textit{E-mail address}, T.~Cruz: \texttt{tiago.cruz@mathematik.uni-stuttgart.de}
		}}
		
		\allowdisplaybreaks

\begin{document}

\maketitle
\vspace*{-2em} 

\begin{abstract}
We prove that cellular Noetherian algebras with finite global dimension are split quasi-hereditary over a regular commutative Noetherian ring with finite Krull dimension and their quasi-hereditary structure is unique, up to equivalence.
In the process, we establish that a split quasi-hereditary algebra is semi-perfect if and only if the ground ring is a local commutative Noetherian ring. We give a formula to determine the global dimension of a split quasi-hereditary algebra over a commutative regular Noetherian ring (with finite Krull dimension) in terms of the ground ring and finite-dimensional split quasi-hereditary algebras. For the general case, we give upper bounds for the finitistic dimension of split quasi-hereditary algebras over arbitrary commutative Noetherian rings. We apply these results to  Schur algebras over regular Noetherian rings and to Schur algebras over quotients rings of the integers.
\end{abstract}

 \tableofcontents

\section{Introduction}\label{Introduction}

Two main problems in representation theory of algebras are to classify the simple modules over an algebra and to understand the homological properties of an algebra.
Cellular algebras $B$ are certain algebras characterized by the existence of an involution $i$ with $i^2=\textrm{id}_B$ and a certain chain of ideals that provide a filtration of the regular module $B$. They were introduced by Graham and Lehrer \citep{zbMATH00871761}, so that the cellular structure can be used to reduce problems of representation theory to problems of linear algebra. For example, the parametrization of simple modules over finite-dimensional cellular algebras is reduced to a problem in linear algebra. Several important classes of algebras are cellular, in particular,  Iwahori-Hecke algebras, Brauer algebras, BMW algebras.
Furthermore, Schur algebras and block algebras of the category $\mathcal{O}$ of a complex semi-simple Lie algebra are both cellular and split quasi-hereditary. Over (split) quasi-hereditary finite-dimensional algebras we can also deduce properties from a certain chain of ideals, like the number of simple modules, but also homological properties like finiteness of global dimension.

In \cite{zbMATH01384521}, Koenig and Xi showed that a finite-dimensional cellular algebra is split quasi-hereditary if and only if it has finite global dimension. Previously in \cite{zbMATH01218863}, they observed that (split) quasi-hereditary algebras are cellular exactly when there is an involution that fixes a set of primitive idempotents defining the ideals that give the quasi-hereditary structure. Recently \cite{zbMATH07203140}, Coulembier proved that cellular finite-dimensional algebras which are (split) quasi-hereditary have a unique quasi-hereditary structure. 

All of these considerations so far were for finite-dimensional algebras. What can we say in the integral setup? 
The definition of cellular algebras also makes sense for Noetherian algebras over commutative Noetherian rings and they even possess a base change property. Moreover, the classical examples of cellular algebras also admit integral versions which are again cellular. But, unfortunately, in the literature, the properties mentioned above are not studied for Noetherian cellular algebras. 
In the literature for the integral setup for cellular algebras slightly different directions were taken. For instance, the homological structure and simple modules were studied for a generalization of cellular algebras known as affine cellular algebras introduced in \citep{zbMATH05994238}. 
But, unfortunately, affine cellular algebras do not behave nicely enough under change of rings in contrast to cellular algebras and they might fail to be Noetherian even if the ground ring is Noetherian. In particular, the methods used to study cellular and affine cellular algebras are distinct.

 The concept of (split) quasi-hereditary algebra admits a generalisation to Noetherian algebras due to \cite{CLINE1990126}. This concept was used as the setup of split quasi-hereditary cover theory in \cite{Rouquier2008}. Other approaches to generalise this concept can be found in \cite{zbMATH01198520, zbMATH02037023, zbMATH07203067, zbMATH06751586}.  All these approaches yield equivalent concepts  for Noetherian algebras over local commutative Noetherian rings. The affine case for quasi-hereditary algebras (over Noetherian Laurentian graded algebras) was studied in \citep{zbMATH06434447}.
So, we can ask:
\begin{enumerate}[(Q1)]
	\item Are cellular algebras  of finite global dimension over a commutative Noetherian ring split quasi-hereditary?\label{IQ1}
	\item If so, does the uniqueness of the quasi-hereditary structure still hold in the integral setup?
	\item When is a split quasi-hereditary algebra over a commutative Noetherian ring a cellular algebra?
\end{enumerate}

The aim of this paper is to give answers to these questions and to expand the theory of integral split quasi-hereditary algebras. 
Question (Q1) is answered positively in Theorem \ref{cellularfiniteglobaldimension} and extended in Theorem \ref{cellularfiniteglobaldimensiontwo} to include a characterization of cellular Noetherian algebras of finite global dimension also in terms of Cartan matrices in the spirit of \cite{zbMATH01384521}:
	\begin{thmintroduction}
	Let $R$ be a commutative Noetherian regular ring with finite Krull dimension. Let $A$ be a cellular $R$-algebra with cell datum $(\L, M, C, \iota)$. 
	Then, the following assertions are equivalent:
	\begin{enumerate}[(i)]
				\item $A$ together with the cell modules labelled by the poset $\L$ with reversed order is a split quasi-hereditary algebra;
		 \item $A$ is locally semi-perfect and for every prime ideal $\pri$ of $R$,  $$\det [\rank_{R_\pri} \Hom_{A_\pri}(P_i, P_j) ]=1, $$ where $P_i$, $i=1, \ldots, r$ for some natural number $r$, are the projective indecomposable modules of $A_\pri$;
		\item $A$ has finite global dimension.
	\end{enumerate}
\end{thmintroduction}

The main idea behind this theorem is that over cellular finite-dimensional algebras for which integral versions exist,  the simple modules can be obtained by extension of scalars from a module over the integral cellular algebra which is also finitely generated projective over the ground ring. This fact is quite rare. Since split quasi-hereditary algebras behave quite well under change of rings  we can consider an integral version of the Grothendieck group of a split quasi-hereditary algebra which completely determines the Grothendieck groups of finite-dimensional split quasi-hereditary algebras. Combining this with the arguments of \cite{zbMATH07203140}, we obtain in Theorem \ref{Coulembieruniqueness} a positive answer to (Q2). 
Here, in the process to establish (Q1) and (Q2) we suggest a new proof for an opposite algebra of a split quasi-hereditary algebra being again split quasi-hereditary and how split heredity chains behave under techniques of change of rings (see Subsection \ref{Relation between heredity chains and standard modules}). Here, the split condition makes the argument for the opposite algebra of a split quasi-hereditary being again split quasi-hereditary a bit more technical and need more details than just the non-split case. The non-split version can be found in \cite{CLINE1990126}. Concerning change of rings, we make available more criteria to test if a given algebra is split quasi-hereditary using change of rings. In particular, we show that a quasi-hereditary algebra remains quasi-hereditary with the same quasi-hereditary structure under change of rings if and only if it is a  split quasi-hereditary algebra (see Subsection \ref{Split highest weight categories under change of rings}).  For example, all quasi-hereditary algebras over algebraically closed fields are split quasi-hereditary algebras. To obtain an analogue of the classification of cellular Noetherian algebras being split quasi-hereditary in terms of Cartan matrices we show in Subsection \ref{Split quasi-hereditary algebras and existence of projective covers} that, in general, split quasi-hereditary algebras are locally semi-perfect. This means that split quasi-hereditary algebras over local commutative Noetherian rings are semi-perfect algebras making the concept of projective cover defined in such a setup.

The answer of (Q3) relies on  the existence of a certain set of orthogonal idempotents that under tensoring with residue fields makes the one labelled by a maximal element a primitive idempotent. Furthermore,  from such set after tensoring with a residue field, a complete set of primitive orthogonal idempotents can be extracted (see Section \ref{Cellular algebras}). For example, there exists such a set of idempotents over Schur algebras.

Under this methodology, we also generalise Theorem 2 of \cite{zbMATH02164791} about finitistic dimensions of split quasi-hereditary algebras (for more details see Theorems ~\ref{finitisticqh} and ~\ref{globaldimensionqh}):
\begin{thmintroduction} Let $R$ be a commutative Noetherian ring.
	Let $A$ be a split quasi-hereditary $R$-algebra. Then,
	\begin{align*}
	\findim A&\leq  \findim R +\sup\{\gldim A/\mi A\colon \mi \text{ is a maximal ideal of } R \}\\&\leq \dim R +\sup\{\gldim A/\mi A\colon \mi \text{ is a maximal ideal of } R \}.
	\end{align*}	
These inequalities are equalities when $R$ is a regular ring: in such a case the finitistic dimension of a split quasi-hereditary coincides with the global dimension.
\end{thmintroduction}
In fact, a split quasi-hereditary algebra over a commutative Noetherian ring $R$ has finite global dimension if and only if $R$ is a regular ring with finite Krull dimension. As an application of the above result and Totaro's results on global dimension of Schur algebras \cite{zbMATH00966941}, we compute the global dimension and the finitistic dimension of integral Schur algebras (over commutative regular rings and over quotients of $\mathbb{Z}$) (see Proposition \ref{schuralgebrasglobaldimension}) extending Totaro's formula to cover more integral Schur algebras.

This paper is structured as follows:

In Subsection \ref{Notation and change of rings}, we collect some elementary results on Noetherian algebras. In Subsection \ref{Cellular_algebras}, we discuss the definition of cellular algebras and some immediate properties of these algebras. In Subsection \ref{Split highest weight category over noetherian rings}, we discuss the definition of split highest weight categories over commutative Noetherian rings introduced in \cite{Rouquier2008} and we give details why this concept coincides with the usual concept of  (split) highest weight category over finite-dimensional algebras over a field. We remark that the projectives in this definition for the integral setup are not necessarily indecomposable. In Subsection \ref{Split highest weight categories under change of rings}, we collect some properties about change of rings on split highest weight categories generalizing some methods to split highest weight categories that were available before only to specific cases of split highest weight categories. In particular, we obtain here that split highest weight categories behave quite well under change of rings contrary to non-split highest weight categories. In Subsection \ref{Uniqueness of standard modules with respect to the poset}, we study how the poset determines the standard modules in the integral setup. In Subsection \ref{Relation between heredity chains and standard modules}, we discuss the concept of split quasi-hereditary. This concept is equivalent to the concept of split highest weight category in the sense that an algebra is split quasi-hereditary if and only if its module category is a split highest weight category. Here we establish that the opposite algebra of a split quasi-hereditary algebra is again split quasi-hereditary and we discuss how some split heredity chains behave under change of rings.
In Subsection \ref{Split quasi-hereditary algebras and existence of projective covers},  we prove that split quasi-hereditary algebras are locally semi-perfect. In Subsection \ref{Global dimension of (split) quasi-hereditary algebras}, we reduce the computation of global dimension of split quasi-hereditary algebras over regular Noetherian commutative rings to computations over finite-dimensional algebras. In the cases that split quasi-hereditary algebras do not have finite global dimension, we give an upper bound for the finitistic dimension. In Section \ref{Cellular algebras}, we give answers to Questions (Q1) and (Q2) discussed above. Further, we discuss the concept of duality for Noetherian algebras whose regular module is finitely generated and free over the ground ring. In Section \ref{Finitistic dimension of some integral Schur algebras}, we show that Schur algebras are split quasi-hereditary over any commutative Noetherian ring (using only that they are quasi-hereditary over algebraically closed fields) and use their quasi-hereditary structure to compute the global dimension and the finitistic dimension of integral Schur algebras.

At the end of this paper, there are two appendices providing details that seem to be missing in the literature.
In Appendix \ref{Rsplit modules}, the reader can find detailed results about the objects called projective $R$-split $A$-modules which are foundations in the theory of split highest weight categories over commutative Noetherian rings.
In Appendix \ref{Filtrations in split highest weight categories}, we discuss the details behind the construction of filtrations into standard modules in the integral setup.

\section{Preliminaries}\label{Preeliminaries}

\subsection{Notation and change of rings}\label{Notation and change of rings}
Throughout this paper, we assume that $R$ is a Noetherian commutative ring with identity and $A$ is a projective Noetherian $R$-algebra, unless stated otherwise. By a projective Noetherian $R$-algebra $A$ we mean an $R$-algebra whose regular module is finitely generated projective as $R$-module. By $A\m$ (resp. $\rmod A$) we denote the category of all finitely generated left (resp. right) $A$-modules. By ${}_A A$ (resp. $A_A$) we denote the regular left (resp. right) $A$-module. Given $X\in A\m$ we denote by $\add X$ the additive closure of $X$ over $A$ and by $\id_X$ the identity map on $X$. We denote by $A\proj$ (resp. $\projr A$) the subcategory $\add {}_A A$ (resp. $\add A_A$). By a progenerator we mean a projective module whose additive closure coincides with the additive closure of the regular module $A$.
By $A^{op}$ we mean the opposite algebra of $A$. By $D$ we mean the standard duality functor $\Hom_R(-, R)\colon A\m\rightarrow A^{op}\m$. The projective dimension of an $A$-module $M$ is denoted by $\pdim_A M$ while $\gldim A$ and $\findim A$ denote the global dimension and the finitistic dimension of $A$, respectively.

We denote by $\Spec R$ (resp. $\MaxSpec R$) the set of all prime (resp. maximal) ideals of $R$. For every $\pri\in \Spec R$, and every $M\in A\m$ we shall write $M_\pri$ (resp. $R_\pri$) to denote the localization of $M$ (resp. $R$)  at the prime ideal $\pri$. In particular, $M_\pri\simeq R_\pri\otimes_R M$. Observe that, for every $\pri\in \Spec R$, the localization functor $A\m\rightarrow A_\pri\m$ is a dense functor (see for example Corollary 4.79 of \citep{Rotman2009a}).

Although, if $M_\mi\simeq N_\mi$ for every maximal ideal $\mi$ does not imply that $M\simeq N$ as $R$-modules we can compare them if they are both submodules of the same module.
\begin{Lemma}\label{equalitybeinglocal} Let $R$ be a commutative ring.
	Let $M, N, P$ be $R$-modules. Suppose that $N, M\subset P$. Then, $N=M$ if and only if $N_\mi=M_\mi$ for all maximal ideals $\mi$ of $R$.
\end{Lemma}
\begin{proof}
	Let $\pi\colon P\rightarrow P/M$ be the canonical surjection. For every maximal ideal $\mi$ in $R$,
$
		\pi(N)_\mi=\pi_\mi(N_\mi)=\pi_\mi(M_\mi)=0. 
$ Thus, $\pi(N)=0$. Therefore, $N\subset M$. Symmetrically, we deduce that $M\subset N$. Hence, $M=N$.
\end{proof}

We denote by $R(\pri)$ the residue field $R_\pri/\pri_\pri$ associated to the prime ideal $\pri$ of $R$. For maximal ideals $\mi$ of $R$, $R(\pri)$ is exactly $R/\mi$. Given $M\in A\m$, by $M(\pri)$ we mean the module $R(\pri)\otimes_R M$.

The projective modules can be determined using change of rings using the following result \citep[Lemma 3.3.2]{CLINE1990126}.
\begin{Theorem}\label{projectivitytermsmaximal}
	Let $R$ be a commutative Noetherian ring. Let $A$ be a projective Noetherian $R$-algebra. Let $M\in A\m$. Then $M$ is projective over $A$ if and only if $M\in R\proj$ and $M(\mi)$ is $A(\mi)$-projective for every maximal ideal $\mi$ in $R$.
\end{Theorem}
For Noetherian algebras, the Nakayama's Lemma plays an important role. For modules belonging to $A\m\cap R\proj$ the following version of the Nakayama's Lemma is useful:
\begin{Lemma}Let $R$ be a commutative ring.
	Let $N\in R\proj$. Let $\psi\in \Hom_R(M, N)$ such that $\psi(\mi)$ is an isomorphism for every maximal ideal $\mi$ in $R$. Then, $\psi$ is an isomorphism. \label{nakayamalemmasurjectiveproj}
\end{Lemma}
\begin{proof}
	By Nakayama's Lemma, $\psi$ is surjective. Since $N\in R\proj$ there exists an $R$-homomorphism $\delta\colon N\rightarrow M$ such that $\psi\circ \delta =\id_N$. In particular, $\delta$ is injective. Applying $R(\mi)\otimes_R -$, we get $\psi(\mi)\circ \delta(\mi)=\id_{N(\mi)}$. Thus, $\delta(\mi)$ is an isomorphism for every maximal ideal $\mi$ in $R$. By Nakayama's Lemma, $\delta$ is surjective. Thus, $\delta$ is an isomorphism, and $\psi$ is an isomorphism as well.  
\end{proof}
A local Noetherian commutative ring is called \textbf{regular} if it has finite global dimension. In such a case, the global dimension coincides with the Krull dimension of $R$ which coincides with the projective dimension of the residue field (see for example \citep[Theorem 8.62, Proposition 8.60]{Rotman2009a}). A Noetherian commutative ring is called \textbf{regular} if every localization $R_\pri$ is regular for every prime ideal $\pri$ of $R$.

As usual, in the integral setup $(A, R)$-exact sequences are better suited to work with techniques involving change of rings. By an \textbf{$(A, R)$-exact sequence} we mean an exact sequence of $A$-modules that splits when viewed as exact sequence of $R$-modules. An homomorphism $f\in \Hom_A(M, N)$ is called an \textbf{$(A, R)$-monomorphism} if the sequence $0\rightarrow M\xrightarrow{f} N$ is $(A, R)$-exact. Analogously, $(A, R)$-epimorphisms are defined. A module $M\in A\m$ is known as \textbf{$(A, R)$-projective module} if the functor $\Hom_A(M, -)$ sends $(A, R)$-epimorphisms to epimorphisms.

\subsection{Cellular algebras}\label{Cellular_algebras}

Graham and Lehrer introduced cellular algebras over commutative rings. However, in applications, cellular algebras are considered over a field.  Some of the properties we are interested in can be found in \citep{zbMATH00871761}, \citep{zbMATH01218863}, \citep{zbMATH01463504}, \citep{zbMATH01384521}, \citep{zbMATH01475237}. 

We will start by collecting some properties of cellular algebras which remain valid over commutative Noetherian rings.

Explicitly, the common definition of cellular algebras used for practical purposes is the following:

\begin{Def}\label{cellularbasis}
	Let $R$ be a commutative Noetherian ring. Let $A$ be a free Noetherian $R$-algebra, that is, $A$ is free as $R$-module. $A$ is called \textbf{cellular} with cell datum $(\L, M, C, \iota)$ if the following holds:
	\begin{enumerate}[(C1)]
		\item The finite set $\L$ is partially ordered. Associated with each $\l\in \L$ there is a finite set $M(\l)$. The algebra $A$ has an $R$-basis \begin{align}
			\{ C_{S, T}^\l \ | \ S, T\in M(\l), \ \l \in \L \}.
		\end{align}
		\item The map $\iota\colon A\rightarrow A$ is an $R$-linear anti-isomorphism with $\iota^2=\id_A$ which sends $C_{S, T}^\l$ to $C_{T, S}^\l$, $S, T\in M(\l)$, $\l\in \L$.
		\item For each $\l\in \L$ and $S, T\in M(\l)$ and each $a\in A$ we can write
		\begin{align}
			aC_{S, T}^\l = \sum_{U\in M(\l)} r_a(U, S)C_{U, T}^\l + r',
		\end{align}where $r'$ is a linear combination of basis elements with upper index $\mu$ strictly smaller than $\l$, and where the coefficients $r_a(U, S)\in R$ do not depend on $T$.
	\end{enumerate}
\end{Def}

\begin{Lemma} Consider the following condition.
	\begin{enumerate}[(C3')]
		\item For each $\l\in \L$ and $S, T\in M(\l)$ and each $a\in A$ we can write
		\begin{align}
			C_{S, T}^\l a = \sum_{U\in M(\l)} r_a(U, T)C_{S, U}^\l + r',
		\end{align}where $r'$ is a linear combination of basis elements with upper index $\mu$ strictly smaller than $\l$, and where the coefficients $r_a(U, T)\in R$ do not depend on $S$.
	\end{enumerate}
	Under conditions (C2) and (C1), condition (C3) is equivalent to (C3').
\end{Lemma}
\begin{proof} Assume that (C3) holds.
	We can write, for $a=i(x)\in A$, $x\in A$,
	\begin{align}
		C_{S, T}^\l a \overset{(C2)}{=} \iota(C_{T, S}^\l) \iota(x) & = \iota(x C_{T, S}^\l) \overset{(C3)}{=} \iota\left( \sum_{U\in M(\l)}r_x(U, T)C_{U, S}^\l  +r'\right) = \sum_{U\in M(\l)} r_x(U, T) \iota(C_{U, S}^\l) +\iota(r')\\
		&=\sum_{U\in M(\l)} r_x(U, T) C_{S, U}^\l +\iota(r').
	\end{align}Since $\iota$ only changes the lower indexes the upper indexes of the basis elements in the linear combination of $\iota(r')$ are strictly smaller than $\l$. Putting $r_a(U, T)$ equal to $r_{\iota(a)}(U, T)$, condition (C3') follows. The converse implication is analogous.
\end{proof}
The map $\iota\colon A\rightarrow A$ is called an \textbf{involution} of $A$.

\begin{Cor}
	$A$ is a cellular $R$-algebra with cell datum $(\L, M, C, \iota)$ if and only if the opposite algebra $A^{op}$ is a cellular $R$-algebra with cell datum $(\L, M, \iota(C), \iota)$.
\end{Cor}

Iwahori-Hecke algebras are a classical example of cellular algebras (see \citep[Example 1.2]{zbMATH00871761}). The cell basis is the Kazhdan-Lusztig basis. In fact,  the axioms of cellular basis presented in Definition \ref{cellularbasis} are based on the Kazhdan-Lusztig basis of Hecke algebras.

There is a more abstract definition of cellular algebras due to Koenig and Xi \citep{zbMATH01218863} which illustrates better its structural properties. 

\begin{Def}\label{abstractcellular}
	Let $R$ be a commutative Noetherian ring and let $A$ be a projective Noetherian $R$-algebra. Assume that there is an $R$-linear anti-isomorphism $\iota$ on $A$ with $\iota^2=\id_A$.
	\begin{enumerate}[(i)]
		\item A two-sided ideal $J$ of $A$ is called a \textbf{cell ideal} (with respect to $\iota$) if
		\begin{enumerate}[(a)]
			\item $\iota(J)=J$;
			\item There exists a left ideal $\theta\in A\m$, free as $R$-module, such that $\theta\subset J$;
			\item There is an isomorphism of $A$-bimodules $\alpha\colon J \rightarrow \theta \otimes_R \iota(\theta)$ making the following diagram commutative:
			\begin{center}\vspace{-\abovedisplayskip}
				\begin{tikzcd}
					J \arrow[r, "\alpha"] \arrow[d, "\iota"] & \theta\otimes_R \iota(\theta) \arrow[d, "x\otimes y \mapsto \iota(y)\otimes \iota(x)"] \\
					J \arrow[r, "\alpha", swap] & \theta\otimes_R \iota(\theta) \\
				\end{tikzcd}.
			\end{center}\vspace{-1.5\belowdisplayskip}
		\end{enumerate}
		\item The algebra $A$ (with involution $\iota$) is called \textbf{cellular} if 
		\begin{enumerate}[(a)]
			\item There is an $R$-module decomposition $A=J_1'\bigoplus \cdots\bigoplus J_n'$ (for some $n$) with $\iota(J_j')=J_j'$ for each $j$;
			\item Setting $J_j=\bigoplus_{l=1}^j J_l'$ gives a chain of two-sided ideals of $A$, called \textbf{cell chain}: \linebreak${0\subset J_1\subset\cdots\subset J_n=A}$ (each of them fixed by $\iota$);
			\item For each $j$ ($j=1, \ldots, n$) the quotient $J_j'=J_j/J_{j-1}$ is a cell ideal (with respect to the involution induced by $\iota$ on the quotient) of $A/J_{j-1}$.
		\end{enumerate}
	\end{enumerate} 
\end{Def}
In particular, this definition requires that every cell ideal is a free $R$-module. The modules $\theta(j)$ associated with each cell ideal $J_j'$ are called \textbf{cell modules.}

We note that the original definition in \citep{zbMATH01218863} requires $R$ to be an integral domain but the arguments easily pass to the general case. For sake of completeness, we will write the proof of equivalence of both notions.

\begin{Prop}\label{cellulardeficoincide}
	The two above definitions of cellular algebras are equivalent.
\end{Prop}
\begin{proof}
The arguments used on pages 8-9 of \cite{zbMATH01218863} do not require the ground ring to be integral. 
\end{proof}

From the proof of Proposition \ref{cellulardeficoincide}, we can deduce the following result.

\begin{Cor}Let $A$ be a cellular $R$-algebra with cell datum $(\L, M, C, \iota)$.\label{cellmodules}
	Let $A(<\l)$, $\l\in \L$, be the free $R$-module with $R$-basis 
	\begin{align}
		\{ C_{S, T}^\mu \colon \mu<\l, \ S, T\in M(\mu) \}.
	\end{align}The (left) cell modules are the $A$-modules which are free over $R$ with $R$-basis
	\begin{align}
		\theta_l(\l)=\{C_{S, T_0}^\l + A(<\l) \colon S\in M(\l)  \}, \quad \text{for some} \quad T_0\in M(\l),  \quad \l\in \L.
	\end{align}
	The (right) cell modules are the right $A$-modules which are free over $R$ with basis
	\begin{align}
		\theta_r(\l)=\{C_{S_0, T}^\l + A(<\l) \colon T\in M(\l)  \}, \quad \text{for some} \quad S_0\in M(\l), \quad \l\in \L.
	\end{align}
\end{Cor}
The statement for right modules follows using condition (C3') instead of (C3).

\begin{Prop}Let $A$ be a cellular $R$-algebra with cell datum $(\L, M, C, \iota)$. \label{twistedmodules}
	Let $M\in A\m$. Then, $M^{\iota}$ denotes a right $A$-module, known as the \textbf{twisted module} of $M$,  by setting $x\cdot_{\iota} a = \iota(a)x$. Similarly, for any $N\in \rmod A$, ${}^{\iota}$ denotes a left $A$-module by setting $a\cdot_{\iota} x= x\iota(a)$.  Moreover, 
	\begin{enumerate}[(i)]
		\item $\theta_l(\l)^\iota\simeq \theta_r(\l)$ as right $A$-modules, $\l\in \L$;
		\item ${}^{\iota}\theta_r(\l)\simeq \theta_l(\l)$ as left $A$-modules, $\l\in \L$.
	\end{enumerate}
\end{Prop}
\begin{proof}
	Consider the map $\psi\colon \theta_l(\l)^{\iota}\rightarrow \theta_r(\l)$ that sends $C_{S, T_0}^\l+A(<\l)$ to $C_{S_0, S}^\l + A(<\l)$. Thus, $\psi$ is bijective. We want to show that $\psi$ is an $A$-isomorphism. To obtain that we can observe that
	\begin{align}
		\psi( (C_{S, T_0}^\l + A(<\l))\cdot_{\iota} \iota(a)   )&=\psi(\iota^2(a)\cdot (C_{S, T_0}^\l + A(<\l))  )= \psi(a\cdot (C_{S, T_0}^\l +A(<\l) )  ) \\&= \psi \left( \sum_{U\in M(\l)} r_a(U, S)C_{U, T_0}^\l + A(<\l)\right) = \sum_{U\in M(\l)} r_a(U, S)\psi( C_{U, T_0}^\l + A(<\l) ) \nonumber\\&= \sum_{U\in M(\l)} r_a(U, S)( C_{S_0, U}^\l + A(<\l)  ).
	\end{align}
	On the other hand,
	\begin{align}
		\psi (C_{S, T_0}^\l + A(<\l))\cdot \iota (a) &= (C_{S_0, S}^\l + A(<\l)  )\iota(a)  = \iota (C_{S, S_0}^\l + A(<\l) )\iota(a)= \iota ( a\cdot C_{S, S_0}^\l + A(<\l) )\\
		&= \iota( \sum_{U\in M(\l)} r_a(U, S)(C_{U, S_0}^\l + A(<\l)) ) = \sum_{U\in M(\l)} r_a(U, S)(C_{S_0, U}^\l + A(<\l)).\nonumber
	\end{align}
	Therefore, $\psi$ is a right $A$-isomorphism. Using the map ${}^\iota \theta_r(\l)\rightarrow \theta_l(\l)$, mapping $C_{S_0, T}^\l+A(<\l)$ to \linebreak $C_{T, T_0}^\l + A(<\l)$, (ii) follows.
\end{proof}

We can define a duality functor ${}^\natural(-)\colon A\m\rightarrow A\m$ which sends $M$ to $D M^{\iota}$ and a duality functor $(-)^\natural\colon \rmod A\rightarrow \rmod A$ which sends $N$ to $D {}^\iota N$. The following corollary is an immediate consequence of Proposition \ref{twistedmodules}.

\begin{Cor}
	Let $A$ be a cellular $R$-algebra with cell datum $(\L, M, C, \iota)$.  Let $\l\in \L$. Then,\begin{enumerate}[(i)]
		\item ${}^\natural \theta_l(\l)\simeq D\theta_r(\l)$ as left $A$-modules;
		\item $\theta_r(\l)^\natural\simeq D\theta_l(\l)$ as right $A$-modules.
	\end{enumerate}
\end{Cor}

\begin{Prop}
	Let $A$ be a cellular $R$-algebra with cell datum $(\L, M, C, \iota)$. Then, $A\in \mathcal{F}(\theta_{\l\in \L})$.
\end{Prop}
\begin{proof}We can consider an increasing bijection between the posets $\L$ and $\{1, \ldots, n\}$.
	We want to show that there exists a filtration 
	\begin{align}
		0=P_0\subset P_1\subset \cdots \subset P_n=A
	\end{align}with $P_i/P_{i-1}\simeq \theta_i\otimes_R U_i$ for some free $R$-module $U_i$ where the cell module $\theta_i$ is associated with the cell ideal $J_i'$. We shall proceed by induction on $n$. Assume $n=1$. Then, $A$ is a cell ideal of $A$. Thus, there exists $\theta_1\subset A$ such that $A\simeq \theta_1\otimes_R \iota(\theta_1)$ and $\iota(\theta_1)$ is $R$-free. So, $A\in \mathcal{F}(\theta_1)$. Assume now that the result holds for $n-1$. The modules $\theta_j$, $j>1$, are cell modules of $A/J_1$. By induction, $A/J\in \mathcal{F}(\theta_{j>1})$. So, there exists a filtration 
	$
	0=P_1'\subset \cdots\subset P_n'=A/J, \quad P_i'/P_{i-1}'=\theta_i\otimes_R U_i,
	$ where $U_i$ is a free $R$-module. Thus, there exists a chain 
	\begin{align}
		J=P_1\subset \cdots\subset P_n=A, \quad P_i/P_{i-1}\simeq P_i/J/P_{i-1}/J\simeq P_i'/P_{i-1}'.
	\end{align}Since $J$ is a cell ideal, $J\simeq \theta_1\otimes_R \iota(\theta_1)$. Putting $U_1=\iota(\theta_1)$, the result follows.
\end{proof}

Cellular algebras have a base change property (see \citep[(1.8)]{zbMATH00871761}). 

\begin{Prop}\label{changeringcellular}
	Let $S$ be a commutative $R$-algebra which is a Noetherian ring. Let $A$ be a cellular $R$-algebra with cell datum $(\L, M, C, \iota)$ then $S\otimes_R A$ is cellular with cell datum $(\L, M, 1_S\otimes_R C, \id_S\otimes_R \iota)$.
\end{Prop}
\begin{proof}
	The algebra $S\otimes_R A$ has an $S$-basis $\{S\otimes_R C_{U, T}^\l\ | \ U, T\in M(\l), \l\in \L \}$. Hence, condition (C1) holds. Since $\iota$ is an anti-isomorphism over $R$, so it is $S\otimes_R \iota$ over $S$. Moreover, $(\id_S\otimes_R \iota)^2=\id_S\otimes_R \iota^2=\id_S\otimes_R \id_A=\id_{S\otimes_R A}$ and $\id_S\otimes_R \iota (1_S\otimes_R C_{U, T}^\l)=1_S\otimes_R \iota(C_{U, T}^\l)=1_S\otimes_R C_{T, U}^\l$ for $U, T\in M(\l)$, $\l\in \L$. So, condition (C2) holds. It remains to check condition (C3). For $s\otimes_R a\in S\otimes_R A$,
	\begin{align}
		(s\otimes_R a) (1_S\otimes C_{V, T}^\l)&=s\otimes aC_{V, T}^\l=s\otimes \left( \sum_{U\in M(\l)} r_a(U, V)C_{U, T}^\l + r' \right) = \sum_{U\in M(\l)} s\otimes r_a(U, V)C_{U, T}^\l + s\otimes r' \\
		&= \sum_{U\in M(\l)} sr_a(U, V) (1_S\otimes_R C_{U, T}^\l) +s\otimes r',
	\end{align}where $s\otimes r'$ is a linear combination of basis elements $1_S\otimes_R C_{l, t}^\mu$ with upper index $\mu$ strictly smaller than $\l$ and $\l\in \L$, $l, t\in M(\mu)$,  $V, T\in M(\l)$.
\end{proof}

The following result due to \citep[Proposition 4.3]{zbMATH01218863} is fundamental to understand under what conditions an endomorphism algebra of a projective module over a cellular algebra remains cellular. However, we need further assumptions on the ground ring. By a commutative \textbf{projective-free} ring $R$ we mean a commutative ring $R$ with every finitely generated projective $R$-module being free. Properties about these rings can be found in \citep{zbMATH02166983}.
\begin{Prop}\label{endomorphismcellular}
	Let $R$ be a commutative Noetherian projective-free ring. Let $A$ be a cellular $R$-algebra with involution $\iota$ and with cell chain
	\begin{align}
		0\subset J_1\subset\cdots\subset J_n=A.
	\end{align} Let $e$ be an idempotent of $A$ which is fixed by $\iota$. Then, $eAe$ is a cellular $R$-algebra with involution $\iota_{|_{eAe}}$ and with cell chain 
	\begin{align}
		0\subset eJ_1e \subset\cdots\subset eJ_ne=eAe.
	\end{align} 
\end{Prop}
\begin{proof}
	Since $\iota$ fixes the idempotent $e$, the restriction of $\iota\colon A\rightarrow A$ to $eAe$ has image in $eAe$. Thus, $\iota_{|_{eAe}}$ is an involution of $eAe$. Let $J$ be a cell ideal of $A$. We claim that $eJe$ is a cell ideal of $eAe$. Let $j\in J$. By assumption, there exists $j'$ such that $\iota(j')=j$. Hence, $\iota(ej'e)=\iota(e)\iota(ej')=\iota(e)\iota(j')\iota(e)=eje$. This shows that $\iota_{|_{eAe}}(eJe)=eJe$. Let $\theta$ be the left ideal associated with $J$. Then, $e\theta=eA\otimes_A \theta\in \add_R \theta$. Hence, $e\theta\in R\proj$. Since $R$ is projective-free $e\theta$ is $R$-free and $\iota(e\theta)=\iota(\theta)e$. Applying the functors $eA\otimes_A -$ and $-\otimes_A Ae$ to $\alpha$ we obtain an isomorphism $e\alpha e$ compatible with the desired commutative diagram. So, $eJe$ is a cell ideal. Proceeding by induction, multiplication by $e$ on both sides on a cell chain of $A$ yields a cell chain for $eAe$.
\end{proof}
Of course, $\mathbb{Z}$ is a principal ideal domain, and thus it is a projective-free ring. Due to \citep{zbMATH03627311}, the Laurent polynomial ring $\mathbb{Z}[X, X^{-1}]$ is projective-free. These observations are important to give proofs of Hecke algebras being cellular using the cellularity of $q$-Schur algebras using Proposition \ref{endomorphismcellular}.

In \citep{zbMATH01463504}, it is shown that in characteristic two not every projective module can be given by an idempotent fixed by the involution. 
Hence, being cellular is not a categorical concept. The situation becomes even worse for cellular algebras over commutative rings which are not projective-free. Still in \citep{zbMATH01463504}, they show that cellular algebras over fields of characteristic different from two are preserved under Morita equivalence.
This is another evidence that cellular algebras have nicer properties over $\mathbb{Z}[\frac{1}{2}]$ and over Laurent polynomial rings over $\mathbb{Z}[\frac{1}{2}]$.

\subsection{Split highest weight categories over commutative Noetherian rings} \label{Split highest weight category over noetherian rings} 

Quasi-hereditary algebras, introduced in \citep{MR961165}, play an important role in the representation theory of Lie algebras and algebraic groups. During the last 30 years, many approaches have been suggested \cite{zbMATH01198520, Rouquier2008, CLINE1990126, zbMATH02037023} to this subject. 
In this section, we will follow Rouquier's approach.  For the study of quasi-hereditary algebras over fields, we refer to \cite{MR961165}, \cite{PS88}, \cite{Dlab1989d}, \cite{MR1128706}, \citep[A]{MR1284468},  \cite{zbMATH00140218}, \cite{Dlab1989}.

\begin{Def}\label{splithwc}
	Let $R$ be a Noetherian commutative ring. Let $A$ be a projective Noetherian $R$-algebra. Let $\Lambda$ be a finite preordered set. We say that $(A\m, \Lambda)$ is a \textbf{split highest weight category} if there exist finitely generated modules $\{\Delta(\lambda)\colon\lambda\in \Lambda \}$ such that \begin{enumerate}[(i)]
		\item $\Delta(\lambda)$ is a projective $R$-module;
		\item If $\Hom_A(\Delta(\lambda'), \Delta(\lambda''))\neq 0$, then $\lambda'\leq \lambda''$.
		\item If $N\in A\m$ is such that $\Hom_A(\Delta(\lambda), N)=0$ for all $\lambda\in \Lambda$, then $N=0$.
		\item For each $\lambda\in \Lambda$, there exists a projective $A$-module $P(\lambda)$ such that there is an exact sequence
		\begin{align*}
			0\rightarrow C(\lambda)\rightarrow P(\lambda)\xrightarrow{\pi_\l} \Delta(\lambda)\rightarrow 0,
		\end{align*} where $C(\lambda)$ has a finite filtration by modules of the form $\St(\mu)\otimes_R U_\mu$ with $U_\mu\in R\proj$ and $\mu>\l$. 
		\item $\End_A(\Delta(\lambda))\simeq R$ for all $\lambda\in \Lambda$.
	\end{enumerate}
\end{Def}
We say also that $(A\m, \{\St(\l)_{\l\in \L}  \})$ is a split highest weight category. For simplicity, occasionally we just write $\St$ to mean the set $\{\Delta(\lambda)\colon\lambda\in \Lambda \}$.

Here, Condition (iii) in  \ref{splithwc} means that the direct sum of all projective modules $P(\lambda)$ constructed in Condition \ref{splithwc}(iv) is a progenerator of $A\m$. In fact, under Conditions (i), (ii), (iv), (v) the Condition (iii) is equivalent to $\bigoplus_{\l\in \St} P(\l)$ being a progenerator of $A\m$.

Before we proceed any further we shall see that when $R$ is a field and $R$ is a splitting field for $A$ then split highest weight category here defined is the classical notion of highest weight category.
First, we need the following technical lemma.
\begin{Lemma}\label{lemmamultiplicityofdelta}
	Let $(A\m, \{\Delta(\lambda)_{\lambda\in \Lambda}\})$ be a split highest weight category over $R$.  Then, the following assertions hold.
	\begin{enumerate}[(i)]
		\item For any $\lambda\in \Lambda$, $\Hom_A(P(\lambda), \Delta(\lambda))\simeq \End_A(\Delta(\lambda))\simeq R$. \label{multiplicityofdelta}
		\item If $P(\mu)$ is an $A$-summand of $P(\lambda)$, then $\lambda\leq \mu$. \label{ordering on summands of projectives}
		\item If $\Hom_A(P(\mu), \Delta(\lambda))\neq 0$, then $\mu\leq \lambda$. \label{multiplicity of delta part 2}
		\item Assume that $R$ is a field. If $\Hom_A(P(\mu), \rad \Delta(\lambda))\neq 0$, then $\mu<\lambda$.\label{radical multiplicity} 
		\item Assume that $R$ is a field. For each $\l\in \L$, $\top \Delta(\lambda)$ is simple. Moreover, $\St(\lambda)$ is indecomposable. \label{field case projective indc}
	\end{enumerate}
\end{Lemma}
\begin{proof}
	Consider $h\in \End_A(\Delta(\lambda))$. By \ref{splithwc}(iv), we have a surjective map $P(\lambda)\xrightarrow{\pi_\lambda} \St(\l)$, hence we have a map $P(\lambda)\xrightarrow{\pi_\lambda}\Delta(\lambda)\xrightarrow{h}\Delta(\lambda)\in \Hom_A(P(\lambda), \Delta(\lambda))$. Take a map $g\in \Hom_A(P(\lambda), \Delta(\lambda))$.
	
	Consider the following commutative diagram,
	\begin{center}
		\begin{tikzcd}
			0\arrow[r]& C(\lambda)\arrow[r, "i"]\arrow[dr, "g\circ i", swap]& P(\lambda)\arrow[r, "\pi_\lambda"]\arrow[d, "g"]& \Delta(\lambda)\arrow[r]& 0\\
			& &  \Delta(\lambda) & &
		\end{tikzcd}.
	\end{center} 
	Assume $g\circ i\neq 0$. Since $C(\lambda)$ has a finite filtration into modules of the form $\St\otimes_R X$, $X\in R\proj$, there exists $\Delta(\mu)$ with $\mu>\lambda$ such that $\Hom_A(\Delta(\mu), \Delta(\lambda))\neq 0$. By \ref{splithwc}(ii) we get that $\mu\leq \lambda$, which contradicts $\mu>\lambda$. Hence, $g\circ i=0$. So, $g$ induces uniquely a map $g'\in \End_A(\Delta(\lambda))$ such that $g'\circ \pi_\lambda=g$. Notice that $(h\circ \pi_\lambda)'$ satisfies $(h\circ \pi_\lambda)'\circ \pi_\lambda=h\circ \pi_\lambda$. Since $\pi_\lambda$ is an epimorphism, we get $(h\circ \pi_\lambda)'=h$. It follows that $\Hom_A(P(\lambda), \Delta(\lambda))\simeq \End_A(\Delta(\lambda))$ as $R$-modules. So, (i) follows.
	
	Assume that $P(\mu)$ is an $A$-summand of $P(\lambda)$.	Consider the following commutative diagram
	\begin{center}
		\begin{tikzcd}
			0\arrow[r]& C(\lambda)\arrow[r]\arrow[ddr]& P(\lambda)\arrow[r, "\pi_\lambda"]\arrow[d, twoheadrightarrow]& \Delta(\lambda)\arrow[r]& 0\\
			& &  P(\mu) \arrow[d, "\pi_\mu", twoheadrightarrow] & &\\
			& & \Delta(\mu) & &
		\end{tikzcd}.
	\end{center} If the map $C(\lambda)\rightarrow \Delta(\mu)$ is non-zero, then exists some module $\Delta(l)$ factor of $C(\lambda)$, hence $l>\lambda$ such that $\Hom_A(\Delta(l), \Delta(\mu))\neq 0$. By \ref{splithwc}(ii), we get $l\leq \mu$, which implies $\lambda<l\leq \mu$.
	
	If the map $C(\lambda)\rightarrow \Delta(\mu)$ is zero, then there exists a non-zero $A$-homomorphism $h\colon\Delta(\lambda)\rightarrow \Delta(\mu)$ which makes the diagram commutative. By \ref{splithwc}(ii), $\lambda\leq \mu$. Hence, (ii) follows.
	
		Let $0\neq \phi\in \Hom_A(P(\mu), \Delta(\lambda))$. Denote by $i$ the inclusion $C(\mu)\hookrightarrow P(\mu)$. If $\phi\circ i=0$, then $\phi$ induces a non-zero map in $\Hom_A(\Delta(\mu), \Delta(\lambda))$. By \ref{splithwc}(ii), $\mu\leq \lambda$. If $\phi\circ i\neq 0$ then exists by \ref{splithwc}(iv) $l>\mu$ such that $\Hom_A(\Delta(l), \Delta(\lambda))\neq 0$. By \ref{splithwc}(ii), $l\leq \lambda$. So, (iii) follows.
		
		Assume that $R$ is a field. Let $h\in \Hom_A(P(\mu), \rad \Delta(\lambda))$. Denote by $i$ the inclusion $\rad \Delta(\lambda)\hookrightarrow \Delta(\lambda)$.
		Applying \ref{multiplicity of delta part 2} with $i\circ h\neq 0$ it follows that $\mu\leq \lambda$. Now consider additionally that $\dim_R \Hom_A(P(\lambda), \Delta(\lambda))=1$. If $\mu=\lambda$, then $i\circ h=\alpha \pi_{\lambda}$ for some $\alpha\in R$. Thus, $i\circ h$ is surjective. Consequently, $i$ is an isomorphism. By Nakayama's Lemma, we get a contradiction. Hence, (iv) follows.
		
	It remains to prove (v).	If $\top P(\l)$ is simple, then there is nothing to prove. 	Assume that $\top P(\lambda)$ is not simple. Choose $S$ a simple module summand of $\top \St(\lambda)$ which is a summand of $\top P(\lambda)$. Denote  by $P$ the projective cover of $S$. Hence, $P$ is an indecomposable summand of $P(\lambda)$. And so, the canonical map $P\rightarrow \top P(\lambda)$ factoring through $S$ is non-zero. 
		We have a commutative diagram
		\begin{center}
			\begin{tikzcd}
				S=\top P\arrow[r, hookrightarrow] &\top P(\lambda)\arrow[r] & \top \Delta(\lambda)\\
				P\arrow[u]\arrow[rr, "\exists f\neq 0"]& & \Delta(\lambda) \arrow[u, twoheadrightarrow]
			\end{tikzcd}
		\end{center}
		Note that the existence of such non-zero map $f$ is due to $P$ being projective and the upper row being a monomorphism.
		Since $\dim_R \Hom_A(P(\lambda), \Delta(\lambda))=1$, there exists $\alpha\in R$ such that $f\circ \pi=\alpha \pi_\lambda$, where $\pi$ denotes the projection $P(\lambda)\twoheadrightarrow P$ and $i$ denotes the inclusion $P\hookrightarrow P(\lambda)$. If $\alpha=0$, then $f=f\circ \pi \circ i$ would be zero. Since $R$ is a field, $\alpha \id_{\Delta(\lambda)}$ is an isomorphism. Hence, $f$ is surjective. By the commutativity of the diagram, the map $S\hookrightarrow \top P(\lambda)\rightarrow \top \Delta(\lambda)$ is  surjective. Since $S$ is simple, it is an isomorphism. Therefore, $P$ is the projective cover of $\St(\l)$ and $\St(\l)$ is indecomposable.
\end{proof}

\begin{Remark} Let $R$ be a field.
	Condition \ref{splithwc}(iii) ensures that each simple module appears as a $\top$ of a standard module. Denote by $\rad A$ the Jacobson radical of $A$ and $S$ a simple $A$-module. Since $S$ is simple, either $\rad A S=0$ or $\rad AS=S$. By Nakayama's Lemma, if $\rad AS=S$, then $S=0$. Thus, $\rad AS=0$ and hence $\top S=S$. If $\Hom_A(\Delta(\lambda), S)\neq 0$, then $\top \Delta(\lambda)\rightarrow \top S=S$ is surjective. In other words, $S$ would appear as a summand of $\top \Delta(\lambda)$. Therefore, if $S$ never occurs as a summand of $\top \Delta(\lambda)$ for some $\lambda$, it would follow that $\Hom_A(\Delta(\lambda), S)=0$ for every $\lambda$. So, \ref{splithwc}(iii) would have implied $S=0$. \label{simples in top of standards field case}
\end{Remark}

We see immediately that even for split quasi-hereditary algebras over fields the conditions in Definition \ref{splithwc} are not enough to enforce the projectives $P(\l)$ to be indecomposable.
For example, we can consider a semi-simple algebra with two simple modules say $S_1$ and $S_2$ over an algebraically closed field. Fixing $P(1)=S_1\oplus S_2$ and $P(2)=S_2$ with $\St(1)=S_1$ together with the usual order we see that all conditions  of Definition \ref{splithwc} are satisfied. However, $P(1)$ is not indecomposable.
As we will see next, for split quasi-hereditary algebras over fields we can replace the projectives $P(\l)$ with the projective covers of the standard modules.

\begin{Prop}\label{splithwcfieldcase} (See also \citep[Lemma 4.31]{zbMATH05707949}).
	Let $R$ be a splitting field for $A$. Then, $(A\m, \{\Delta(\lambda)_{\lambda\in \Lambda}\})$ is a split highest weight category according to Definition \ref{splithwc} if and only if there is a correspondence between the poset $\Lambda$ and the isomorphism classes of simple $A$-modules which we denote by $S(\lambda)=\top \St(\l)$, and for all $\lambda\in\Lambda$, $\Delta(\lambda)$ satisfies
	\begin{enumerate}[label=(A\arabic*)]
		\item There is an exact sequence $0\rightarrow K(\lambda)\rightarrow \Delta(\lambda)\rightarrow S(\lambda)\rightarrow 0$ and the composition factors of $K(\lambda)$, $S(\mu)$, satisfy $\mu<\lambda$.
		\item There is an exact sequence $0\rightarrow C(\lambda)\rightarrow P_{c}(\lambda)\rightarrow \Delta(\lambda)\rightarrow 0$ and $C(\lambda)$ is filtered by modules $\Delta(\mu)$ with $\mu>\lambda$, where $P_c(\l)$ denotes the projective cover of $\St(\l)$.
	\end{enumerate}
\end{Prop}
\begin{proof}
	Let $(A\m, \{\Delta(\lambda)_{\lambda\in \Lambda}\})$ be a highest weight category in the classical sense, that is, satisfying (A1) and (A2). Let $\Lambda$ be the set of isomorphism classes of simple $A$-modules. \ref{splithwc}(i) is trivially checked since $R$ is a field. By (A1) and Lemma \ref{lemmamultiplicityofdelta}\ref{multiplicityofdelta}, Condition \ref{splithwc}(v) holds. Condition \ref{splithwc}(ii) is also satisfied since every non-zero map between standard modules $\St(\mu)$ and $\St(\l)$ can be extended to a non-zero map between the projective cover of $\St(\mu)$ and $\St(\l)$. Consequently, such a case would lead to the multiplicity of $S(\mu)$ in $\St(\l)$ being positive. By (A1), this occurs only if $\mu\leq \l$. Define $P_c(\lambda)$ to be the projective cover of $S(\lambda)$. By Axiom (A2) of highest weight categories, it follows that \ref{splithwc}(iv) is satisfied. Assume that $N\in A\m$ such that $\Hom_A(\Delta(\lambda), N)=0$ for all $\lambda\in \Lambda$. If $N\neq 0$, then $\soc N\neq 0$. Let $S(\lambda)\subset \soc N$. By Axiom (A1), there exists an exact sequence 
	\begin{center}
		\begin{tikzcd}
			0\arrow[r]& K(\lambda) \arrow[r] &\Delta(\lambda)\arrow[r]\arrow[rd, "\neq 0"]&S(\lambda)\arrow[r]\arrow[d, hookrightarrow]&0\\
			& & & N &
		\end{tikzcd}
	\end{center}
	This contradicts our assumption that $\Hom_A(\Delta(\lambda), N)=0$. So, $N=0$ and \ref{splithwc}(iii) holds. So, $(A, \{\Delta(\lambda)_{\lambda\in \Lambda}\})$ is a split highest weight category.
	
	Conversely, assume that $(A\m, \{\Delta(\lambda)_{\lambda\in \Lambda}\})$ is a split highest weight category. Since $$\dim_R \Hom_A(P(\lambda), \Delta(\lambda))=\dim_R \End_A(\Delta(\lambda))=1$$ all standard modules have a simple top. It can be seen that the category of objects admitting a filtration by standard modules is closed under direct summands for example by using trace filtrations (see  \citep[A.2]{MR1284468}). Hence, $P_c(\l)$ satisfies (A2). Alternatively, we can also apply Proposition \ref{everyprojectivehasafiltration} to see that $P_c(\l)$ satisfies (A2). 
	By Remark \ref{simples in top of standards field case}, $|\Lambda|$ is greater than or equal to the number of classes of non-isomorphic simple $A$-modules. Assume that there exist $\lambda$ and $\mu$ such that $\Delta(\lambda)$ and $\Delta(\mu)$ have the same projective cover. By Lemma \ref{lemmamultiplicityofdelta}\ref{ordering on summands of projectives} now using $P_c$ instead of $P$ we deduce that $\lambda=\mu$. Hence, $|\Lambda|$ is equal to the number of non-isomorphic simple $A$-modules. Now $[\Delta(\lambda)\colon S(\lambda)]=\dim_R \Hom_A(P_c(\lambda), \Delta(\lambda))=1$. By Lemma \ref{lemmamultiplicityofdelta}\ref{radical multiplicity}, if $[\rad \Delta(\lambda)\colon S(\mu)]\neq 0$, then $\mu< \lambda$. So, Axiom (A1) holds.
\end{proof}

The standard modules with maximal index can be generalised into the concept of projective $R$-split module. 
An $A$-module $L\in A\proj$ which is faithful over $R$ is called \textbf{projective $R$-split $A$-module} if the canonical morphism 
\begin{align}
	\tau_{L, P}\colon L\otimes_R \Hom_A(L, P)\rightarrow P, \quad l\otimes f\mapsto f(l)
\end{align} is an $(A, R)$-monomorphism for all projective $A$-modules $P$. We denote by $\tau_{L}$ the homomorphism $\tau_{L, A}$.
By $\mathcal{M}(A)$ we mean the set of isomorphism classes of projective $R$-split $A$-modules.  These modules were introduced in \cite{Rouquier2008}. In Appendix \ref{Rsplit modules} we give a detailed exposition on these modules. We highlight the following result which allows us to stratify split highest weight categories.  

\begin{Lemma}(see \citep[Lemma 4.12]{Rouquier2008})
	Let $A$ be a projective Noetherian $R$-algebra. Let ${\{\St(\l)\colon \l\in\L \}}$ be a finite set of modules in $A\m$ together with a poset structure on $\L$. Let $\alpha$ be a maximal element in $\L$. Then, $(A\m, \{\Delta(\lambda)_{\lambda\in \Lambda}\})$ is a split highest weight category if and only if $\St(\alpha)\in \mathcal{M}(A)$ and $(A/J\m, \{\Delta(\lambda)_{\lambda\in \Lambda\backslash \{\alpha \} } \})$  is a split highest weight category, where $J=\im \tau_{\St(\alpha)}$ is an ideal of $A$.
\end{Lemma}
\begin{proof}
	See for example Lemma \ref{splithwcinduction}.
\end{proof}

\begin{Notation}\label{filtrationnotation}
	Denote by $\mathcal{F}(\St)$ the full subcategory of $A\m$ whose objects have filtration by objects in $\St$. Denote $\Stsim(\l)=\St(\l)\otimes_R U$, $U\in R\proj$, $\l\in\L$. Denote by $\mathcal{F}(\Stsim)$ the full subcategory of $A\m$ whose objects have filtrations by objects of the form $\St(\l)\otimes_R U$, \mbox{$U\in R\proj$,} $\l\in \L$. Here, we are abusing the notation by writing $\mathcal{F}(\Stsim)$ instead of $\mathcal{F}(\Stsim_{\l\in \L})$. 
\end{Notation}

Notice that the modules of $\mathcal{F}(\Stsim)$ with the $U_j$ being free $R$-modules are exactly the modules in $\mathcal{F}(\Stsim)$ and $\mathcal{F}(\St)=\mathcal{F}(\Stsim)$ exactly when all projective $R$-modules are free.
By Proposition \ref{everyprojectivehasafiltration} (see also \citep[Proposition 4.13]{Rouquier2008}) the category $\mathcal{F}(\Stsim)$ contains all projective $A$-modules.

\begin{Prop}(see \citep[Proposition 4.13]{Rouquier2008})
	\label{extonstandards}	Suppose $(A\m, \{\Delta(\lambda)_{\lambda\in \Lambda}\})$ is a split highest weight category. Then, 
	\begin{enumerate}[(i)]
		\item If $\Ext_A^1(\St(\l), \St(\mu))\neq 0$, then $\l<\mu$.
		\item If $\Ext_A^i(\St(\l), \St(\mu)\neq 0$ for some $i>0$, then $\l<\mu$. In particular, $\Ext_A^i(\St(\l), \St(\l))=0, \ i>0$.
	\end{enumerate}
\end{Prop}
\begin{proof}
Assume that $\Ext_A^1(\St(\l), \St(\mu))\neq 0$.	Consider the exact sequence $\delta\colon 0\rightarrow C(\l)\rightarrow P(\l)\rightarrow \St(\l)\rightarrow 0$. Applying $\Hom_A(-, \St(\mu))$ we obtain the exact sequence\begin{align*}
		\Hom_A(C(\l), \St(\mu))\rightarrow \Ext_A^1(\St(\l), \St(\mu))\rightarrow \Ext_A^1(P(\l), \St(\mu)).
	\end{align*} We deduce that $\Hom_A(C(\l), \St(\mu))\neq 0$. So, there is a factor of $C(\l)$, say $\St(\alpha)\otimes_R U_\alpha$ such that $\Hom_A(\St(\alpha), \St(\mu))\neq 0$. Thus, $\alpha\leq \mu$. Since $\St(\alpha)\otimes_R U_\alpha$ is a factor of $C(\l)$ we get that $\alpha>\l$. Thus, $\mu\geq \alpha>\l$. So, (i) follows.
	
	Now assume $\Ext_A^i(\St(\l), \St(\mu))\neq 0$ for some $i>0$. Applying $\Hom_A(-, \St(\mu))$ to $\delta$ we deduce that \linebreak$0\neq \Ext_A^i(\St(\l), \St(\mu))\simeq \Ext_A^{i-1}(C(\l), \St(\mu))$. Now consider the following filtration of $C(\l)$
	\begin{align}
		0\rightarrow C_1(\l)\rightarrow C(\l)\rightarrow \St(\alpha)\otimes_R U_\alpha \rightarrow 0.
	\end{align} Recall that its factors are of the form $\St(\alpha)\otimes_R U_\alpha$ with $\alpha>\l$ and $U_\alpha\in R\proj$. Applying again $\Hom_A(-, \St(\mu))$ it yields the exact sequence
	\begin{align}
		\Ext_A^{i-1}(C_1(\l), \St(\mu))\rightarrow \Ext_A^{i-1}(C(\l), \St(\mu))\rightarrow \Ext_A^{i-1}(\St(\alpha)\otimes_R U_\alpha, \St(\mu)).
	\end{align}
	We can assume that $\l$ is the maximal term that satisfies $\Ext_A^i(\St(\l), \St(\mu))\neq 0$ for some $i>0$. Otherwise, we can consider $\Ext_A^{i-1}(\St(\alpha), \St(\mu))\neq 0$ and repeat the process until either $\St(\alpha)$ is chosen to be projective or $i-1=1$. Then, we are in situation (i) and we are done since $\mu>\alpha>\l$. Thus, now assume $\Ext_A^{i-1}(\St(\alpha), \St(\mu))=0$. Hence, $\Ext_A^i(C_1(\l), \St(\mu))\neq 0$, we can continue the procedure using the factors of $C_1(\l)$ until either we get $\Ext_A^l(\St(\alpha), \St(\mu))\neq 0$ and $\alpha>\l$ which by previous discussion leads to $\mu>\alpha>\l$. In case,  $\Hom_A(C_1(\l), \St(\mu))\neq 0$ we will get $\Hom_A(\St(\alpha), \St(\mu))\neq 0$ for some $\St(\alpha)$ factor of $C_1(\l)$. Thus, $\l<\alpha\leq \mu$. In particular, if $\Ext_A^i(\St(\l), \St(\l))\neq 0$ for some $i>0$, then by (ii), $\l<\l$ which is an absurd.
\end{proof}

As in the classical case, Proposition \ref{extonstandards} says that all filtrations can be rearranged so that the standard modules with higher weight appear at the bottom of the filtration (see for example Proposition \ref{usualfiltrationisthisoneaswell}). In particular, all filtrations in $\mathcal{F}(\Stsim)$ can take the form of Proposition \ref{everyprojectivehasafiltration} (see also \citep[Proposition 4.13]{Rouquier2008}).

\section{Some properties of split highest weight categories}

Some known results in the theory of split highest weight categories over fields will be generalised to split highest weight categories over commutative Noetherian rings. The first step is to look at change of rings and at reducing questions on split highest weight categories over commutative Noetherian rings to questions on split highest weight categories over fields.

\subsection{Split highest weight categories under change of rings}\label{Split highest weight categories under change of rings}

\begin{Prop}\label{hwctensorproduct}
	(see \citep[Proposition 4.14]{Rouquier2008}) Let $S$ be a commutative $R$-algebra and Noetherian ring. Let $(A\m, \{\Delta(\lambda)_{\lambda\in \Lambda}\})$ be a split highest weight category. Then, $(S\otimes_R A\m, \{S\otimes_R\Delta(\lambda)_{\lambda\in \Lambda}\})$ is a split highest weight category. Moreover, $S\otimes_R (A/J)\m\simeq (S\otimes_RA)/(S\otimes_R J)\m$, where $J=\im \tau_{\St(\l)}$ for all $\l\in \L$ so that $\St(\l)\in \mathcal{M}(A)$.
\end{Prop}
\begin{proof}
	We shall proceed by induction on $t=|\L|$. Assume $t=1$. Hence, $\St(\l)\in \mathcal{M}(A)$. By Lemma \ref{splitmoduleschangering}, \mbox{$S\otimes_R \St(\l)\in \mathcal{M}(S\otimes_R A)$}. Fix $J=\im \tau_{\St(\l)}$. By Condition \ref{splithwc}(iii) of split highest weight category $A/J\m=0$ and $S\otimes_R J=S\otimes_R \im \tau_{\St(\l)}=\im (S\otimes_R \tau_{\St(\l)})=J_S$ is a submodule of $S\otimes_R A$ since $S\otimes_R \tau_{\St(\l)}$ is an $(S\otimes_R A, S)$-monomorphism. 
	
	By Proposition \ref{split modules characterization}, it follows that $A/J$ is projective over $R$, thus $\Tor_1^R(S, A/J)=0$. Thus, the sequence $0\rightarrow J \rightarrow A \rightarrow A/J\rightarrow 0$ remains exact under the functor $S\otimes_R -$. Hence, the sequence \begin{align}
		0\rightarrow J_S\rightarrow S\otimes_R A\rightarrow S\otimes_R A/S\otimes_R J\rightarrow 0
	\end{align}is exact. Thus, $S\otimes_R A/J\simeq S\otimes_R A/S\otimes_R J$ as $S$-algebras. Now assume the result is known for $t-1$. Let \mbox{$\alpha\in \L$} be a maximal element. Then, $\St(\alpha)\in \mathcal{M}(A)$ and $(A/J\m, \L\backslash \{\l\})$ is a split highest weight category with \mbox{$J=\im \tau_{\St(\alpha)}$}. Analogous to the case $t=1$, $S\otimes_R \St(\l)\in \mathcal{M}(S\otimes_R A)$ and $S\otimes_R A/S\otimes_R J=S\otimes_R A/J$ as $S$-algebras. By induction, $(S\otimes_R A/J\m, \L\backslash\{\alpha\})$ is split highest weight category with standard modules \mbox{$\{S\otimes_R \St(\l)\colon \l\in \L\backslash\{\alpha\} \}$.} By Lemma \ref{splithwcinduction}, $(S\otimes_R A\m, \L)$ is a split highest weight category.
\end{proof}

\begin{Theorem}
	(see \citep[Theorem 4.15]{Rouquier2008}) Let $A$ be a projective Noetherian $R$-algebra and let $\{\Delta(\lambda)_{\lambda\in \Lambda}\}$ be a set of finitely generated $A$-modules indexed by a poset.
	$(A\m, \{\Delta(\lambda)_{\lambda\in \Lambda}\})$ is a split highest weight category if and only if $\Delta(\lambda)$ are projective $R$-modules, $\l\in\L$ and $(A(\mi)\m, \{\Delta(\lambda)(\mi)_{\lambda\in \Lambda}\})$ is a split highest weight category for every maximal ideal $\mi$ of $R$.\label{hwcresiduefield}
\end{Theorem}
\begin{proof}
	For every maximal ideal $\mi$ in $R$, the residue field $R(\mi)$ is a Noetherian commutative algebra over $R$, so by Proposition \ref{hwctensorproduct} $(A(\mi)\m, \L)$ is split highest weight category with standards $\St(\l)(\mi)$. The modules $\St(\l)$ are projective over $R$ by definition of $A\m$ being a split highest weight category.
	
	Conversely, we shall proceed by induction on $t=|\L|$. Let $\mi$ be a maximal ideal of $R$. Assume $t=1$. By assumption $\St(\l)(\mi)\in \mathcal{M}(A(\mi))$. By Lemma \ref{splitmoduleschangering}, $\St(\l)\in \mathcal{M}(A)$. Let $M\in A\m$ be such that $\Hom_A(\St(\l), M)=0$. Then, $$\Hom_{A(\mi)}(\St(\l)(\mi), M(\mi))=0.$$ Since $A(\mi)\m$ is a split highest weight category $M(\mi)=0$ for every maximal ideal $\mi$ in $R$. Thus, $M=0$. Therefore, $(A\m, \L)$ is a split highest weight category.
	
	Now assume the result known for $t-1$. Let $\alpha$ be a maximal element in $\L$. By assumption, $(A(\mi)\m, \L)$ is a split highest weight category for every maximal ideal $\mi$ in $R$. By Lemma \ref{splithwcinduction}, $\St(\alpha)(\mi)\in \mathcal{M}(A(\mi))$ and $(A(\mi)/J(\mi)\m, \L\backslash\{\alpha \})$ is split highest weight category for every maximal ideal $\mi$ in $R$. Since $\St(\alpha)$ is projective over $R$, it follows, by Lemma \ref{splitmoduleschangering}, that $\St(\alpha)\in \mathcal{M}(A)$. Here, \begin{align*}
		J(\mi)=\im \tau_{R(\mi)\otimes_R \St(\alpha)}=R(\mi)\otimes_R \im \tau_{\St(\alpha)}=R(\mi)\otimes_R J.\end{align*} As $\St(\alpha)\in \mathcal{M}(A)$, $A/\im \tau_{\St(\alpha)}=A/J$ is a projective $R$-module. So, $\Tor_1^R(R(\mi), A/J)=0$. We deduce that \begin{align}
		A/J(\mi)=R(\mi)\otimes_R A/J\simeq R(\mi)\otimes_R A/R(\mi)\otimes_R J=A(\mi)/J(\mi).	\end{align}
	Thus, $(A/J(\mi)\m, \L\backslash\{\alpha \})$ is a split highest weight category for every maximal ideal $\mi$ in $R$. By induction, $(A/J\m, \L\backslash\{\alpha \})$ is a split highest weight category. Finally,  by Lemma \ref{splithwcinduction}, the result follows.
\end{proof}

\begin{Theorem}\label{hwclocalization}
	Let $A$ be a projective Noetherian $R$-algebra and let $\{\Delta(\lambda)_{\lambda\in \Lambda}\})$ be a set of finitely generated $A$-modules indexed by a poset. $(A\m, \{\Delta(\lambda)_{\lambda\in \Lambda}\})$ is a split highest weight category if and only if  $(A_\mi\m, \{(\Delta(\lambda)_\mi)_{\lambda\in \Lambda}\})$ is a split highest weight category for every maximal ideal $\mi$ of $R$. 
\end{Theorem}
\begin{proof}
	The proof is analogous to Theorem \ref{hwcresiduefield}. For every maximal ideal $\mi$ in $R$, $R_\mi$ is a Noetherian commutative ring which is an $R$-algebra. By Proposition \ref{hwctensorproduct}, if $(A\m, \{\Delta(\lambda)_{\lambda\in \Lambda}\})$ is a split highest weight category, then $(A_\mi\m, \{(\Delta(\lambda)_\mi)_{\lambda\in \Lambda}\})$ is a split highest weight category for every maximal ideal $\mi$ in $R$.
	
	Conversely, we shall proceed by induction on $t=|\L|.$ By assumption $\St(\l)_\mi\in \mathcal{M}(A_\mi)$ for every maximal ideal $\mi$ in $R$. By Lemma \ref{splitmoduleschangering}, $\St(\l)\in \mathcal{M}(A)$. Let $M\in A\m$ be such that $\Hom_A(\St(\l),M )=0$. Then, $\Hom_{A_\mi}(\St(\l)_\mi, M_\mi)=0$ which implies that $M_\mi=0$. Hence, $M=0$. Therefore, the result holds  for $t=1$.
	
	Assume the result known for $t-1$. Let $\alpha$ be a maximal element in $\L$. By Lemma \ref{splithwcinduction},  $\St(\alpha)_\mi\in \mathcal{M}(A_\mi)$ and $(A_\mi/J_\mi\m, \L\backslash\{\alpha \})$ is a split highest weight category for every maximal ideal $\mi$ in $R$. By Lemma \ref{splitmoduleschangering}, \mbox{$\St(\alpha)\in \mathcal{M}(A)$}. Since $R_\mi$ is flat over $R$, we deduce that $(A/\im \tau_{\St(\alpha)})_\mi=(A/J)_\mi\simeq A_\mi/J_\mi$. By induction, $(A/J\m, \L\backslash\{\alpha \})$ is a split highest weight category. By Lemma \ref{splithwcinduction}, the result follows.
\end{proof}

Parallelly to Lemma \ref{splitmodulesalgebraicclosure}, the following result says that determining whether a given module category with a certain collection of finitely generated modules is a split highest weight category can be reduced to determining if after extension of scalars to an algebraically closed field this collection of modules gives a highest weight structure to the module category of a finite-dimensional algebra over an algebraically closed field.

\begin{Theorem}\label{splitqhalgebraicallyclosed}
	Let $A$ be a finite-dimensional $k$-algebra for some field $k$ and let $\{\Delta(\lambda)_{\lambda\in \Lambda}\})$ be a set of finitely generated $A$-modules indexed by a poset. If $\overline{k}$ is the algebraic closure of $k$, then $(A\m, \{\Delta(\lambda)_{\lambda\in \Lambda}\})$ is a split highest weight category if and only if $(\overline{k}\otimes_k A\m, \{\overline{k}\otimes_k\Delta(\lambda)_{\lambda\in \Lambda}\})$ is a split highest weight category.
\end{Theorem}
\begin{proof}
	The result follows by Proposition \ref{hwctensorproduct}, Lemma \ref{splithwcinduction} and \ref{splitmodulesalgebraicclosure}.
\end{proof}

Given the formulations of Theorems \ref{hwcresiduefield} to \ref{splitqhalgebraicallyclosed}, we can ask whether there is a version involving the quotient field of an integral domain. The following tries to address this question and it aims to generalise Lemma 1.6 of \citep{zbMATH01139665}.

\begin{Lemma}\label{splithwcintermsofquotientfield}
	Let $R$ be a regular domain with quotient field $K$. Let $A$ be a projective Noetherian $R$-algebra. Assume that $\{\Delta(\lambda)_{\lambda\in \Lambda}\})$ is a set of finitely generated $A$-modules indexed by a poset and the following conditions hold:
	\begin{enumerate}[(i)]
		\item For $\l\in \L$, $\St(\l)\in R\proj$;
		\item For each $\l\in \L$, there exists a projective $A$-module $P(\l)$ so that there is an exact sequence
		\begin{align}
			0\rightarrow C(\l)\rightarrow P(\l)\rightarrow \St(\l)\rightarrow 0,
		\end{align}where $C(\l)\in \mathcal{F}(\Stsim(\mu)_{\mu>\l})$;
		\item $\displaystyle  \bigoplus_{\l\in \L} P(\l)$ is a progenerator for $A\m$.
	\end{enumerate}
	Then, $(A\m, \{\Delta(\lambda)_{\lambda\in \Lambda}\})$ is a split highest weight category if and only if $(K\otimes_R A\m, K\otimes_R\Delta(\lambda)_{\lambda\in \Lambda}\})$ is a split highest weight category. 
\end{Lemma}
\begin{proof}
	By Proposition \ref{hwctensorproduct}, one of the implications is clear. 
	
	Conversely, assume that \mbox{$(K\otimes_R A\m, K\otimes_R\Delta(\lambda)_{\lambda\in \Lambda}\})$} is a split highest weight category.  It is enough to show that $\End_A(\St(\l))\simeq R$ and Condition \ref{splithwc}(ii). Suppose that $\Hom_A(\St(\l), \St(\mu))\neq 0$. Then,
	\begin{align}
		0\neq K\otimes_R \Hom_A(\St(\l), \St(\mu))\simeq \Hom_{K\otimes_R A}(K\otimes_R \St(\l), K\otimes_R \St(\mu)).
	\end{align} Hence, $\l\leq \mu$. Let $\pri$ be a prime ideal of $R$ with height one. $K$ is the quotient field of $R_\pri$ and $\dim R_\pri=1$. In particular,
	\begin{align}
		K\otimes_{R_\pri} \End_{A_\pri}(\St(\l)_\pri)\simeq \End_{K\otimes_{R_\pri}A_\pri} (K\otimes_{R_\pri}\St(\l)_\pri)\simeq \End_{K\otimes_R A}(K\otimes_R \St(\l))\simeq K. \label{eqqh54}
	\end{align}On the other hand, using the monomorphism $\End_{A_\pri}(\St_\pri(\l))\rightarrow \Hom_{A_\pri}(P(\l)_\pri, \St(\l)_\pri)$ we obtain that  $\End_{A_\pri}(\St_\pri(\l))\in R_\pri\proj$. Thus, (\ref{eqqh54}) implies that $\End_{A_\pri}(\St_\pri(\l))\simeq R_\pri$. This shows that $\End_{A_\pri}(\St_\pri)$ is a maximal order in $K$. By Theorem 1.5 of \citep{zbMATH03190382}, $\End_A(\St(\l))$ is a maximal order in $K$. By Theorem 4.3 of \citep{zbMATH03190382}, we conclude that $\End_A(\St(\l))\simeq R$.
\end{proof}

\begin{Remark}
	If, in addition to knowing \ref{splithwcintermsofquotientfield}(i) we know that $\St(\l)$ is $R$-faithful, then we can consider another approach without using maximal orders. In fact, $\End_A(\St(\l))$ is torsion free over $R$ and there exists an exact sequence $0\rightarrow R\rightarrow \End_A(\St(\l))\rightarrow X\rightarrow 0$. By Proposition 3.4 of \citep{zbMATH03151673}, if $X\neq 0$, then $X_\pri\neq 0$ for some prime ideal of $R$ with height one. But, as we showed this cannot happen.
\end{Remark}

\subsection{Uniqueness of standard modules with respect to the poset $\L$}\label{Uniqueness of standard modules with respect to the poset}

We are now ready to address some questions concerning the uniqueness of standard modules and the projective modules $P(\l)$. Given the existence of $\St(\l)$ we saw that the projective modules $P(\l)$ given by the condition \ref{splithwc}(iv) of split highest weight category are not unique up to isomorphism. However, we saw that for split quasi-hereditary algebras over fields we could replace the projective modules in \ref{splithwc}(iv) with indecomposable projective modules. In the following, we will see a sort of generalisation of this phenomenon to general commutative rings.

\begin{Prop}\label{uniquenessofprojgivenst}
	Let $(A\m, \{\Delta(\lambda)_{\lambda\in \Lambda}\})$ be a split highest weight category so that the projective modules $P(\l)$ in \ref{splithwc}(iv) become indecomposable under $R(\mi)\otimes_R -$ for every maximal ideal $\mi$ of $R$. Assume that there exists $Q(\l)\in A\proj$ which becomes indecomposable under $R(\mi)\otimes_R -$ together with an exact sequence \begin{align*}
		0\rightarrow S(\l)\rightarrow Q(\l) \xrightarrow{p_\l}\St(\l)\rightarrow 0, \quad  \text{such that}\quad  S(\l)\in \mathcal{F}(\Stsim_{\mu>\l}).
	\end{align*} Then, there is an isomorphism  $g\colon Q(\l)\rightarrow P(\l)$ making the following diagram commutative
	\begin{center}
		\begin{tikzcd}
			Q(\l)\arrow[rr,"p_\l"] \arrow [dr, "g", swap] & & \St(\l)\\
			& P(\l)\arrow[ur, "\pi_\l", swap] &
		\end{tikzcd}.
	\end{center} 
\end{Prop}
\begin{proof}
	Since $P(\l)$ and $Q(\l)$ are projective $A$-modules, there are $A$-homomorphisms $f$ and $g$ making the following diagram commutative:
	\begin{equation}
		\begin{tikzcd}
			P(\l)\arrow[r, "\pi_\l"]\arrow[d, "f"]& \St(\l)\arrow[r] \arrow[d, equal]& 0\\
			Q(\l)\arrow[r,"p_\l"] \arrow [d, "g"] & \St(\l)\arrow[r] \arrow[d, equal]& 0\\
			P(\l)\arrow[r, "\pi_\l"] & \St(\l)\arrow[r]& 0
		\end{tikzcd}.
	\end{equation} 
	Applying the right exact functor $R(\mi)\otimes_R -$ for every maximal ideal $\mi$ of $R$ we obtain the commutative diagram
	\begin{equation}
		\begin{tikzcd}
			P(\l)(\mi)\arrow[r, "\pi_\l(\mi)"]\arrow[d, "f(\mi)"]& \St(\l)(\mi)\arrow[r] \arrow[d, equal]& 0\\
			Q(\l)(\mi)\arrow[r,"p_\l(\mi)"] \arrow [d, "g(\mi)"] & \St(\l)(\mi)\arrow[r] \arrow[d, equal]& 0\\
			P(\l)(\mi)\arrow[r, "\pi_\l(\mi)"] & \St(\l)(\mi)\arrow[r]& 0
		\end{tikzcd}. \label{qhtikz52}
	\end{equation} Note that $g(\mi)\circ f(\mi)=g\otimes \id_{R(\mi)} \circ f \otimes \id_{R(\mi)}=g\circ f\otimes \id_{R(\mi)}=g\circ f(\mi)$. For any maximal ideal $\mi$ of $R$, $(A(\mi), \{\Delta(\lambda)(\mi)_{\lambda\in \Lambda}\})$ is a split highest weight category with projectives $P(\l)(\mi)$ and $Q(\l)(\mi)$. Further, $(Q(\l)(\mi), \pi_\l(\mi))$ and $(P(\l)(\mi), p_\l(\mi))$ are projective covers of $\St(\l)(\mi)$. It follows by diagram (\ref{qhtikz52}) that $P(\l)\subset \im g\circ f(\mi)+\ker pi_\l(\mi)$. Since $\ker \pi_\l(\mi)$ is a superfluous module, it follows that $g\circ f(\mi)$ is surjective for every maximal ideal $\mi$ of $R$. By Nakayama's lemma, $g\circ f$ is surjective. Since $g\circ f\in \End_A(P(\l))$, this surjective must be an isomorphism by Nakayama's Lemma for endomorphisms. Since  $(Q(\l)(\mi), \pi_\l(\mi))$ is a projective cover of $\St(\l)(\mi)$, it follows, by symmetry, that $f\circ g$ is an isomorphism. Hence, both $f$ and $g$ are isomorphisms. So, the claim follows.
\end{proof}

Before we proceed any further we should pay attention to the following fact.

\begin{Observation}\label{obs563}
	Assume that $R$ is a local commutative Noetherian ring with unique maximal ideal $\mi$ and $A\m$ is a split highest weight category with standard modules $\St(\mu)$, $\mu\in \L$. Then, we can pick the projective modules in \ref{splithwc}(iv) so that they become indecomposable under $R(\mi)\otimes_R -$. Such construction can be made by reverse induction. If $\l\in \L$ is maximal, then define $P(\l):=\St(\l)$. For the induction step, assume that $\mu$ is maximal in $\L\backslash \{\l \}$ and $\l$ is maximal in $\L$. The Picard group of $R$ is trivial and the multiplicity of $\St(\l)$ in the projective associated with $\St(\mu)$ is controlled by $\Ext_A^1(\St(\mu), \St(\l))\in R\m$ in view of Lemma \ref{AandAmoduloJproj}. Since all extensions between $\St(\mu)$ and $\St(\l)$ are $(A, R)$-exact sequences we can pick by Nakayama's Lemma a minimal set of generators for $\Ext_A^1(\St(\mu), \St(\l))$ of size $\dim_{R(\mi)} \Ext_{A(\mi)}^1(\St(\mu)(\mi), \St(\l)(\mi))$. Using Lemma \ref{AandAmoduloJproj}, this means that we can construct $P(\mu)$ so that the multiplicities of $\St(\l)$ in $P(\mu)$ and of $\St(\l)(\mi)$ in the projective cover of $\St(\mu)(\mi)$ over $A$ coincide. Hence, $P(\mu)$ can be constructed so that $P(\mu)(\mi)$ is the projective cover of $\St(\mu)(\mi)$.
\end{Observation}

In the field case, given an order on $\L$, the standard modules when defined are unique (see for example \citep[A]{MR1284468}). This result can be extended to local commutative Noetherian rings in the following way.

\begin{Prop}\label{uniquenessofstandardgivenproj}
	Let $(A\m, \{\Delta(\lambda)_{\lambda\in \Lambda}\})$ be a split highest weight category over a local commutative Noetherian ring. Let $\St\rightarrow \{1, \ldots, t\}$, $\St_i\mapsto i$ be an increasing bijection. Choose $P_i\in A\proj$ so that $P_i(\mi)$ is the projective cover of $\St_i(\mi)$ for all $i\in \{1, \ldots, t \}$.
	Define \begin{align*}
		U_i=\sum_{j>i}\sum_{f\in \Hom_A(P_j, P_i)} \im f.
	\end{align*} Then, $\St_i\simeq P_i/U_i$.
\end{Prop}
\begin{proof}Let $\mi$ be the unique maximal ideal of $R$.
	By Theorem \ref{hwcresiduefield},  $(A(\mi), \{\Delta(\lambda)(\mi)_{\lambda\in \Lambda}\})$ is split highest weight category. Since $R(\mi)$ is a field, $\St_i(\mi)\simeq P_i(\mi)/C_i(\mi)$. We have,
	\begin{align*}
		C_i(\mi)&\simeq \sum_{j>i}\sum_{f\in \Hom_{A(\mi)}(P_j(\mi), P_i(\mi))} \im f\simeq \sum_{j>i}\sum_{f\in \Hom_{A}(P_j, P_i)(\mi)}\im f\\
		&\simeq \sum_{j>i}\sum_{f\in \Hom_{A}(P_j, P_i)} \im (f\otimes_R \id_{R(\mi)})\simeq \left( \sum_{j>i}\sum_{f\in \Hom_{A}(P_j, P_i)} \im f\right) (\mi)=U_i(\mi)
	\end{align*}
	
	\textit{Claim.} $\Hom_A(M, P_i/U_i)=0$ for $M\in \mathcal{F}(\St_{j>i})$.
	We shall proceed by induction on the size of the filtration of $M$. Assume $t=1$. Then, $M\simeq \St_j$ for some $j>i$. Let $g\in \Hom_A(\St_j, P_i/U_i)$. Since $P_j$ is projective over $A$ we have a commutative diagram
	\begin{center}
		\begin{tikzcd}
			P_j\arrow[r, "\pi_j"]\arrow[drr, "\exists f", swap] & \St_j \arrow[r, "g"]& P_i/U_i\\
			& & P_i \arrow[u, twoheadrightarrow, "\pi"]
		\end{tikzcd}.
	\end{center}
	By definition of $\pi$ and $U_i$, $0=\pi \circ f= g\circ \pi_j $. Then, $g=0$, since $\pi_j$ is surjective. Now consider the result known for filtrations of size less than $t$. Assume that $M$ has a filtration with size $t$. Let $g\in \Hom_A(\St_j, P_i/U_i)$. Consider the exact sequence $0\rightarrow M_{t-1}\xrightarrow{i}M\xrightarrow{k} \St_j\rightarrow 0$, $j>i$. By induction, $\Hom_A(M_{t-1}, P_i/U_i)=0$. In particular, $g\circ i=0$. So, $g$ induces a map $g'\in \Hom_A(\St_j, P_i/U_i)$ such that $g'\circ k=g$. By $t=1$, $g'=0$. Therefore, $g=0$ and the claim follows.
	
	Consider the following diagram
	\begin{center}
		\begin{tikzcd}
			0\arrow[r]& C_i\arrow[r, "k_i"]& P_i\arrow[d, equal]\arrow[r, "\pi_i"]& \St_i\arrow[r]&0\\
			0\arrow[r]&U_i\arrow[r, "k"]&P_i \arrow[r, "\pi"]&P_i/U_i\arrow[r]&0
		\end{tikzcd}.
	\end{center}
	Since $C_i\in \mathcal{F}(\St_{j>i})$ we get that $\Hom_A(C_i, P_i/U_i)=0$. In particular, $\pi\circ k_i=0$. So, the image of $k_i$ is contained  in $\ker\pi=\im k$, and thus there exists an $A$-homomorphism $f\colon C_i\rightarrow U_i$ which makes the previous diagram commutative. On the other hand, since $\pi\circ k_i=0$ there exists a map $\tilde{\pi}\in \Hom_A(\St_i, P_i/U_i) $  such that $\tilde{\pi}\circ \pi_i=\pi$. By Snake Lemma, $f$ is injective and $\tilde{\pi}$ is surjective. For every maximal ideal $\mi$ of $R$, applying the right exact functor $R(\mi)\otimes_R -$ yields the commutative diagram with exact rows:
	\begin{center}
		\begin{tikzcd}
			0\arrow[r]& C_i(\mi)\arrow[r, "k_i(\mi)"]\arrow[d, "f(\mi)"]& P_i(\mi)\arrow[d, equal]\arrow[r, "\pi_i(\mi)"]& \St_i(\mi)\arrow[d, "\tilde{\pi}(\mi)"]\arrow[r]&0\\
			&U_i(\mi)\arrow[r, "k(\mi)"]&P_i(\mi) \arrow[r, "\pi(\mi)"]&P_i/U_i(\mi)\arrow[r]&0
		\end{tikzcd}.
	\end{center}
	The first row is exact since $\St_i$ is projective over $R$. By the commutativity of the diagram, \linebreak${k(\mi)\circ f(\mi)}=k_i(\mi)$ is injective, which implies that $f(\mi)$ is a monomorphism. Since $C_i(\mi)\simeq U_i(\mi)$, we have $\dim_{R(\mi)} C_i(\mi)=\dim_{R(\mi)} U_i(\mi)$, thus $f(\mi)$ is an $R(\mi)$-isomorphism for every maximal ideal $\mi$ of $R$. Thus, $f(\mi)$ is an $A(\mi)$-isomorphism. By Nakayama's Lemma, $f$ is surjective. Hence, $f$ is an isomorphism. By Snake Lemma, $\tilde{\pi}$ is an isomorphism and it follows that $\St_i\simeq P_i/U_i$. 
\end{proof}

Note that this does not guarantee uniqueness of standard modules as in the field case, since in Noetherian rings we can have many choices for the projective modules $P(\l)$ even when they are indecomposable.

A natural question that arises is whether or not the projective modules $P(\l)$ are indecomposable. In the following proposition, we find a positive answer for  commutative Noetherian local rings.

\begin{Prop}\label{projisindecomposable}
	Let $(A\m, \{\Delta(\lambda)_{\lambda\in \Lambda}\})$ be a split highest weight category. If $R$ has no non-trivial idempotents, then all $\St(\l)$ are indecomposable. Furthermore, if $R$ is a commutative Noetherian local ring,  then there exists a choice of $P(\l)$ satisfying \ref{splithwc}(iv) so that  $\End_A(P(\l))$ is a local ring.
\end{Prop}
\begin{proof}
	Assume by contradiction that $\St(\l)=X_1\oplus X_2$ then $\St(\l)\twoheadrightarrow X_1 \hookrightarrow \St(\l)$ is a non-trivial idempotent in $\End_A(\St(\l))^{op}.$ Thus, we have a non-trivial idempotent in $R$. 
	
	Assume that $R$ is local. Let $f\in \End_A(P(\l))$. Let $\mi$ be the unique maximal ideal in $R$. Then, $f(\mi)\in \End_{A(\mi)}(P(\l)(\mi))$, since $P(\l)\in A\proj$. By Observation \ref{obs563}, we can consider projective modules $P(\l)$ so that $P(\l)(\mi)$ is indecomposable. In view of Proposition \ref{splithwcfieldcase}, $\End_{A(\mi)}(S)\simeq R(\mi)$ for all simple $A(\mi)$-modules. Thus, the endomorphism ring of a finite-dimensional indecomposable $A(\mi)$-module is a local ring. In particular, $\End_{A(\mi)}(P(\l)(\mi))$ is a local ring. Hence, if $f(\mi)$ is not an isomorphism, then $\id_{P(\mi)}-f(\mi)$ is an isomorphism. Note that $\id_{P(\mi)}-f(\mi)=(\id_P-f)(\mi)$. Applying Nakayama Lemma's \ref{nakayamalemmasurjectiveproj}, it follows that $\id_P-f$ is an isomorphism or $f$ is an isomorphism or both. Thus, $\End_A(P(\l))$ is a local ring and  $P(\l)$ is indecomposable.
\end{proof}

\subsection{Split heredity chains} \label{Relation between heredity chains and standard modules}

We shall now discuss another approach to split highest weight categories using split heredity chains.

	\begin{Def}
	Let $R$ be a commutative Noetherian ring and let $A$ be a projective Noetherian $R$-algebra. Let $J$ be an ideal of $A$. We call $J$ a \textbf{split heredity ideal} of $A$ if the following holds:
	\begin{enumerate}[(i)] \setlength\itemsep{0em}
		\item $A/J$ is projective over $R$;
		\item $J$ is projective as left ideal over $A$;
		\item $J^2=J$;
		\item The $R$-algebra $\End_A(_AJ)^{op}$ is Morita equivalent to $R$.
	\end{enumerate}\label{splithereditydef}
\end{Def} 
Split heredity ideals were defined by \citep{CLINE1990126} and later studied in \citep{Rouquier2008}.

\begin{Def}\label{qhdef}
	A projective Noetherian $R$-algebra $A$ is called \textbf{split quasi-hereditary} if there exists a finite split heredity chain of ideals $0=J_{t+1}\subset J_t\subset \cdots \subset J_1=A$ such that $J_i/J_{i+1}$ is a split heredity ideal in $A/J_{i+1}$ for $1\leq i\leq t$. 
\end{Def} It is clear from the definition that if the ground ring is a field then this concept coincides with the usual concept of split quasi-hereditary algebras over fields.

The difference between split quasi-hereditary algebras and quasi-hereditary lies in Condition (iv) of Definition \ref{splithereditydef}. In fact, for quasi-hereditary finite-dimensional algebras over a field instead of (iv) it is only necessary to impose that the endomorphism algebra $\End_A(J)^{op}$ is semi-simple. Over algebraically closed fields, all quasi-hereditary algebras are split quasi-hereditary. Although we prefer to work with split hereditary algebras because they are better equipped for the integral setup behaving quite well under change of rings in contrast with (non-split) quasi-hereditary algebras.

From definition, it is clear that for split quasi-hereditary algebras the regular module $A$ is faithful as $R$-module.

	\begin{Prop} \label{quotientbysplitheredity}
		Let $A$ be a projective Noetherian $R$-algebra. 
	The algebra $A$ is split quasi-hereditary if and only if $A/J$ is split quasi-hereditary for some split heredity ideal $J$ of $A$.
\end{Prop}
\begin{proof}
	Assume that $A$ has a split heredity chain $0=J_{t+1}\subset J_t\subset \cdots J_1=A$. Consider $J=J_t$. The chain of ideals $0=J_t/J\subset J_{t-1}/J \subset J_1/J=A/J$ in $A/J$ is split heredity. In fact, $J_i/J/J_{i+1}/J\simeq J_i/J_{i+1}$ in $A/J/J_{i+1}/J\simeq A/J_{i+1}$. As $J_i/J_{i+1}$ is split heredity in $A/J_{i+1}$ and $\End_{A/J/J_{i+1}/J}(J_i/J/J_{i+1}/J)\simeq \End_{A/J_{i+1}}(J_i/J_{i+1})$ we obtain that $A/J$ is a split quasi-hereditary algebra.
	
Conversely, assume that $A/J$ is split quasi-hereditary where $J$ is some split heredity ideal.	By assumption, $0=I_t\subset I_{t-1} \subset\cdots\subset I_1=A/J$ is a split heredity chain. Now each ideal in $A/J$ can be written as $J_i/J=I_i, \ t\leq i \leq 1$ by the correspondence theorem for quotient rings. Here, $J_i/J_{i+1}\simeq J_i/J/J_{i+1}/J\simeq I_i/I_{i+1}$ as $A$-modules and $\End_{A/J/J_{i+1}/J}(J_i/J/J_{i+1}/J)\simeq \End_{A/J_{i+1}}(J_i/J_{i+1})$. Therefore, $J_i/J_{i+1}$ is split heredity in $A/J_{i+1}$. So, $0\subset J\subset J_{t-1} \subset \cdots \subset J_1=A$ is a split heredity chain.
\end{proof}

The following result is Theorem 4.16 of \cite{Rouquier2008}. 
\begin{Theorem}\label{quasihereditaryhwc}
	Let $A$ be a projective Noetherian $R$-algebra. $(A\m, \{\Delta(\lambda)_{\lambda\in \Lambda}\})$ is a split highest weight category if and only if $A$ is a split quasi-hereditary algebra. Let $\St\rightarrow \{1, \ldots, t\}$, $\St_i\mapsto i$ be an increasing bijection. Here the standard modules and the split heredity chain are related in the following way:
	\begin{align*}
		\im \tau_{\St_i}=J_i/J_{i+1} \text{ in } A/J_{i+1}, \quad J_{t+1}=0\subset J_t\subset J_{t-1}\subset \cdots \subset J_1=A \quad \text{is a split heredity chain.}
	\end{align*}
Here, $\tau_{\St_i}$ denotes the map $\St_i\otimes_R \Hom_{A/J_{i+1}}(\St_i, A/J_{i+1})\rightarrow A/J_{i+1}, $ given by $ l\otimes f\mapsto f(l)$.
\end{Theorem}
\begin{proof}
	Let $A$ be split quasi-hereditary with split heredity chain: $J_{t+1}=0\subset J_t\subset J_{t-1}\subset \cdots \subset J_1=A$. We shall proceed by induction on the size of the split heredity chain of $A$ to show that $A\m$ can have a split highest weight category structure.
	
	Assume $t=1$. Then, $0\subset A$ is a split heredity chain. So, $A$ is split heredity in $A$. By Proposition \ref{bijectionsplitheredity}, there is $L\in \mathcal{M}(A)$ such that $\im \tau_L=A$. Put $\St(1)=L$ and since $A/\im \tau_L=0$, it follows by Lemma \ref{splithwcinduction} that $(A\m, \{\St(1)\} )$ is a split highest weight category.
	
	Assume now that the result holds for $t-1$. Fix $J=J_t$. $A/J$ is split quasi-hereditary with split heredity chain $0\subset J_{t-1}/J\subset \cdots \subset J_1/J=A/J$. By induction, $A/J\m$ is a split highest weight category with standards $\St(i)$, \mbox{$1\leq i\leq t-1$,} satisfying $\im \tau_{\St(i)}=(J_i/J)/(J_{i+1}/J)\simeq J_i/J_{i+1}$, $1\leq i\leq t-1$. By Proposition \ref{bijectionsplitheredity}, there is \mbox{$L\in \mathcal{M}(A)$} such that $\im \tau_L=J$. Put $\St(t)=L$. Since each $\St(i)\in A/J\m$, we get that $\Hom_A(\St(t), \St(i))=0, \ 1\leq i\leq t-1$, by Corollary  \ref{full subcategory in terms of split module}. So, we can consider the usual order $t\geq i,\ 1\leq i\leq t-1$. By  Lemma \ref{splithwcinduction}, $(A\m, \{ \St(i)_{i\in \{1, \ldots, t\} } \})$ is a split highest weight category.
	
	Now assume that $(A\m, \{\Delta(\lambda)_{\lambda\in \Lambda}\})$ is a split highest weight category. Let $\Lambda\rightarrow \{1, \ldots, t\}$, $\lambda\mapsto i_\l$ be an increasing bijection.
	We shall proceed by induction on $t$. If $t=|\L|=1$, then $L=\St(1)\in \mathcal{M}(A)$. By Proposition \ref{bijectionsplitheredity}, there exists $J$ split heredity such that $J=\im \tau_L$. By Corollary \ref{full subcategory in terms of split module}, $A/J\m=0$. In particular, $A/J=0$, so $J=A$. Thus, $0\subset J=A$ is a split heredity chain.
	
	Now assume the result known for $t-1$. Consider  a maximal element $\alpha\in \L$ satisfying $i_\alpha=t$. By Lemma \ref{splithwcinduction}, $\St(t)\in \mathcal{M}(A)$ and $(A/J\m, \{\Delta(i)_{1\leq i\leq t-1}\})$ is a split highest weight category, where $J=\im \tau_{\St(t)}$. By induction, there exists a split heredity chain 
	\begin{align*}
		0\subset I_{t-1}\subset \cdots\subset I_1=A/J\quad \text{ such that } \im \tau_{\St(i)}=I_i/I_{i+1}.
	\end{align*} Fix $J_t=J$.
	By the correspondence theorem, there are ideals $J_i$ of $A$ such that $I_i=J_i/J$. It follows that $J_i/J_{i+1}\simeq J_i/J/J_{i+1}/J=I_i/I_{i+1}=\im \tau_{\St(i)}$ split heredity in $A/J/J_{i+1}/J\simeq A/J_{i+1}$ for $i=1, \ldots, t-1$. Now since $J_t$ is split heredity in $A$, it follows by the discussed argument above that $0=J_{t+1}\subset J_t\subset J_{t-1}\subset \cdots\subset J_1=A$ is a split heredity chain of $A$ satisfying $\im \tau_{\St(i)}= J_i/J_{i+1}$.
\end{proof}

Due to Theorem \ref{quasihereditaryhwc}, we can say that $(A, \{\Delta(\lambda)_{\lambda\in \Lambda}\})$ is a split quasi-hereditary algebra when \linebreak$(A\m, \{\Delta(\lambda)_{\lambda\in \Lambda}\})$ is a split highest weight category without mentioning the split heredity chain.

Note that, by the bijection given in Proposition \ref{bijectionsplitheredity}, the standard modules are not unique in this construction unless the Picard group is trivial.

In the construction of the standard modules, in general, there are many choices of suitable standard modules that can be given using the same heredity chain. Although, they all lead to the same highest weight structure for $A$. In fact, we have the following concept.

\begin{Def}\label{equivalencehwc}
	Let $(A\m, \{\Delta(\lambda)_{\lambda\in \Lambda}\})$ and $(B\m, \{\Omega(\chi)_{\chi\in X}\})$ be two split highest weight categories. A functor $F\colon A\m\rightarrow B\m$ is an \textbf{equivalence of split highest weight categories} if 
	\begin{enumerate}[(i)]
		\item it is an equivalence of categories;
		\item there is a bijection $\phi\colon \Lambda \rightarrow X$ and invertible $R$-modules $U_\l$ such that $F(\St(\l))\simeq \Omega(\phi(\lambda))\otimes_R U_\l, \ \l\in\L$.
	\end{enumerate}
\end{Def}

It is also clear from the definition of split quasi-hereditary that if two projective Noetherian $R$-algebras $A$ and $B$ are Morita equivalent and if $(A, \St)$ is a split quasi-hereditary algebra then so is $(B, F\St)$, for some equivalence of categories $F$.

\begin{Remark}\label{Remarkpassingtofields}
	Note that when passing from $A$ to $A(\mi)$, there is no confusion if we take equivalent standard modules. That is, $\St_i'(\mi)=\St_i(\mi)$.
\end{Remark}
In fact, consider $\St_i'=\St_i\otimes_R F, \ F\in Pic(R)$. Then, there exists $G$ such that $G\otimes_R F\simeq R$. Moreover,
\begin{align*}
	F(\mi)\otimes_{R(\mi)}G(\mi)\simeq F\otimes_R R(\mi)\otimes_R G \simeq F\otimes_R G\otimes_R R(\mi)\simeq R\otimes_R R(\mi)\simeq R(\mi). 
\end{align*}Hence, $F(\mi)\in Pic(R(\mi))=\{R(\mi)\}$ since $R(\mi)$ is a field.
Therefore, \begin{align*}
	\St_i'(\mi)=\St_i\otimes_R F (\mi)\simeq \St_i(\mi)\otimes_{R(\mi)}F(\mi)\simeq \St_i(\mi).
\end{align*}

Now we must observe that Remark 4.18 in \cite{Rouquier2008} is not accurate. Theorem 3.3 in \citealp{CLINE1990126} does not involve split quasi-hereditary algebras, but instead, it involves (non-split) quasi-hereditary algebras. On the other hand, in general, we cannot construct standard modules $\St$ just knowing the modules over the residue field. Here the difficulty lies that a priori there is not an $R$-homomorphism that its image under the functor $-\otimes_R R(\mi)$ is the isomorphism. This problem also occurs when dealing with localizations.  

So to conclude the split version of Corollary 3.4 of \cite{CLINE1990126}, a direct approach like in its proof might not work in this case. We suggest the following:

\begin{Prop}
	Let $R$ be a Noetherian commutative ring and $A$ a projective Noetherian $R$-algebra. $J$ is a split heredity ideal in $A$ if and only if $J$ is split heredity ideal in $A^{op}$. \label{heredityop}
\end{Prop}
\begin{proof} %
Assume first that $R$ is a field. In such a case, we can write $J=AeA$ for some primitive idempotent $e$ of $A$.
	The result holds for heredity ideals (see \citep[Theorem 4.3 (b)]{PS88}). Now assume that $\End_A(AeA)^{op}$ is Morita equivalent to $ R$. 
	We have that 
	\begin{align*}
		\End_A(Ae)^{op}\simeq eAe \simeq \End_A(eA).
	\end{align*}Since $\End_A(_A AeA)^{op}$ is Morita equivalent to $\End_A(Ae)^{op}$ and $	\End_A(AeA_A)$ is Morita equivalent to $ \End_A(eA)$ we obtain that $\End_A(_A AeA)^{op}$ is Morita equivalent to $ \End_A(AeA_A)$.
Hence,  $\End_A(AeA_A)$ is Morita equivalent  $R$, and the result follows for fields.
	
	Now assume that $R$ is a commutative Noetherian  ring and $J$ is a split heredity ideal in $A$. Conditions (i) and (ii) of Definition \ref{splithereditydef} for $J^{op}$ in $A^{op}$ are clear. 
	Moreover, for any maximal ideal $\mi$ of $R$, $A/J(\mi)\simeq A(\mi)/J(\mi)$, since $A/J$ is projective over $R$, and thus $\Tor_1^R(A/J, R(\mi))=0$.
	We have $J(\mi)^2=J\otimes_R R(\mi) J\otimes_R R(\mi)= J^2\otimes_R R(\mi)=J(\mi)$, and clearly $J(\mi)$ is projective as left $A(\mi)$-module.
	
	Since $\End_A(_AJ)^{op}$ is Morita equivalent to $ R$, there exists an $R$-progenerator $P$ such that $\End_A(_AJ)^{op}\simeq \End_R(P)^{op}$. Therefore,
	\begin{align*}
		\End_{A(\mi)}(J(\mi))^{op}&\simeq R(\mi)\otimes_R \End_A(J)^{op}, \text{ since } J\in A\proj\\
		&\simeq R(\mi)\otimes_R \End_R(P)^{op} \simeq \End_{R(\mi)}(P(\mi))^{op}, \text{ since } P\in R\proj.
	\end{align*}
	Since the functor $R(\mi)\otimes_R -$ preserves finite direct sums, it preserves the progenerators, hence $P(\mi)$ is an $R(\mi)$-progenerator and $\End_{A(\mi)}(J(\mi))^{op}$ is Morita equivalent to  $R(\mi)$. So, $J(\mi)$ is split heredity in $A(\mi)$. Since $R(\mi)$ is a field, $J(\mi)^{op}$ is split heredity in $A(\mi)^{op}$. In particular, $J(\mi)$ is projective as right $A(\mi)$-module for every maximal ideal $\mi$ of $R$. By Theorem \ref{projectivitytermsmaximal}, $J$ is projective as right $A$-module. 
	
	Consider $L\in \mathcal{M}(A)$ such that $\im \tau_L=J$. By Proposition \ref{bijection splits}, $\End_R(\Hom_A(L, A))\simeq \Hom_A(J, A)$. By Remark \ref{remendprojonma}, $\End_R(L)\simeq \Hom_{A^{op}}(J^{op}, A)$. Since both $\Hom_A(L, A)$ and $L$ are $R$-progenerators, it follows that $\add_R \Hom_A(L, A) =\add_R L$, thus $\End_R(\Hom_A(L, A))^{op}$ is Morita equivalent to $\End_R(L)^{op}$.
	
	Now applying $\Hom_A(J, -)$ and $\Hom_{A^{op}}(J^{op}, -)$ to the exact sequence $0\rightarrow J\rightarrow A\rightarrow A/J\rightarrow 0$ yields the following exact sequences
	\begin{align*}
		0\rightarrow \Hom_A(J, J)\rightarrow \Hom_A(J, A)\rightarrow \Hom_A(J, A/J)\rightarrow 0\\
		0\rightarrow \Hom_{A^{op}}(J^{op}, J^{op})\rightarrow \Hom_{A^{op}}(J^{op}, A^{op})\rightarrow \Hom_{A^{op}}(J^{op}, A^{op}/J^{op})\rightarrow 0.
	\end{align*}
	By Lemma \ref{A/Jmod}, $\Hom_{A^{op}}(J^{op}, A^{op}/J^{op})=\Hom_A(J, A/J)=0$. So, we conclude that \begin{align*}
		\Hom_A(J, J)\simeq \End_R(\Hom_A(L, A)) \text{ and } \Hom_{A^{op}}(J^{op}, J^{op})\simeq \End_R(L).
	\end{align*} Therefore, $R$ is Morita equivalent to $\End_A(J)^{op}\simeq \End_R(\Hom_A(L, A))^{op}$ which  in turn is Morita equivalent to $\End_R(L)\simeq \End_{A^{op}}(J^{op})^{op}$. So, $J^{op}$ is split heredity in $A^{op}$.
\end{proof} 

\begin{Theorem}\label{spliqhop}
	$A$ is split quasi-hereditary with split heredity chain
	$
	0\subset J_t\subset \cdots\subset J_1=A
	$ if and only if $A^{op}$ is split quasi-hereditary with split heredity chain $
	0\subset J_t^{op}\subset \cdots\subset J_1^{op}=A^{op}.
	$
\end{Theorem}
\begin{proof}
	By Proposition \ref{heredityop}, $J_i/J_{i+1}$ is split heredity in $A/J_{i+1}$ if and only if $(J_i/J_{i+1})^{op}=J_i^{op}/J_{i+1}^{op}$ is split heredity in $(A/J_{i+1})^{op}=A^{op}/J_{i+1}^{op}$ for all $1\leq i\leq t$.
\end{proof}

In the following, we want to obtain further insight into what information about split heredity chains can we gain from applying change of rings techniques on split heredity chains.

\begin{Lemma}\label{splithereditychainequalunderresiduefield}
	Let $R$ be a commutative Noetherian ring and let $A$ be a projective Noetherian $R$-algebra. Assume that $A$ has two split heredity chains\begin{align}
		0\subset J_t\subset J_{t-1}\subset \cdots\subset J_1=A\\
		0\subset I_t\subset I_{t-1}\subset \cdots\subset I_1=A.
	\end{align}
	If $J_j(\mi)=I_j(\mi)$ for every maximal ideal $\mi$ of $R$, then $J_j=I_j$ for all $j$.
\end{Lemma}
\begin{proof}
	Let $J$ and $I$ be split heredity ideals of $A$ satisfying $I(\mi)=J(\mi)$ for every maximal ideal $\mi$ of $R$. Let $\St$ and $L$ be the modules in $\mathcal{M}(A)$ associated with $J$ and $I$, respectively. Let $\mi$ be a maximal ideal of $R$. By Proposition \ref{bijectionsplitheredity}, $\St(\mi)\simeq L(\mi)$. Therefore, we have surjective $A$-maps $\pi_\St \colon \St\rightarrow \St(\mi) $, $\pi_L\colon L\rightarrow \St(\mi)$. In particular, $\pi_\St(\mi)$ and $\pi_L(\mi)$ are $A(\mi)$-isomorphisms. Since $L\in A\proj$ there exists an $A$-homomorphism $f\in \Hom_A(L, \St)$ satisfying $\pi_\St\circ f=\pi_L$. Therefore, $f(\mi)$ is an isomorphism. It follows by Lemma \ref{nakayamalemmasurjectiveproj}, $\St_\mi\simeq L_\mi$.
	The following commutative diagram
	\begin{equation}
		\begin{tikzcd}
			\St_\mi \otimes_{R_\mi} \Hom_{A_\mi}(\St_\mi, A_\mi) \arrow[r, "\tau_{\St_\mi}"] \arrow[d, "f_\mi\otimes \Hom_{A_\mi}(f_\mi^{-1}{,}A_\mi)", swap]& A_\mi \arrow[d, equal] \\
			L_\mi\otimes_{R_\mi} \Hom_{A_\mi}(L_\mi, A_\mi) \arrow[r, "\tau_{L_\mi}"] & A_\mi
		\end{tikzcd}
	\end{equation}
	yields that $I_\mi=\im \tau_{L_\mi} =\im \tau_{\St_\mi}=J_\mi$. The choice of $\mi$ is arbitrary, thus this equality holds for every maximal ideal $\mi$ of $R$. By Lemma \ref{equalitybeinglocal}, $I=J$. As $J_{t-1}/J_t\in R\proj$ we can write
	\begin{align}
		J_{t-1}/J_t(\mi)\simeq J_{t-1}(\mi)/J_t(\mi)\simeq I_{t-1}(\mi)/I_t(\mi)\simeq I_{t-1}/I_t(\mi),
	\end{align}for every maximal ideal $\mi$ of $R$. Thus, we obtain $J_{t-1}/J_t=I_{t-1}/J_t$. It follows that $J_{t-1}=I_{t-1}$. Continuing this argument, by induction on $t$, we conclude the result.
\end{proof}

Another interpretation of Observation \ref{obs563} is the following statement.

\begin{Lemma}Let $R$ be a commutative Noetherian ring.
	Let $A$ be a split quasi-hereditary $R$-algebra and $J$ be a split heredity ideal in $A$.
	Then, for each maximal ideal $\mi$ of $R$, the canonical map \begin{align}
		\Ext_A^1(A/J, J)(\mi)\rightarrow \Ext_{A(\mi)}^1(A(\mi)/J(\mi), J(\mi)) 
	\end{align}is an isomorphism.
\end{Lemma}
\begin{proof}
	Consider the $(A, R)$-exact sequence $0\rightarrow J\rightarrow A\rightarrow A/J\rightarrow 0$. Applying $\Hom_A(-, J)$ and the tensor product $R(\mi)\otimes_R -$ we obtain the commutative diagram with exact rows
	\begin{equation}
		\begin{tikzcd}
			J(\mi) \arrow[r] \arrow[d, equal] & \End_A(J)(\mi)\arrow[r] \arrow[d, "\simeq"] & \Ext_A^1(A/J, J)(\mi) \arrow[r] \arrow[d] & 0\\
			J(\mi) \arrow[r] & \End_{A(\mi)}(J(\mi)) \arrow[r] & \Ext_{A(\mi)}^1(A(\mi)/J(\mi), J(\mi))\arrow[r]& 0
		\end{tikzcd}.
	\end{equation}By diagram chasing, we obtain the result.
\end{proof}
Hence, the extensions between the projective $A/J$-modules and the projective standard module of $A$ commute with functor $R(\mi)\otimes_R -$. 

The following gives a parallel result to Theorem  \ref{hwcresiduefield} now using split heredity chains.

\begin{Theorem}\label{splitqhchainschangeofring}
	Let $R$ be a commutative Noetherian ring. Let $A$ be a projective Noetherian $R$-algebra. Assume that $A$ admits a set of orthogonal idempotents $\{e_1, \ldots, e_t \}$ such that for each maximal ideal $\mi$ of $R$ $\{e_1(\mi), \ldots, e_t(\mi) \}$ becomes a complete set of primitive orthogonal idempotents of $A(\mi)$. Then, $A$ is split quasi-hereditary with split heredity chain \begin{align}
		0\subset Ae_tA\subset \cdots \subset A(e_1+\cdots+ e_t)A=A\label{eqqh53}
	\end{align} if and only if for each maximal ideal $\mi$ of $R$, $A(\mi)$ is split quasi-hereditary with  split heredity chain
	\begin{align}
		0\subset A(\mi)e_t(\mi)A(\mi)\subset \cdots \subset A(\mi)(e_1(\mi)+\cdots+ e_t(\mi))A(\mi)=A(\mi). \label{eqqd54}
	\end{align}
\end{Theorem}
\begin{proof}
	Assume that $A$ is split quasi-hereditary. Let $\mi$ be a maximal ideal of $R$. As $A/Ae_tA\in R\proj$, we can write $A e_tA(\mi)\simeq A(\mi)e_t(\mi)A(\mi)\in A(\mi)\proj$ and $A/Ae_tA(\mi)\simeq A(\mi)/A(\mi)e_t(\mi)A(\mi)$. Also, for an $\End_{A}(Ae_tA)$-progenerator $P$ we can write $$
	R(\mi)\simeq \End_{\End_{A}(Ae_tA)}(P)(\mi)\simeq \End_{\End_{A(\mi)}(A(\mi)e_t(\mi)A(\mi))}(P(\mi)).
	$$ Hence, $A(\mi)e_t(\mi)A(\mi)$ is a split heredity ideal of $A(\mi)$. By going through the split heredity chain of $A$ we obtain that $A(\mi)$ is split quasi-hereditary a with split heredity chain (\ref{eqqd54}).
	
	Conversely, assume that $A(\mi)$ is split quasi-hereditary for every maximal ideal $\mi$ of $R$ with split heredity chain \ref{eqqd54}. By Lemma \ref{standardfromsplitquasi}, $A(\mi)e_t(\mi)\in \mathcal{M}(A(\mi))$ for every maximal ideal $\mi$ of $R$. Since $Ae_t$ is an $A$-summand of $A$, the inclusion $Ae_t\rightarrow A$ remains exact under the functor $R(\mi)\otimes_R -$. So, $Ae_t(\mi)= A(\mi)e_t(\mi)\in \mathcal{M}(A(\mi))$. By Lemma \ref{splitmoduleschangering}, $Ae_t\in \mathcal{M}(A)$. Hence, $AeA=\im \tau_{Ae}$ is a split heredity ideal of $A$. In particular, $A/Ae_tA(\mi)\simeq A(\mi)/ A(\mi)e_t(\mi)A(\mi)$. Continuing the same argument with $A/Ae_tA$ we obtain that $A$ is split quasi-hereditary with  split hereditary chain (\ref{eqqh53}).
\end{proof}

\subsection{Split quasi-hereditary algebras and the existence of projective covers} \label{Split quasi-hereditary algebras and existence of projective covers}
Recall that a ring $A$ is called \textbf{semi-perfect} if every  finitely generated left $A$-module has a projective cover.
We call an $R$-algebra $A$ \textbf{locally semi-perfect} if the localization $A_\pri$ is semi-perfect for every prime ideal $\pri$ in $R$.
\begin{Theorem} \label{qhlocalissemiperfectring}
	Every split quasi-hereditary algebra over a local commutative Noetherian ring is semi-perfect.
\end{Theorem}
\begin{proof}
	According to Proposition \ref{projisindecomposable}, we can choose $P(\l)$ in \ref{splithwc}(iv) so that $ \End_A(P(\l))$ is local. Hence, $\sumSt P(\l)$ is a direct sum of modules with local endomorphism rings. Let $_{A}A\simeq Q_0\bigoplus \cdots \bigoplus Q_t$ be a decomposition into indecomposable $A$-modules of regular module $A$. It follows by  Definition  \ref{splithwc} and Theorem \ref{quasihereditaryhwc} that $\sumSt P(\l)$ is an $A$-progenerator. Thus, there is $K\in A\m$ such that \begin{align}
		\left( \sumSt P(\l)\right) ^t\simeq A\bigoplus K.
	\end{align} By Krull-Schmidt-Remak-Azumaya Theorem \citep[Theorem 2.12]{MR3025306} any two direct sum decompositions into indecomposable modules of $\left( \sumSt P(\l)\right) ^t$ are isomorphic. Hence, every $Q_i$ is isomorphic to a projective indecomposable module $P(\l_i)$. Hence, $_{A}A$ is a finite direct sum of $A$-modules with local endomorphism rings. By Theorem \ref{spliqhop}, $A^{op}$ is split quasi-hereditary over $R$, thus  by this discussion $A_{A}$ is a finite direct sum of $A$-modules with local endomorphism rings. By \citep[Proposition 3.14]{MR3025306}, $A\simeq \End_A(A_{A})$ is a semi-perfect ring.
\end{proof}

We observe that as a consequence of Theorem \ref{qhlocalissemiperfectring}, for any $\l\in \L$, we can choose $P(\l)$ so that $P(\l)$ is the projective cover of $\St(\l)$, when $R$ is a local commutative Noetherian  ring.

In fact, assume that $R$ is a local commutative Noetherian  ring. By Theorem  \ref{qhlocalissemiperfectring}, there exists a projective cover $Q$ of $\St(\l)$. Using the surjective homomorphism $\pi_\l\colon P(\l)\rightarrow \St(\l)$ given by \ref{splithwc}(iv) with $P(\l)$ having a local endomorphism ring, it follows that $Q$ is an $A$-summand of $P(\l)$. As $P(\l)$ is indecomposable $P(\l)\simeq Q$. By Nakayama's Lemma, we deduce that $(P(\l), \pi_\l)$ is the projective cover of $\St(\l)$.

\subsection{Global dimension of split quasi-hereditary algebras} \label{Global dimension of (split) quasi-hereditary algebras}
We will now show that split quasi-hereditary algebras over a commutative Noetherian ring have finite global dimension.

\begin{Lemma}
	Let $\cdots \rightarrow P_2\rightarrow P_1\xrightarrow{\alpha_1} P_0\xrightarrow{\alpha_0} M\rightarrow 0$ be a projective $A$-resolution. Define $N=\ker \alpha_{k-1}$. Then, $\pdim_A M\leq k+\pdim_A N$. \label{pd of kernel}
\end{Lemma}
\begin{proof}
	First notice that $\Ext_A^l(M, L)\simeq \Ext_A^{l-k}(\im \alpha_k, L)=\Ext_A^{l-k}(N, L)$ for any $l\geq 0$ and $L\in A\m$. 
	
	If $\pdim_AN<\infty$, then there is nothing to show. Assume $\pdim_AN=s<\infty$. Then, \begin{align*}
		\Ext_A^{s+k+1}(M, L)\simeq \Ext_A^{s+1}(\im \alpha_k, L)=\Ext_A^{s+1}(N, L) = 0, \ \forall L\in A\m.
	\end{align*}
	Hence, $\pdim_AM\leq s+k=\pdim_AN + k$.
\end{proof}

\begin{Theorem}\label{qhglobaldimensionnoetherian}
	Let $A$ be a split quasi-hereditary algebra over a Noetherian commutative ring $R$ with split heredity chain $0=J_{t+1}\subset J_t\subset \cdots \subset J_1=A$. Then, $\gldim A\leq 2(t-1)+\gldim R$.
\end{Theorem}
\begin{proof}
	Consider $M\in A\m \cap R\proj$. By Theorems \ref{quasihereditaryhwc} and \ref{hwcresiduefield},  $A(\mi)$ is split quasi-hereditary with split heredity chain of size $t$ for any maximal ideal $\mi$ of $R$.  By \citep[Statement 9]{Dlab1989d}, it follows that $\gldim A(\mi)\leq 2(t-1)$. Consider a projective $A$-resolution for $M$, $\cdots \rightarrow P_2\rightarrow P_1\xrightarrow{\alpha_1} P_0\xrightarrow{\alpha_0} M\rightarrow 0$. Let $N=\ker \alpha_{2(t-1)-1}$. Since $M$ is projective over  $R$ and $P_i$ is projective over $A$ we obtain $\im \alpha_i\in R\proj$ and this exact sequence is split as sequence of $R$-modules. In particular, $N\in R\proj$. Applying $-\otimes_R R(\mi)$ we obtain the exact sequence
	\begin{align}
		0\rightarrow N(\mi)\rightarrow P_{2(t-1)-1}(\mi)\rightarrow\cdots\rightarrow P_1(\mi)\rightarrow P_0(\mi)\rightarrow M(\mi)\rightarrow 0.
	\end{align}
	Since $\pdim_{A(\mi)}M(\mi)\leq 2(t-1)$, $N(\mi)$ is projective over $A(\mi)$ for every maximal ideal $\mi$ of $R$. Therefore, $N$ is projective $A$-module. Hence, $\pdim M\leq 2(t-1)$.
	
	Now consider $M$ an arbitrary module in $A\m$. Consider a projective $A$-resolution for $M$, $\cdots \rightarrow P_2\rightarrow P_1\xrightarrow{\alpha_1} P_0\xrightarrow{\alpha_0} M\rightarrow 0$. Define $K=\ker \alpha_{r-1}$, $r=\gldim R$. Then, we must have that $K$ is projective over $R$ as $\pdim_R M\leq r$ and all $P_i$ are projective over $R$. As we have seen $\pdim_A K\leq 2(t-1)$. By Lemma \ref{pd of kernel}, it follows that $\pdim_AM\leq 2(t-1)+r$. Therefore, $\gldim A\leq 2(t-1)+\gldim R$.
\end{proof}

It follows that if $R$ is a regular ring with finite Krull dimension, then a split quasi-hereditary algebra over $R$ has finite global dimension. Of course, this can fail if $R$ has infinite global dimension. In such a case, we just need to consider $A=R$. 
For rings $R$ with finite global dimension, we can give a precise value of the global dimension of a split quasi-hereditary algebra in terms of the global dimension of the finite-dimensional algebras $A(\mi)$.

The following proof is inspired on \citep[3]{zbMATH00966941}.
\begin{Theorem}Let $R$ be a commutative Noetherian ring with finite global dimension.\label{globaldimensionqh}
	Let $A$ be a split quasi-hereditary $R$-algebra. Then,
	$$\gldim A=\dim R+\max\{\gldim A(\mi)\colon \mi \in \MaxSpec R \}.$$
\end{Theorem}
\begin{proof}We can assume that $R$ is a local commutative Noetherian ring with unique maximal ideal $\mi$.
	By the proof of Theorem \ref{qhglobaldimensionnoetherian}, we obtain $\gldim A\leq\dim R+\sup\{\gldim A(\mi)\colon \mi \in \MaxSpec R \}$. Consider the surjective map $A\rightarrow A(\mi)$. Let $M\in A\m$. By Theorem 10.75 of \cite{Rotman2009a}, we can consider the spectral sequence
	\begin{align}
		E_2^{i, j}=\Ext_{A(\mi)}^i(M, \Ext_A^j(A(\mi), A))\Rightarrow \Ext_A^{i+j}(M, A).
	\end{align} As $A\in R\proj$, we can write
	\begin{align}
		\Ext_A^j(A(\mi), A)\simeq \Ext_{A\otimes_R R}^j(R(\mi)\otimes_R A, A\otimes_R R)\simeq A\otimes_R \Ext_R^j(R(\mi), R), \ \forall j\geq 0.
	\end{align}
	Since $\pdim_R R(\mi)=\gldim R=\dim R$ we obtain that $\Ext_{R}^{\dim R}(R(\mi), R)\neq 0$. Since $A$ is faithful as $R$-module we obtain that $\Ext_A^{\dim R}(A(\mi), A)\neq 0$. 
	Pick $M=DA(\mi)$ regarded as $A$-module. Then, $\gldim A(\mi)=\pdim_{A(\mi)} DA(\mi)$ and denote by $n$ the value $\gldim A(\mi)$. We claim that $E_2^{n, \dim R}\neq 0$. In fact,
	we can see by induction that $E_k^{n, \dim R}=E_2^{n, \dim R}$ for all $k\geq 2$. Since $\Ext_R^{\dim R}(R(\mi), R)\in R(\mi)\m$  we obtain that $\Ext_A^{\dim R}(A(\mi), A)\in A(\mi)\proj$. Therefore, \begin{align}
		E_2^{n, \dim R}=\Ext_{A(\mi)}^n(DA(\mi), \Ext_A^{\dim R}(A(\mi), A))\neq 0.
	\end{align} Hence, $E_{\infty}^{n, \dim R}\neq0$. So, $\Ext_A^{n+\dim R}(DA(\mi), A)\neq 0$.  
By Theorem \ref{qhglobaldimensionnoetherian}, $\gldim A(\mi)$ is bounded above by the finite value $2(t-1)+\dim R$, where $t$ denotes the size of the split heredity chain of $A$. So, $\max\{\gldim A(\mi)\colon \mi \in \MaxSpec R \}=\sup\{\gldim A(\mi)\colon \mi \in \MaxSpec R \}.$
\end{proof}

Using the next Lemma, we can show that a split quasi-hereditary algebra has finite global dimension if and only if the ground ring has finite global dimension.

\begin{Lemma}\label{extheredityideal}
	Let $A$ be projective Noetherian $R$-algebra and let $J$ be a split heredity ideal in $A$. Then, \linebreak\mbox{$\Ext_A^i(M, N)\simeq \Ext_{A/J}^i(M, N)$} for all $M, N\in A/J\m$ and $i\geq 0$. 
\end{Lemma}
\begin{proof}
	For $i=0$, the result is clear since $A/J\m$ is a full subcategory of $A\m$. Consider the exact sequence\begin{align}
		0\rightarrow J\rightarrow A\rightarrow A/J\rightarrow 0. \label{exc1}
	\end{align} For any $A/J$-module $N$, we deduce $\Ext_A^{i-1}(J, N)\simeq \Ext_A^i(A/J, N)$ for $i\geq 2$ by applying the functor $\Hom_A(-, N)$ on $(\ref{exc1})$. $J$ is projective over $A$, thus $\Ext_A^i(A/J, N)=0$ for $i\geq 2$. Furthermore, by the same argument, the induced map $\Hom_A(J, N)\rightarrow \Ext_A^1(A/J, N)$ is surjective. By Lemma \ref{A/Jmod}, $\Hom_A(J, N)=0$. This implies that $\Ext_A^1(A/J, N)$ also vanishes. So, we conclude that free $A/J$-modules are acyclic for the functor $\Hom_A(-, N)$, for every $N\in A/J\m$. Thus, we can use $A/J$-free resolutions of $M\in A/J\m$ to compute $\Ext_A^{i\geq 0}(M, N)$. Moreover, let $M^{\bullet}$ be an $A/J$-free resolution of $M$ then using the fact that $A/J$ is a full subcategory of $A\m$ we conclude
	$
	\Ext_A^i(M, N)=H^i(\Hom_A(M^{\bullet}, N))=H^i(\Hom_{A/J}(M^{\bullet}, N))=\Ext_{A/J}^i(M, N), \ i\geq 0. 
	$
\end{proof}

\begin{Prop}
	Let $A$ be a split quasi-hereditary algebra over a commutative Noetherian ring $R$ with split heredity chain $0=J_{t+1}\subset J_t\subset \cdots \subset J_1=A$. If $\gldim A<+\infty$, then $\gldim R<+\infty$. Moreover, $R$ is a regular ring with finite Krull dimension.
\end{Prop}
\begin{proof}
	By definition, $A/J_2=J_1/J_2$ is a split heredity ideal of $A/J_2$. Therefore, $\End_{A/J_2}(A/J_2)^{op}\simeq A/J_2$ is Morita equivalent to $R$.  By induction on the split heredity chain together with Lemma \ref{extheredityideal} it follows that $\gldim A/J_2\leq \gldim A$. Thus, the result follows for $R$.
\end{proof}
	
If $R$ has infinite global dimension, the previous arguments also provide an upper bound  to the finitistic dimension of a split quasi-hereditary algebra. Moreover, the next result generalises Theorem 2 of \cite{zbMATH02164791} since every Artinian commutative ring has Krull dimension zero.

\begin{Theorem}Let $R$ be a commutative Noetherian ring.
	Let $A$ be a split quasi-hereditary $R$-algebra. Then, \label{finitisticqh}
\begin{align*}
\sup\{\gldim A(\mi)\colon \mi \in \MaxSpec R \} &\leq \findim A\\ &\leq \findim R+\sup\{\gldim A(\mi)\colon \mi \in \MaxSpec R \} \\&\leq \dim R +\sup\{\gldim A(\mi)\colon \mi \in \MaxSpec R \}.
\end{align*}	
\end{Theorem}
\begin{proof}
	Observe that \begin{align*}
		\pdim_A DA = \sup\{\pdim_{A(\mi)} DA(\mi)\colon \mi \in \MaxSpec R \} = \sup\{\gldim A(\mi)\colon \mi \in \MaxSpec R \}<+\infty
	\end{align*}since $DA\in A\m\cap R\proj$.
	
	Let $M$ be a left $A$-module with finite projective dimension. Since every projective $A$-module is also projective over $R$, the module $M$ viewed as $R$-module has finite projective dimensionConsider a projective $A$-resolution for $M$, $\cdots \rightarrow P_2\rightarrow P_1\xrightarrow{\alpha_1} P_0\xrightarrow{\alpha_0} M\rightarrow 0$. So we can write an exact sequence of $A$-modules
	\begin{align*}
		0\rightarrow X_{\findim R}\rightarrow P_{\findim R-1}\rightarrow \cdots \rightarrow P_1\rightarrow P_0\rightarrow M\rightarrow 0
	\end{align*}for some left $A$-module $X_{\findim R}$. Since every $P_i\in R\proj$ we obtain that $X_{\findim R}\in A\m\cap R\proj$. By the proof of Theorem \ref{qhglobaldimensionnoetherian}, we obtain that $\pdim_A X_{\findim R}\leq \sup\{\gldim A(\mi)\colon \mi \in \MaxSpec R \}$. So, we deduce that $\pdim_AM \leq \findim R+\sup\{\gldim A(\mi)\colon \mi \in \MaxSpec R \}$. By Theorem 1.6 of \cite{zbMATH03116683}, $\findim R\leq \dim R$. Hence, the result follows.
\end{proof}

Note that although we are working with split quasi-hereditary algebras these arguments about global dimensions and finitistic dimensions remain valid with integral quasi-hereditary algebras (not necessarily split).

	\section{Integral Cellular algebras of finite global dimension} \label{Cellular algebras} 

	The following proposition gives a criterion to check if a split quasi-hereditary algebra is cellular. This is a generalisation of Corollary 4.2 of \cite{zbMATH01218863} to commutative Noetherian rings.
	
	\begin{Prop}\label{splitqhwithdualityiscellular}
		Let $R$ be a commutative Noetherian ring. Let $A$ be a free Noetherian $R$-algebra. Assume that $A$ admits a set of orthogonal idempotents $\{e_1, \ldots, e_t \}$ such that for each maximal ideal $\mi$ of $R$ $\{e_1(\mi), \ldots, e_t(\mi) \}$ becomes a complete set of primitive orthogonal idempotents of $A(\mi)$. Suppose that there exists an involution $\iota\colon A\rightarrow A$ that fixes the set of orthogonal idempotents $\{e_1, \ldots, e_t \}$. If $A$ is a split quasi-hereditary with split heredity chain \begin{align}
			0\subset Ae_tA\subset \cdots \subset A(e_1+\cdots+ e_t)A=A, \label{celeq25}
		\end{align} then $A$ is a cellular algebra (with respect to $\iota$)  and with cell chain (\ref{celeq25}).
	\end{Prop}
	\begin{proof}
		Put $e=e_t$. Thus, $\iota(AeA)=A\iota(e)A=AeA$.
		By Theorem \ref{splitqhchainschangeofring}, $Ae\in \mathcal{M}(A)$. Moreover, $\Hom_A(Ae, A)=eA=\iota(e)A=\iota(Ae)$. So, the map $\tau_{Ae}\colon Ae\otimes_R \iota(Ae)\rightarrow AeA$ is an isomorphism. We can consider the diagram
		\begin{equation}
			\begin{tikzcd}
				AeA\arrow[r, "\tau_{Ae}^{-1}"] \arrow[d, "\iota"] &  Ae\otimes_R \iota(Ae) \arrow[r, "\tau_{Ae}"] \arrow[d, "\omega"] & AeA \arrow[d, "\iota"]\\  
				AeA \arrow[r, "\tau_{Ae}^{-1}"] & Ae\otimes_R \iota(Ae)\arrow[r, "\tau_{Ae}"] & AeA
			\end{tikzcd}, 
		\end{equation} where $\omega$ is the usual twist map.
		We claim that the diagram is commutative. To show that, note that
		\begin{align}
			\iota \tau_{Ae}(ae\otimes eb)=\iota(aeb)=\iota(b)e\iota(a)\\
			\tau_{Ae}\omega(ae\otimes eb)=\tau_{Ae}(\iota(eb)\otimes \iota(ae))=\tau_{Ae}(\iota(b)e\otimes e\iota(a))=\iota(b)e\iota(a).
		\end{align}It follows that  \begin{align}
			\tau_{Ae}\omega\tau_{Ae}^{-1}=\iota\tau_{Ae}\tau_{Ae}^{-1}=\iota.
		\end{align} Thus, all interior squares of the diagram are commutative. In particular, $AeA$ is a cell ideal. Proceeding by induction on the heredity chain, we get that (\ref{celeq25}) is a cell chain.
	\end{proof}
	
	We note that if $A$ is split quasi-hereditary with a poset $\L$, $\L$ indexes the cell basis of $A$ but with the reversed order.

	Proposition \ref{splitqhwithdualityiscellular} motivates the following definition of duality for Noetherian algebras.
	
	\begin{Def}
		Let $R$ be a commutative Noetherian ring. Let $A$ be a free Noetherian $R$-algebra. Assume that $A$ admits a set of orthogonal idempotents $\textbf{e}:=\{e_1, \ldots, e_t \}$ such that for each maximal ideal $\mi$ of $R$ $\{e_1(\mi), \ldots, e_t(\mi) \}$ becomes a complete set of primitive orthogonal idempotents of $A(\mi)$. We say that $A$ has a \textbf{duality} $\iota\colon A\rightarrow A$ (with respect to $\textbf{e}$) if  $\iota$ is an anti-isomorphism with $\iota^2=\id_A$ fixing the set of orthogonal idempotents $\{e_1, \ldots, e_t \}$. 
	\end{Def}
Here, a duality is a special case of an involution, which depends on the choice of set of orthogonal idempotents.

	Our next goal is to show a positive answer to Question \ref{IQ1}. The main idea is to show that for a cellular algebra $A$ the simple $A(\mi)$-modules arise from a finitely generated $A$-module which is projective over the ground ring.
	
	To facilitate our life, we will require further notation first. Let $A$ be a cellular algebra over a commutative Noetherian ring $R$. Denote by $A(\leq \l)$ the $A$-submodule of $A$ with $R$-basis $\{C_{S, T}^\mu \colon \mu\leq \l, \ S, T\in M(\mu) \}$ for $\l\in \L$. Denote by $A(<\l)$ the $A$-module with $R$-basis $\{C_{S, T}^\mu \colon \mu<\l, \ S, T\in M(\mu) \}$. In this notation, $A/A(<\l)$ is cellular and $A(\leq \l)/A(<\l)$ is a cell ideal of $A/A(<\l)$.
	
	Using Lemma 1.7 of \citep{zbMATH00871761}, we can define a bilinear form $\phi_\l\colon \theta(\l)\times \theta(\l)\rightarrow R$ by $\phi_\l(C_{U, T_0}^\l, C_{T, T_0}^\l)=\phi_{1_A}(U, T)$ where
	\begin{align}
		C_{U_1, T_1}^\l a C_{U_2, T_2}^\l -\phi_a(T_1, U_2) C_{U_1, T_2}^\l \in A(<\l), \quad U_1, T_1, U_2, T_2\in M(\l). \label{eqcelleq30}
	\end{align} Let $S$ be a commutative Noetherian ring which is an $R$-algebra. $S\otimes_R A$ is cellular $S$-algebra. So, associated with $S\otimes_R \theta(\l)$ there is a bilinear form $\phi_\l^S$. We shall relate the bilinear form $\phi_\l^S$ with $\phi_\l$. %
	
	By considering the maps that carry the basis of $(S\otimes_R A)(<\l)$ (resp. $(S\otimes_R A)(\leq\l)$ ) to $S\otimes_R (A(<\l))$ (resp. $S\otimes_R (A(\leq \l))$ ) we obtain $S\otimes_R A$-isomorphjsms
	\begin{align}
		(S\otimes_R A)(<\l)\simeq S\otimes_R (A(<\l)), \quad (S\otimes_R A)(\leq\l)\simeq S\otimes_R (A(\leq \l)). \label{celleq31}
	\end{align}
	Now observe that,
	\begin{align}
		C_{U_1, T_1}^\l a C_{U_2, T_2}^\l - \phi_a(T_1, U_2)C_{U_1, T_2}^\l\in A(<\l), \quad U_1, T_1, U_2, T_2\in M(\l).
	\end{align}So, for every $s\in S$,
	\begin{align}
		s\otimes (C_{U_1, T_1}^\l a C_{U_2, T_2}^\l - \phi_a(T_1, U_2)C_{U_1, T_2}^\l) \in S\otimes_R A(<\l), \quad U_1, T_1, U_2, T_2\in M(\l).
	\end{align}Under the isomorphism (\ref{celleq31}), we obtain that
	\begin{align}
		(1_S\otimes C_{U_1, T_1}^\l )(s\otimes a)(1_S\otimes C_{U_2, T_2}^\l)-\phi_a(T_1, U_2)s(1_S\otimes C_{U_1, T_2}^\l)\in (S\otimes_R A)(<\l), \quad U_1, T_1, U_2, T_2\in M(\l).
	\end{align}
	On the other hand, applying (\ref{eqcelleq30}) to $S$ and $s\otimes a$ we obtain that
	\begin{align}
		(1_S\otimes C_{U_1, T_1}^\l )(s\otimes a)(1_S\otimes C_{U_2, T_2}^\l)-\phi_{s\otimes a}^S(T_1, U_2)s(1_S\otimes C_{U_1, T_2}^\l)\in (S\otimes_R A)(<\l), \quad U_1, T_1, U_2, T_2\in M(\l).
	\end{align}
	Thus, by comparing basis, $\phi_a(T_1, U_2)s=\phi_{s\otimes a}^S(T_1, U_2)$, $T_1, U_2\in M(\l)$. In particular, $\phi_{1_A}(T_1, U_2)1_S=\phi_{1_{S\otimes_R A}}^S(T_1, U_2)$. We have shown that \begin{Lemma}For $\phi_\l$ and $\phi_\l^S$ the bilinear forms associated with $\theta(\l)$ and $S\otimes_R \theta(\l)$, respectively, we can write
		\begin{align}
			\phi_\l^S(1_S\otimes C_{U, T_0}^\l, 1_S\otimes C_{T, T_0}^\l)=\phi_\l(C_{U, T_0}^\l, C_{T, T_0}^\l)1_S, \label{eqcell36} \quad U, T\in M(\l).
		\end{align}
	\end{Lemma} 
	
	We can now construct modules in $A\m\cap R\proj$ that over the finite-dimensional $A(\mi)$ become simple modules as long as $\phi_\l^{R(\mi)}\neq 0$.
	
	\begin{Lemma}\label{radicalcell}
		Let $R$ be a local commutative Noetherian ring with maximal ideal $\mi$. Let $A$ be a cellular $R$-algebra with cell datum $(\L, M, C, \iota)$. For each $\l\in \L$, define
		\begin{align}
			\rad (\phi_\l)=\{x\in \theta(\l) \  | \ \phi_\l(x, y)\in \mi, \quad \forall y\in \theta(\l) \}.
		\end{align}
		Then, for each $\l\in \L$, there exists an $(A, R)$-exact sequence
		\begin{align}
			0\rightarrow \rad(\phi_\l)\rightarrow \theta(\l)\rightarrow X_\l\rightarrow 0.
		\end{align}
	\end{Lemma}
	\begin{proof}Let $\l\in \L$. 
		We start by observing that $\rad(\phi_\l)$ is an $A$-module.  Since $\phi_\l$ is a bilinear form, it follows that $\rad (\phi_\l)$ is an $R$-submodule of $\theta$.  Let $a\in A$, $x\in \rad(\phi_\l)$.
		By Proposition 2.4 of \citep{zbMATH00871761}, 
		\begin{align}
			\phi_\l(ax, y)=\phi_y(x, \iota(a)y)\in \mi, \ \forall y\in \theta(\l).
		\end{align}Hence, $ax\in \rad(\phi_\l)$.
		We claim now that $\rad(\phi_\l^{R(\mi)})=\rad(\phi_\l)(\mi)$. Suppose, again that $x\in \rad(\phi_\l)$. We can write $x=\sum_{V\in M(\l)} x_V C_{V, T_0}^\l$. 
		By definition, 
		\begin{align}
			\sum_{V\in M(\l)} x_V \phi_\l(C_{V, T_0}^\l, C_{T, T_0}^\l)=\phi_\l(x, C_{T, T_0}^\l)\in \mi, \ \forall T\in M(\l).
		\end{align}
		Therefore, 
		\begin{align}
			0=\sum_{V\in M(\l)} x_V\phi_\l(C_{V, T_0}^\l, C_{T, T_0}^\l)1_{R(\mi)}=\sum_{V\in M(\l)} x_V\phi_\l^{R(\mi)}(1_{R(\mi)}\otimes C_{V, T_0}^\l,1_{R(\mi)}\otimes C_{T, T_0}^\l)
			\\
			=\phi_\l^{R(\mi)}(1_{R(\mi)} \otimes x, 1_{R(\mi)} \otimes C_{T, T_0}^\l), \forall T\in M(\l).
		\end{align}Hence, $1_{R(\mi)} \otimes x\in \rad(\phi_\l^{R(\mi)})$. So, $\rad(\phi_\l)(\mi)\subset \rad(\phi_\l^{R(\mi)})$. Now consider $y\in \rad(\phi_\l^{R(\mi)})\subset \theta(\l)(\mi)$. So, we can write $y=\sum_{U\in M(\l)} y_U 1_{R(\mi)}\otimes C_{U, T_0}^\l$, with $y_U\in R(\mi)$. Further, we can assume that $y_U=r_U1_{R(\mi)}$ for some $r_U\in R$.  For every $T\in M(\l)$,
		\begin{align}
			0= \phi_\l^{R(\mi)}(y, 1_{R(\mi)}\otimes C_{T, T_0}^\l)&=\sum_{U\in M(\l)} r_U \phi_\l^{R(\mi)}(1_{R(\mi)}\otimes C_{U, T_0}^\l, 1_{R(\mi)}\otimes C_{T, T_0}^\l) \\
			&=\sum_{U\in M(\l)} r_U\phi_\l(C_{U, T_0}^\l, C_{T, T_0}^\l)1_{R(\mi)}.
		\end{align}Thus,
		\begin{align}
			\phi_\l(\sum_{U\in M(\l)} r_U C_{U, T_0}^\l, C_{T, T_0}^\l)\in \mi, \forall T\in M(\l).
		\end{align}It follows that $\sum_{U\in M(\l)} r_U C_{U, T_0}^\l\in \rad(\phi_\l)$. Hence, $y=1_{R(\mi)} \otimes_R \sum_{U\in M(\l)} r_U C_{U, T_0}^\l\in \rad(\phi_\l)(\mi)$. This completes our claim.
		
		Let $X_\l$ be cokernel of $\rad(\phi_\l)\rightarrow \theta(\l)$. Applying the functor $R(\mi)\otimes_R -$ yields the long exact sequence
		\begin{align}
			0=\Tor_1^R(\theta(\l), R(\mi))\rightarrow \Tor_1^R(X_\l, R(\mi))\rightarrow \rad(\phi_\l)(\mi)\rightarrow \theta(\l)(\mi)\rightarrow X_\l\rightarrow 0.
		\end{align}Since $\rad(\phi_\l)(\mi)=\rad(\phi_\l^{R(\mi)})\subset \theta(\l)(\mi)$, $\Tor_1^R(X_\l, R(\mi))=0$. So, $X_\l\in R\proj$. So, the exact sequence
		\begin{align}
			0\rightarrow \rad(\phi_\l)\rightarrow \theta(\l)\rightarrow X_\l\rightarrow 0
		\end{align}is $(A, R)$-exact.
	\end{proof}
	
	\begin{Theorem}\label{cellularfiniteglobaldimension}
		Let $R$ be a commutative Noetherian regular ring with finite Krull dimension. Let $A$ be a cellular $R$-algebra with cell datum $(\L, M, C, \iota)$. Then, $(A, \theta_{\l\in \L^{op}})$, with $\L^{op}$ being the poset $\L$ with reversed order, is a split quasi-hereditary algebra if and only if $A$ has finite global dimension.
	\end{Theorem}
	\begin{proof}
		By Theorem \ref{qhglobaldimensionnoetherian}, if $(A, \{\theta_{\l\in \L} \})$ is split quasi-hereditary, then $A$ has finite global dimension. Conversely, assume that $A$ has finite global dimension. Let $\mi$ be a maximal ideal of $R$.  Since localization is a dense functor every module in $A_\mi$ can be written in the form $M_\mi$ for some $M\in A\m$. Thus,
		\begin{align}
			\Ext_{A_\mi}^{\gldim A+1}(X_\mi, Y_\mi)=\Ext_A^{\gldim A+1}(X, Y)_\mi=0.
		\end{align}Thus, $\gldim A_\mi\leq \gldim A$. In view of Theorem \ref{hwclocalization}, we can assume that $R$ is a local regular commutative Noetherian ring. Let $L$ be a simple $A(\mi)$-module. By Propositions 3.2 and 3.4 of \citep{zbMATH00871761}, there exists $\l\in \L$ such that $\phi_\l^{R(\mi)}\neq 0$ and $\theta(\l)(\mi)/\rad (\phi_\l^{R(\mi)})\simeq L$. By Lemma \ref{radicalcell}, $X_\l(\mi)\simeq L$. By assumption, $\pdim_A X_\l$ is finite. Since $X_\l\in R\proj$, any projective $A$-resolution of $X_\l$ remains exact under $R(\mi)\otimes_R -$. In particular, $\pdim_{A(\mi)} L$ is finite. It follows that $A(\mi)$ has finite global dimension. By Theorem 1.1 of \citep{zbMATH01384521}, $(A(\mi), \theta(\mi)_{\l\in \L})$ is split quasi-hereditary. By Theorem \ref{hwcresiduefield}, $(A, \theta_{\l\in \L})$ is a split quasi-hereditary algebra.
	\end{proof}
	
	\begin{Remark}
		Every commutative algebra with finite global dimension over an algebraically closed field is a split quasi-hereditary algebra (see Proposition 3.5 of \citep{zbMATH01218863}).
	\end{Remark}
	
	We wish to proceed further and give a complete characterization for cellular Noetherian algebras in a similar form as in \citep{zbMATH01384521}.
	
	\begin{Theorem}\label{cellularfiniteglobaldimensiontwo}
		Let $R$ be a regular commutative Noetherian ring with finite Krull dimension. Let $A$ be a cellular $R$-algebra with cell datum  $(\L, M, C, \iota)$. The following assertions are equivalent.
		\begin{enumerate}[(i)]
			\item Some cell chain of $A$ is a split heredity chain as well, that is, $A$ is split quasi-hereditary.
			\item There is a cell chain (with respect to some involution possibly distinct from $\iota$) whose length $|\L|$ equals the number of simple $A(\pri)$-modules for every prime ideal $\pri$ of $R$.
			\item Any cell chain of $A$ is a split heredity chain of length $|\L|$.
			\item The algebra $A$ has finite global dimension.
			\item $A$ is locally semi-perfect and the function
			$\operatorname{Cartan}\colon \Spec R\rightarrow \mathbb{Z}$, given by $$\operatorname{Cartan}(\pri)=\det [\rank_{R_\pri} \Hom_{A_\pri}(P_i, P_j) ], \ \pri\in \Spec R,$$ is the constant function 1, where $P_i$, $i=1, \ldots, r$ for some natural number $r$, are the projective indecomposable modules of $A_\pri$.
		\end{enumerate}
	\end{Theorem}
	\begin{proof}
		By Proposition 4.1 of \citep{zbMATH01218863}, given a cell ideal $J$ if $J^2\neq 0$, then $J=AeA$ is a split heredity ideal and $Ae=\theta$, for some primitive idempotent. Hence, $\theta_\l$ are the standard modules of $A$ if $A$ split is quasi-hereditary. In particular, for split quasi-hereditary algebras all split heredity chains have the same size. Together with Theorem \ref{cellularfiniteglobaldimension}, this shows that (iii)$\Leftrightarrow$ (iv) $\Leftrightarrow$ (i).
		Assume that (iv) holds. Let $\pri$ be a prime ideal of $R$. Then, $A(\pri)$ is a cellular algebra with cell datum  $(\L, M, C, \iota)$.  In particular, $A(\pri)$ has a cell chain (given by the cell datum) of length $\L$ and $A(\pri)$ is split quasi-hereditary by Theorem \ref{cellularfiniteglobaldimension}. Therefore, $|\L|$ is equal to the number of standard modules of $A(\pri)$ which is equal to the number of simple $A(\pri)$-modules. So, (ii) holds. 
		
		Assume that (ii) holds. For every prime ideal, $A(\pri)$ has a cell chain whose length equals the number of simple $A(\pri)$-modules. Thus, $A(\pri)$ is split quasi-hereditary with standard modules $\theta_\l(\pri)$, $\l\in \L$ by Theorem 1.1 of \citep{zbMATH01384521}. Therefore, $(A, \theta_{\l\in \L^{op}})$ is split quasi-hereditary. So, (i) holds. Assume that (i) holds. By Theorem \ref{qhlocalissemiperfectring}, $A_\pri$ is semi-perfect  for every prime ideal $\pri$ of $R$. Thus, $A$ is locally semi-perfect. So, we can write $A_\pri$ as a direct sum of unique indecomposable projective module. Moreover, $\Hom_{A_\pri}(P_i, P_j)$ is free over $R_\pri$. Further, $R(\pri)\otimes_{R_\pri} \Hom_{A_\pri}(P_i, P_j)\simeq \Hom_{A(\pri)}(P_i(\pri), P_j(\pri))$ and $P_i(\pri)$ are the indecomposable projective modules of $A(\pri)$. By Theorem 1.1 of \citep{zbMATH01384521},
		\begin{align*}
			1=\det [\dim_{R(\pri)} \Hom_{A(\pri)}(P_i(\pri), P_j(\pri)) ]=\det [\dim_{R(\pri)} R(\pri)\otimes_{R_\pri} \Hom_{A_\pri}(P_i, P_j) ]=\det [\rank \Hom_{A_\pri}(P_i, P_j) ].
		\end{align*}So, (v) holds.
		Finally, assume that (v) holds. Let $\pri$ be a prime ideal of $R$. Applying $R(\pri)\otimes_R -$ we obtain $A(\pri)$ is a direct sum of the projective modules $P_i(\pri)$ with $i=1, \ldots, r$, and
		\begin{align}
			1=\det[\rank_{R_\pri} \Hom_{A_\pri}(P_i, P_j)]=\det[\dim_{R(\pri)} \Hom_{A(\pri)}(P_i(\pri), P_j(\pri)) ]. \label{eqcel50}
		\end{align}Moreover, every map between $P_i(\pri)$ and $P_j(\pri)$ can be lifted to a map between $P_i$ and $P_j$. Since each $P_j\in R\proj$ and by Lemma \ref{nakayamalemmasurjectiveproj}, $P_i(\pri)\simeq P_j(\pri)$ if and only if $P_i\simeq P_j$ if and only if $i=j$. We claim now that each $P_i(\pri)$ is indecomposable over $A(\pri)$.
		Since $A_\pri$ is semi-perfect, $\End_{A_\pri}(P_i)$ is a local ring. Furthermore,
		$\pri_\pri\End_{A_\pri}(P_i)$ is an ideal of $\End_{A_\pri}(P_i)$ and \begin{align}
			\End_{\widehat{A_\pri}}(\widehat{P_i})\simeq \widehat{\End_{A_\pri}(P_i)} = \lim_n \End_{A_\pri}(P_i)/\pri_\pri^n \End_{A_\pri}(P_i) = \lim_n \End_{A_\pri}(P_i)/\left( \pri_\pri\End_A(P_i) \right)^n. 
		\end{align}This last ring is the completion of $\End_{A_\pri}(P_i)$ at the ideal $\pri_\pri\End_{A_\pri}(P_i)$, so it is a local ring. Therefore, $\widehat{P_i}$ is indecomposable. By \citep[(6.5), (6.7)]{zbMATH00046729}, $\widehat{P_i}(\widehat{\pri_\pri})\simeq P_i(\pri)$ is indecomposable. By (\ref{eqcel50}), the Cartan matrix of $A(\pri)$ has determinant 1. Note that $A(\pri)$ is cellular. By Theorem 1.1 of \citep{zbMATH01384521}, $A(\pri)$ is split quasi-hereditary with standard modules $\theta_\l(\pri)$. Therefore, $r=|\L|$ and since $\pri$ is arbitrary $(A, \theta_{\l\in \L^{op}})$ is split quasi-hereditary.
	\end{proof}
	
	Cellular algebras over fields which are quasi-hereditary admit, up to equivalence, only one quasi-hereditary structure. This result is due to Coulembier \citep[Theorem 2.1.1]{zbMATH07203140}. Our focus is now to extend this result to cellular Noetherian algebras.
	To this end, we need to recall some facts about the ordering of the standard modules in  a quasi-hereditary algebra. For finite-dimensional algebras, the order of the split quasi-hereditary algebra is determined by the occurrences of simples $\top P(\mu)$ on $\St(\l)$ and $\St(\l)$ on $P(\mu)$ (see for example Proposition \ref{splithwcfieldcase}). If $A$ has a simple preserving duality $(-)^\natural$, then $\St(\mu)^\natural\simeq \Cs(\mu)$, $\mu\in \L$. Further, the number of occurrences of $\St(\mu)$ in $P(\l)$ is equal to the multiplicity of $\top P(\l)$ in $\St(\mu)$ (see for example Lemma 2.5 of \citep{zbMATH00140218}). So, this information can be recovered to some extent by the Grothendieck group of $A$. The \textbf{Grothendieck group} of $A$, here denoted by $G_0^R(A)$, is the abelian group generated by the symbols $[M]$, $M\in A\m\cap R\proj$ with relations $[M]=[M']+[M'']$ whenever there exists an $(A, R)$-exact sequence $0\rightarrow M'\rightarrow M\rightarrow M''\rightarrow 0$.
	Therefore, if we have two sets of standard modules for a finite-dimensional algebra $A$ with the same image in the Grothendieck group, then we can choose the order so that both sets give the same order in Proposition \ref{splithwcfieldcase}. This order is known as \textbf{essential order}. By Proposition \ref{uniquenessofstandardgivenproj}, these sets of standard modules must coincide. Given the existence of a simple preserving duality, Theorem 2.1.1 of \citep{zbMATH07203140} implies that every two sets of standard modules have the same image in the Grothendieck group for finite-dimensional algebras. In particular, if a cellular algebra is split quasi-hereditary, then there is a bijection $\phi\colon \L\rightarrow \L$ such that $\St(\l)\simeq \theta_{\phi_\l}$ if $A$ is also split quasi-hereditary with standard modules $\St(\l)$. %
	 Therefore, we can establish the following.
	
	\begin{Theorem}\label{Coulembieruniqueness}
		Let $R$ be a commutative regular Noetherian ring with finite Krull dimension. Let $A$ be a cellular $R$-algebra with cell datum $(\L, M, C, \iota)$. Assume that $A$ has finite global dimension and $(A, \{\Delta(\omega)_{\omega\in \Omega}\})$ is a split quasi-hereditary algebra. Then, there exists an equivalence of categories \mbox{$F\colon A\m\rightarrow A\m$} and a bijective map  between posets $\phi\colon \L^{op}\rightarrow \Omega$ such that $
		F\theta_\l\simeq \St(\phi(\l))\otimes_R U_\l,$ with $ \l\in \L, \ U_\l\in Pic(R).
		$
	\end{Theorem}
	\begin{proof}
		Since $A$ is a cellular $R$-algebra, $A(\mi)$ is cellular with cell modules $\theta_\l(\mi)$ for every maximal ideal $\mi$ of $R$ (see Proposition \ref{changeringcellular} and Corollary \ref{cellmodules}). By Theorem \ref{cellularfiniteglobaldimension}, $(A, \theta_{\l\in \L^{op}})$ is split quasi-hereditary. Also $(A(\mi), \{\Delta(\omega)(\mi)_{\omega\in \Omega}\})$ and $(A, \theta(\mi)_{\l\in \L^{op}})$  are split quasi-hereditary algebras for every maximal ideal $\mi$ of $R$. By Proposition \ref{splithwcfieldcase}, $\Omega$ is isomorphic to $\L$ as a set. Let $\phi\colon \L^{op}\rightarrow \Omega$ be such a bijection. By the discussion above, Theorem 2.1.1 of \cite{zbMATH07203140} and Lemma 1.2.6 of \cite{zbMATH07203140}, we obtain that $\St(\phi(\l))(\mi)\simeq \theta_\l(\mi)$ for all $\l\in \L$. By Theorem \ref{quasihereditaryhwc}, Proposition \ref{bijectionsplitheredity}, Lemma \ref{split modules change ring}, and Lemma \ref{splithereditychainequalunderresiduefield}, the result follows and in particular we see that $\phi$ is also a bijection of posets.
	\end{proof}

\section{Finitistic dimension of some integral Schur algebras}\label{Finitistic dimension of some integral Schur algebras}

To compute the global dimension and the finitistic dimension of some Schur algebras we will use the known fact that Schur algebras are split quasi-hereditary. In fact, making use of their split quasi-hereditary structure and applying the results Totaro \cite{zbMATH00966941} for the global dimension of finite-dimensional Schur algebras we can extend these results to Schur algebras having as ground rings regular Noetherian rings with finite Krull dimension. 

There are many proofs for Schur algebras being split quasi-hereditary over any commutative Noetherian ring (see for example \citep[Theorem 3.7.2]{CLINE1990126}, \citep[1.2]{zbMATH04031957},  \citep[7.2]{zbMATH00427660},  \citep[Theorem 4.1]{zbMATH04116809}). Using Theorem \ref{splitqhchainschangeofring}, we can simplify this proof and reduce this proof to just knowing that Schur algebras are quasi-hereditary over algebraically closed fields.

The study of Schur algebras goes back to the PhD thesis of Schur \citep{zbMATH02662157}. Schur used the representation theory of Schur algebras as an intermediary step to determine the polynomial representation theory of the complex general linear group with the representation theory of the symmetric group over the complex numbers. For a detailed exposition on Schur algebras (over infinite fields) we refer to \citep{zbMATH05080041}.

Let $R$ be a commutative Noetherian ring with identity.	Fix natural numbers $n, d$. The symmetric group on $d$ letters $S_d$ acts by place permutation on the $d$-fold tensor product $(R^n)^{\otimes d}$, that is,
$$(v_1\otimes\cdots\otimes v_d)\sigma=v_{\sigma(1)}\otimes\cdots\otimes v_{\sigma(d)}, \ \sigma\in S_d, \ v_i\in R^n.$$ In particular, $V^{\otimes d}:=(R^n)^{\otimes d}$ is a right module over the group algebra $RS_d$.
The endomorphism algebra $\End_{RS_d}\left( V^{\otimes d}\right)$ is known as the \textbf{Schur algebra} $S_R(n, d)$ (\citep{zbMATH03708660}).
We can construct a split heredity chain of $S_R(n, d)$ using a well-known basis of $S_R(n, d)$.

Let $I(n, d)$ be the set of maps $i\colon \{1, \ldots, d\}\rightarrow \{1, \ldots, n\}$. We write $i(a)=i_a$. $S_d$ acts on $I(n, d)$  by place permutation. In the same way, $S_d$ acts on $I(n, d)\times I(n, d)$ by the diagonal action.
We will write $(i, j)\sim (f, g)$ if $(i, j)$ and $(f, g)$ belong to the same $S_d$-orbit. $S_R(n, d)$ has a basis over $R$ ${\{\xi_{i, j}\ | \ (i, j)\in I(n, d)\times I(n, d) \}}$ satisfying \begin{align}
	\xi_{i, j}(e_{s_1}\otimes\cdots\otimes e_{s_d})=\sum_{\substack{l\in I(n, d)\\ (l, s)\sim (i, j)}} e_{l_1}\otimes\cdots\otimes e_{l_d}, \label{eqex2}
\end{align} for a given basis $\{e_{{s_1}}\otimes \cdots \otimes e_{{s_d}}\colon 1\leq s_1, \ldots, s_d\leq n \}$ of $V^{\otimes d}$. In particular, $\xi_{i, j}=\xi_{f, g}$ if and only if $(i, j)\sim (f, g)$.

The existence of the $R$-basis of $S_R(n, d)$ satisfying (\ref{eqex2}) implies the existence of a base change property
\begin{align}
	R\otimes_{\mathbb{Z}} S_{\mathbb{Z}}(n, d)\simeq S_R(n, d).
\end{align}

For each $i\in I(n, d)$ we can associate a weight $\l(i)$. More precisely, a \textbf{weight} of an element $i\in I(n, d)$ is the composition $\l=(\l_1, \ldots, \l_n)$ of $d$ in at most $n$ parts with $\l_j=|\{1\leq \mu\leq d\colon i_\mu=j \}$. Let $\L(n, d)$ be the set of all weights associated with $I(n, d)$. Then, by (\ref{eqex2}), for each $\l(i)\in \L(n, d)$ there exists an idempotent $\xi_\l:=\xi_{i, i}$. Let $\L^+(n, d)$ be the subset of $\L(n, d)$ formed by the partitions $\l=(\l_1\geq \ldots\geq\l_n)$ of $d$ in at most $n$ parts. $\L^+(n, d)$ is partially ordered by the dominance order $\leq$, that is, $\l\leq \mu$ if and only if $\l_1+\ldots + \l_j\leq \mu_1+\ldots+\mu_j$, for all $j$.  Let $\L^+(n, d)\rightarrow \{1, \ldots, t\}, \ \l^k\mapsto k$ be an increasing bijection. Set $e^k$ to be the idempotent $\sum_{l\geq k} \xi_{\l^l}$. Put $J_k=S_R(n, d)e^kS_R(n, d)$. 
\begin{Theorem}\label{schuralgebraisqh}
	For any commutative Noetherian ring $R$, the Schur algebra $S_R(n, d)$ is a split quasi-hereditary algebra over $R$ with split heredity chain $0\subset J_t\subset \cdots\subset J_2\subset J_1=S_R(n, d)$.
\end{Theorem}
\begin{proof}
	The statement for algebraically closed fields follows from Theorem 4.1 of \citep{zbMATH04116809}. Applying Theorem \ref{splitqhchainschangeofring} together with Theorem \ref{quasihereditaryhwc} and Theorem \ref{splitqhalgebraicallyclosed}, the result follows.
\end{proof} 

Schur algebras are also cellular algebras. Consider the $R$-linear map $\iota\colon S_R(n, d)\rightarrow S_R(n, d)$ given by $\iota(\xi_{f, g})=\xi_{g, f}$, $f, g\in I(n, d)$. We call $\iota$ the involution of the Schur algebra. Observe that $\iota(\xi_\l)=\xi_\l$ for every $\l\in \L^+(n, d)$. In particular, $\iota$ preserves all idempotents in the split heredity chain of $S_R(n, d)$. Hence, by Proposition \ref{splitqhwithdualityiscellular}, $S_R(n, d)$ is a cellular algebra. Note that, the order of $\L^+(n, d)$ for the definition of cellular algebra is now the reversed order of the dominance order.  By Theorem \ref{Coulembieruniqueness}, this means that the Schur algebras $S_R(n, d)$ have a unique split quasi-hereditary structure. 

Given a prime number $p$, we denote by $\alpha_p(d)$ the sum of all the digits of the representation of $d$ in base $p$. By $\alpha_0(d)$ we mean just $d$. 

We are now ready to give some applications of the results discussed in Subsection \ref{Global dimension of (split) quasi-hereditary algebras}.
\begin{Prop}\label{schuralgebrasglobaldimension}
	Let $R$ be a regular Noetherian ring with finite Krull dimension. Assume that $m, n, d$ are natural numbers satisfying $n\geq d$. Then, the following holds:\begin{align}
		\gldim S_R(n, d)&=\dim R+2d-\underset{\mi\in \MaxSpec R}{\min}2\alpha_{\characteristic R(\mi)}(d)\label{eq69}\\
		\findim S_{\mathbb{Z}/m\mathbb{Z}}(n, d)&=2d-\underset{\substack{p\in \mathbb{N}\colon p|m \\ p \ \mathrm{ is} \ \mathrm{ prime } \\ }}{\min} 2\alpha_p(d). \label{eq70}
	\end{align}
\end{Prop}
\begin{proof}
	By Theorem 2 of \cite{zbMATH00966941}, \begin{align}
		\max\{\gldim S_R(n, d)(\mi)\colon \mi \in \MaxSpec R \}&= \max \{\gldim S_{R(\mi)}(n, d)\colon \mi\in \MaxSpec R \}
		\\&=\max \{2d-2\alpha_{\characteristic R(\mi)}(d) \colon \mi \in \MaxSpec R \}\\&= 2d-\underset{\mi\in \MaxSpec R}{\min}2\alpha_{\characteristic R(\mi)}(d).
	\end{align}
By Theorem \ref{globaldimensionqh}, (\ref{eq69}) follows. 

Observe that the maximal ideals of $\mathbb{Z}/m\mathbb{Z}$ are of the form $p\mathbb{Z}/m\mathbb{Z}$ where $p$ is a prime number dividing $m$. Hence, again by  Theorem 2 of \cite{zbMATH00966941},
\begin{align*}
	2d-\underset{\substack{p\in \mathbb{N}\colon p|m \\ p \ \mathrm{ is} \ \mathrm{ prime } \\ }}{\min} 2\alpha_p(d) = 	\max\{\gldim S_{\mathbb{Z}/m\mathbb{Z}}(n, d)(\mi)\colon \mi \in \MaxSpec \mathbb{Z}/m\mathbb{Z} \}.
\end{align*}Using Baer's criterion we can see that for any principal ideal domain $S$ and any ideal $I$ of $S$, the ring $S/I$ is self-injective. Since the finitistic dimension of any self-injective ring is zero we obtain (\ref{eq70}) by applying Proposition \ref{finitisticqh}.
\end{proof}
Here, $\characteristic R$ denotes the characteristic of $R$.
\begin{appendices}

\appendix

For the convenience of the reader, we will collect some results based on \cite{Rouquier2008} aiming to strengthen the foundations of the theory of integral split highest weight categories. We will give a detailed exposition on the bijection between standard modules and split heredity ideals in the integral setup and filtrations of modules into standard modules making the material of integral highest weight categories more accessible. Along the way, we give more properties of projective $R$-split $A$-modules which are the building blocks of split quasi-hereditary algebras. In particular, we clarify that such objects are a generalization of maximal standard modules over finite-dimensional quasi-hereditary algebras over algebraically closed fields. 

\section{Projective $R$-split $A$-modules} \label{Rsplit modules} 

Throughout this appendix assume, unless stated otherwise, that $R$ is a commutative Noetherian ring with identity and $A$ is a projective Noetherian $R$-algebra.
Recall that a module $L\in A\proj$ which is faithful over $R$ is called projective $R$-split $A$-module if the canonical morphism 
	\begin{align}
	\tau_{L, P}\colon L\otimes_R \Hom_A(L, P)\rightarrow P, \quad l\otimes f\mapsto f(l)
\end{align} is an $(A, R)$-monomorphism for all projective $A$-modules $P$.
We denote $\mathcal{M}(A)$ the set of isomorphism classes of projective $R$-split $A$-modules.  

For $P=A$ we can consider the map $\tau_{L}\colon L\otimes_{\End_A(L)^{op}} \Hom_A(L, P)\rightarrow P, \quad l\otimes f\mapsto f(l)$.

We will see that these modules  in $\mathcal{M}(A)$ are exactly the standard (not necessarily indecomposable) modules  with maximal index when $A$ is split quasi-hereditary. Whereas the image of the map $\tau_{L, A}$ is a split heredity ideal for $L\in \mathcal{M}(A)$. For rings $R$ with non-trivial idempotents, we will be able to say that the modules in $\mathcal{M}(A)$ are projective indecomposable.

\subsection{Properties of projective $R$-split $A$-modules}

The motivation behind the definition of projective $R$-split $A$-modules lies in the following result.
\begin{Lemma}\label{standardfromsplitquasi}
	Let $K$ be a field and $A$ a finite-dimensional $K$-algebra. If $AeA$ is a split heredity ideal of $A$ for some primitive idempotent $e\in A$, then $Ae\in \mathcal{M}(A)$.
\end{Lemma}
\begin{proof}
	Since $AeA$ is projective as left ideal of $A$ we obtain that the multiplication map $Ae\otimes_{eAe} eA\rightarrow AeA$ is an isomorphism (see Statement 7 \citep{Dlab1989d}). Since $\Hom_A(Ae, A)\simeq eA$ it remains to show that $eAe=K$. Again, as $AeA$ is projective, $AeA\in \add_A Ae$. By projectivization,
	\begin{align}
		eA=eAeA=\Hom_A(Ae, AeA)\in \End_A(Ae)\proj=eAe\proj.
	\end{align} The identification $eA=eAeA$ is obtained by applying the tensor product $eA\otimes_A-$ to the multiplication map $Ae\otimes_{eAe} eA\rightarrow AeA$.  On the other hand, we can write $eA=eAe\oplus eA(1-e)$ as left $eAe$-modules. Thus, $eA$ is an $eAe$-progenerator. By Tensor-Hom adjunction,
	\begin{align}
		\End_A(AeA)\simeq \Hom_{eAe}(eA, \Hom_A(Ae, Ae\otimes_{eAe} eA))\simeq \End_{eAe}(eA).
	\end{align}Therefore, $eAe$ is Morita equivalent to $\End_A(AeA)$. By assumption, $\End_A(AeA)$ is Morita equivalent to $K$. So, $eAe$ is Morita equivalent to $K$.
	Since $e$ is primitive, $Ae$ is indecomposable. Thus, $eAe$ is local. So, we must have $eAe=K$.
\end{proof}

In the Noetherian case, not all projective $R$-split modules can be considered under this form. The first problem we encounter is that projective modules cannot be decomposed into projective modules defined by idempotents. And therefore, the definition of heredity ideals in the form $AeA$ is no longer suitable, and so neither is the choice of standard $Ae$.

\begin{Lemma}\label{imtauisideal}
	Let $L$ be a finitely generated projective $A$-module. Then, $\tau_{L}$ is an $(A,A)$-bimodule morphism. Moreover, $J=\im \tau_L$ is an ideal of $A$ and $J^2=J$. \label{im tau is ideal}
\end{Lemma}
\begin{proof}
	See 	\citep[Lemma 4.3]{Rouquier2008}.
\end{proof}

The reason to require the map $\tau_{L}$ to be an $(A, R)$-monomorphism is that in this way $A/\im \tau_{L}$ is a projective Noetherian $R$-algebra. 

\begin{Lemma}(see \citep[Lemma 4.4]{Rouquier2008})\label{AmoduloJmod}
	Let $J$ be an ideal of $A$. Assume that $J^2=J$. Let $M$ be an $A$-module. Then $\Hom_A(J, M)=0$ if and only if $JM=0$ if and only if $M\in A/J\m$.\label{A/Jmod}
\end{Lemma}
\begin{proof}
	Assume that $\Hom_A(J, M)=0$. Consider $m\in M.$ We can define the $A$-homomorphism $f\colon J\rightarrow M$, with $f(j)=jm, \ j\in J$. By assumption $f=0$, hence $JM=0$. 
	
	Reciprocally, assume $JM=0$. Let $g\in \Hom_A(J, M)$. For any $j\in J$, there exists $j_1, j_2$ such that $j=j_1 j_2$, hence $g(j)=g(j_1j_2)=j_1g(j_2)\in JM=0$. Hence, $g=0$. 
\end{proof}

Note that the condition $J=J^2$ is fundamental. In fact, assume that if $JM=0$, then $\Hom_A(J, M)=0$. Then, consider $M:=J/J^2$. We have $\Hom_A(J, J/J^2)=0$. In particular, the canonical epimorphism $J\twoheadrightarrow J/J^2$ is zero. Hence, $J=J^2$.

\begin{Prop}(see \citep[Lemma 4.5]{Rouquier2008}) \label{split modules characterization}
	Let $L$ be a finitely generated projective $A$-module which is a faithful $R$-module. The following are equivalent:
	\begin{enumerate}[(i)]
		\item $\tau_{L, P}\colon L\otimes_R \Hom_A(L, P)\rightarrow P$ is an $(A, R)$-monomorphism for all $P\in A\proj$.
		\item $\tau_{L, A}\colon L\otimes_R \Hom_A(L, A)\rightarrow A$ is an $(A, R)$-monomorphism.
		\item $R\simeq \End_A(L)$ and given $P\in A\proj$, there is a submodule $P_0$ of $P$ satisfying the following conditions
		\begin{enumerate}[(a)]
			\item $P/P_0\in R\proj$,
			\item $\Hom_A(L, P/P_0)=0$,
			\item $P_0\simeq L\otimes_R U$ for some $R$-progenerator $U$.
		\end{enumerate}
	\end{enumerate}
\end{Prop}
\begin{proof}
	(i)$\implies$ (ii). Clear since $A\in A\proj$.
	
	(ii)$\implies$ (i). Notice that for $P=P_1\oplus P_2$, \begin{align*}
		L\otimes_R \Hom_A(L, P)\simeq L\otimes_R (\Hom_A(L, P_1)\oplus \Hom_A(L, P_2)) \simeq (L\otimes_R \Hom_A(L, P_1))\oplus L\otimes_R \Hom_A(L, P_2),
	\end{align*} hence $\tau_{L, P_1\oplus P_2}$ is equivalent to $\tau_{L, P_1}\oplus \tau_{L, P_2}$. So, it follows that $\tau_{L, P}$ is an $(A, R)$-monomorphism for any $P\in A\proj$.
	
	(i)$\implies$ (iii). By (i), $\tau_{L, L}$ is an $(A, R)$-monomorphism. Putting $f=\id_L$, we see that ${\tau_{L, L}(l\otimes f)=l}, l\in L$. Hence, it is an $R$-isomorphism.
	Since $L\in A\proj$, it follows that $L\in R\proj$. As $R$ is commutative and $L$ is faithful, $L$ is an $R$-progenerator \citep[Proposition 12.2]{Faith1973}. Define $B=\End_R(L)^{op}$. Then, ${F=L\otimes_R -\colon R\m \rightarrow B\m}$ is an equivalence of categories with right adjoint \linebreak${G=\Hom_B(L, -)\colon B\m \rightarrow R\m}$. Notice that $F\End_A(L)=L\otimes_R \End_A(L)\simeq L$. Furthermore, $\End_A(L)\simeq GF\End_A(L)\simeq GL=\End_B(L)\simeq R$, since the double centralizer property holds on generators.
	
	Let $P$ be a finitely generated projective $A$-module. Define $P_0=\im \tau_{L, P}$. As $\tau_{L, P}$ is an $(A, R)$-monomorphism we obtain that $P_0$ is an $R$-summand of $P$. Moreover $P/P_0$ is an $R$-summand of $P$, hence it is projective over $R$. Since $L\in A\proj$, we have the exact sequence
	\begin{align}
		0\rightarrow \Hom_A(L, P_0)\rightarrow \Hom_A(L, P)\rightarrow \Hom_A(L, P/P_0)\rightarrow 0.\label{eq3}
	\end{align}
	However, the canonical map $\Hom_A(L, P_0)\rightarrow \Hom_A(L, P)$ is surjective:  
	In fact, for each $h\in \Hom_A(L, P)$, define $g\in \Hom_A(L, P_0)$ such that $g(l)=l\otimes h$. Hence, $\tau_{L, P}\circ g=h$.
	Therefore, by the exactness of (\ref{eq3}), $\Hom_A(L, P/P_0)=0$. Since $L$ is faithful over $R$, it follows that $U=\Hom_A(L, P)$ is faithful over $R$, and $P_0\simeq L\otimes U$.
	
	(iii)$\implies$ (i). Let $P\in A\proj$. Consider the exact sequence $0\rightarrow P_0\rightarrow P\rightarrow P/P_0$. Applying $\Hom_A(L, -)$ we obtain the exact sequence
	$0\rightarrow \Hom_A(L, P_0)\rightarrow \Hom_A(L, P)\rightarrow \Hom_A(L, P/P_0)=0\rightarrow 0$. Therefore, the map $\Hom_A(L, P_0)\rightarrow \Hom_A(L, P)$ is an isomorphism. 
	
	We the following diagram
	
	\begin{center}
		\begin{tikzcd}
			\Hom_A(L, P_0) \arrow[r, "w"] & \Hom_A(L, L\otimes_R U) \arrow[r, "z"] & \Hom_A(L, L)\otimes_R U \arrow[d, "\simeq"]\\
			& & R\otimes_R U\arrow[d, "\simeq"]\\
			\Hom_A(L, P)\arrow[uu, "\simeq"]\arrow[rr] & & U
		\end{tikzcd}.
	\end{center}
	Here, $w$ is an isomorphism since  $P_0\simeq L\otimes_R U$ by assumption, and $z$ is an isomorphism since the map \linebreak$\Hom_A(Q, L)\otimes_R U\rightarrow \Hom_A(Q, L\otimes U)$ is an isomorphism for any $Q\in A\proj$.
	
	Therefore,
	\begin{center}
		\begin{tikzcd}
			L\otimes_R \Hom_A(L, P)\arrow[r, "\tau_{L, P}"] \arrow[d, "\simeq"]& P\\
			L\otimes_R U \arrow[r, "\simeq"] & P_0\arrow[u, hookrightarrow]
		\end{tikzcd}.
	\end{center}
	Now since $P_0$ is a summand of $P$, we get that $\tau_{L, P}$ is an $(A, R)$-monomorphism.
\end{proof}

\begin{Remark}
	Notice that for any $L\in \mathcal{M}(A)$, $R\simeq \End_A(L)$. Hence, $\tau_{L, A}=\tau_L$. Furthermore, $\End_R(L)\simeq \Hom_A(\im \tau_L, A)$.  \label{remendprojonma}
\end{Remark}
In fact, \begin{align*}
	\Hom_{A^{op}}(\im \tau_L, A)&\simeq \Hom_{A^{op}}(L\otimes_R \Hom_A(L, A), A)\\
	&\simeq \Hom_R(L, \Hom_{A^{op}}(\Hom_A(L, A), A))\simeq \End_R(L).
\end{align*}

We can also use the relative projective modules to determine if a given projective module is $R$-split. But first, we need the following elementary lemma.

\begin{Lemma}\label{tensoronsecondcomponetofhom}
	Let $M\in A\proj$. Then, the $R$-homomorphism $$\varsigma_{M, N, U}\colon \Hom_A(M, N)\otimes_R U\rightarrow \Hom_A(M, N\otimes_R U),$$ given by $ \ g\otimes u\mapsto g(-)\otimes u$ is an $R$-isomorphism.
\end{Lemma}
\begin{proof}
	Consider $M=A$. The following diagram is commutative:
	\begin{center}
		\begin{tikzcd}
			\Hom_A(A, N)\otimes_R U \arrow[r, "\varsigma_{A, N, U}"] \arrow[d]& \Hom_A(M, N\otimes_R U)\arrow[d]\\
			N\otimes_R U\arrow[r, equal] &N\otimes_R U
		\end{tikzcd}
	\end{center} 
	Both columns are isomorphisms, thus $\varsigma_{A, N, U}$ is an isomorphism. Since this map is compatible with direct sums, it follows that $\varsigma_{M, N, U}$ is an isomorphism for every $M\in A\proj$ and any $N\in A\m$, $U\in R\m$.
\end{proof}

\begin{Lemma}
	Let $L\in A\proj$ which is a faithful $R$-module. The following are equivalent:
	\begin{enumerate}[(i)]
		\item $\tau_{L, A}\colon L\otimes_R \Hom_A(L, A)\rightarrow A$ is an $(A, R)$-monomorphism.
		\item $\tau_{L, M}\colon L\otimes_R \Hom_A(L, M)\rightarrow M$ is an  $(A, R)$-monomorphism for every finitely generated $(A, R)$-projective module $M$.
	\end{enumerate}
\end{Lemma}
\begin{proof}
	The implication (ii)$\implies$ (i) is clear since $A$ is $(A, R)$-projective.
	Assume that (i) holds. Let $M$ be an $(A, R)$-projective module. Since  $\tau_{L, X_1\oplus X_2}$ is equivalent to $\tau_{L, X_1}\oplus \tau_{L, X_2}$ for every $X_1, X_2\in A\m$ we can assume that $M=A\otimes_R X$ for some $X\in R\m$.  There is a commutative diagram
	\begin{equation}
		\begin{tikzcd} 
			L\otimes_R \Hom_A(L, A)\otimes_R X \arrow[d, "L\otimes_R \varsigma_{L{,}A{,}X}", outer sep=0.75ex] \arrow[r, "\tau_{L {, }A}\otimes_R X", outer sep=0.75ex]& A\otimes_R X \arrow[d, equal]\\
			L\otimes_R \Hom_A(L, A\otimes_R X)\arrow[r, "\tau_{L{, }A\otimes_R X}"] & A\otimes_R X
		\end{tikzcd}.
	\end{equation}
	In fact, following the notation of Proposition \ref{tensoronsecondcomponetofhom}, for every $l\in L$, $g\in \Hom_A(L, A)$, $x\in X$,
	\begin{align}
		\tau_{L, A\otimes_R X}\circ L\otimes_R \varsigma_{L, A, X}(l\otimes g\otimes x)=\tau_{L, A\otimes_R X}(l\otimes g(-)\otimes x)=g(l)\otimes x=\tau_A\otimes_R X(l\otimes g\otimes x).
	\end{align}
	By assumption, there exists an $R$-map $t\colon A\rightarrow L\otimes_R \Hom_A(L, A)$ satisfying $t\circ \tau_{L, A}=\id_{L\otimes_R \Hom_A(L, A)}$. It follows that
	$L\otimes_R \varsigma_{L, A, X}\circ t\otimes_R X\circ \tau_{L, A\otimes_R X}=\id_{L\otimes_R \Hom_A(L, A\otimes_R X)}$.
\end{proof}

We can observe that the  projective $R$-split left $A$-modules determine the projective $R$-split right  $A$-modules.

\begin{Lemma}\label{rightprojectivesplitmodules}
	If $L\in \mathcal{M}(A)$, then $\Hom_A(L, A)\in \mathcal{M}(A^{op})$.
\end{Lemma}
\begin{proof}
	Since $L\in A\proj$, $\Hom_A(L, A)$ is projective as right $A$-module and $\End_A(\Hom_A(L, A))\simeq \End_A(L)\simeq R$. Further, $\Hom_A(\Hom_A(L, A), A)\simeq L$ as $A$-modules and the following diagram is commutative:
	\begin{equation}
		\begin{tikzcd}
			\Hom_A(\Hom_A(L, A), A)\otimes_R \Hom_A(L, A) \arrow[r, "\tau_{\Hom_A(L {,}A) {, A}}", outer sep=0.75ex] & A \\
			L\otimes_R \Hom_A(L, A)\arrow[u, "\simeq"] \arrow[r, "\tau_{L{,} A}"] & A \arrow[u, equal]
		\end{tikzcd}. \tag*{\qedhere}
	\end{equation}
\end{proof}

In the following proposition, we determine when a projective $R$-split $A$-module is indecomposable.
\begin{Prop} Assume that $R$ has no non-trivial idempotents. Then, all modules in $\mathcal{M}(A)$ are projective indecomposable $A$-modules.\label{faithfulmodulesovermaximalideals}
\end{Prop}
\begin{proof} 
	Let $L\in \mathcal{M}(A)$. By definition, $L$ is projective over $A$. Assume that $L$ is decomposable, $L\simeq X_1\bigoplus X_2$. Then, we have a non-trivial idempotent $L\twoheadrightarrow X_1\hookrightarrow L$ in $\End_A(L)\simeq R$. So, $L$ must be indecomposable as $A$-module.
\end{proof}

The following lemma shows that $\mathcal{M}(A)$ behaves well with respect to ground ring change. 

\begin{Lemma}
	(see \citep[Proof of Lemma 4.10]{Rouquier2008}) Let $L$ be a finitely generated $A$-module. Let $S$ be a commutative  $R$-algebra and Noetherian ring. If $L\in \mathcal{M}(A)$, then $S\otimes_R L\in \mathcal{M}(S\otimes_R A)$. Moreover, the following are equivalent: \label{split modules change ring}
	\begin{enumerate}[(i)]
		\item $L\in \mathcal{M}(A)$;
		\item The localization $L_\mi=R_\mi\otimes_R L\in \mathcal{M}(A_\mi)$ for every maximal ideal $\mi$ of $R$;
		\item $L$ is projective over $R$ and $L(\mi)\in \mathcal{M}(A(\mi))$ for every maximal ideal $\mi$ of $R$.		
	\end{enumerate}\label{splitmoduleschangering}
\end{Lemma}
\begin{proof}
	Since $L\in \mathcal{M}(A)$, $L$ is a projective $A$-module and an $R$-progenerator. This gives that $S\otimes_R L$ is projective over $S\otimes_R A$ and $R$ is a summand of a finite direct sum of copies of $L$.
	Thus, $(S\otimes_R L)^t\simeq S\otimes_R L^t\simeq S\otimes_R R\oplus K\simeq S\oplus S\otimes K$, for some $K$. Hence, $S\otimes_R L$ is an $S$-progenerator. 
	
	Note that \begin{align*}
		S\otimes_R (L\otimes_R \Hom_A(L, A))&\simeq S\otimes_S S\otimes_R (L\otimes_R \Hom_A(L, A))\simeq S\otimes_R L\otimes_S S\otimes_R \Hom_A(L, A)\\&\simeq S\otimes_R L\otimes_S \Hom_{S\otimes_R A}(S\otimes_RL, S\otimes_R A).
	\end{align*}
	Denote this isomorphism by $\alpha$ and its inverse by $\beta$. The following diagram is commutative:
	\begin{center}
		\begin{tikzcd}[every arrow/.append style={shift left}] 
			S\otimes_R L\otimes_S \Hom_{S\otimes_R A}(S\otimes_RL, S\otimes_R A) \arrow[rr, "\tau_{S\otimes_R L, S\otimes_R A}"] \arrow[d, "\beta"]& & S\otimes_R A\\
			S\otimes_R (L\otimes_R \Hom_A(L, A)) \arrow[u, "\alpha"] \arrow[rru, "S\otimes_R \tau_{L, A}", swap] & &
		\end{tikzcd}.
	\end{center}
	In fact, \begin{multline*}
		\tau_{S\otimes_R L, S\otimes_R A} \circ \alpha (s\otimes l\otimes g)=\tau_{S\otimes_R L, S\otimes_R A}(s\otimes l\otimes 1_S\otimes g)= (1_S\otimes g)(s\otimes l)=s\otimes g(l)\\=\id_S\otimes_R\tau_{L, A}(s\otimes l\otimes g).
	\end{multline*}
	Thus, $\tau_{S\otimes_R L, S\otimes_R A}$ is a composition of an $(S\otimes_R A, S)$-monomorphism with an isomorphism, and so it is an $(S\otimes_RA, S)$-monomorphism. By \ref{split modules characterization}, $S\otimes_R L\in \mathcal{M}(S\otimes_R A)$.
	
	Now assume $i$. $ii$ follows putting $S=R_\mi$ for each maximal ideal $m$ of $R$. $iii$ follows putting $S=R(\mi)$ for each maximal ideal $\mi$ of $R$. Clearly, in this case, $L$ is projective over $R$. 
	
	Now assume $ii$. $L_\mi$ is faithful for any $\mi$ maximal ideal of $R$. Hence, $\Ann L_\mi=0$, where $\Ann L$ denotes the annihilator of $L$. Take $r\in \Ann L$, then $s\otimes r s'\otimes l= ss'\otimes rl=ss'\otimes 0=0$ for any $s,s'\in R_\mi$, $l\in L$. This means that $s\otimes r\in \Ann L_\mi=0$. So, any element in $(\Ann L)_\mi=R_\mi\otimes \Ann L$ is zero. Thus, $\Ann L=0$, that is, $L$ is faithful over $R$. As $L_\mi$ is projective over $A_\mi$ for any maximal ideal $\mi$ of $R$, it follows that $L$ is projective over $A$. Now, $\tau_{L_\mi, A_\mi}=(\tau_{L, A})_\mi $ is injective, so $\tau_{L, A}$ is injective. $R_\mi$ is a flat $R$-module, hence $(\coker \tau_{L, A})_\mi = \coker \tau_{L_\mi, A_\mi}$ which is projective over $R_\mi$. So, $\coker \tau_{L, A}$ is projective over $R$. Therefore, $\tau_{L, A}$ is an $(A, R)$-monomorphism and $i$ follows.
	
	Finally, assume $iii$. Since $L$ is projective over $R$ and $L(\mi)$ is projective over $A(\mi)$, it follows that $L$ is projective over $A$. Consider the canonical map $R\rightarrow \End_A(L)$, given by $r\mapsto (l\mapsto rl)$. Denote this map by $\phi$. Since $L\in A\proj$, we have $\End_A(L)$ is a projective $R$-module. By assumption, $L(\mi)\in \mathcal{M}(A(\mi))$ and by Proposition \ref{split modules characterization}, $\phi(\mi)$ is an isomorphism for every maximal ideal $\mi$ in $R$. According to Lemma \ref{nakayamalemmasurjectiveproj}, $\phi$ is an $R$-isomorphism. In particular, $L$ is faithful over $R$. Let $D$ be the standard duality. By Tensor-Hom adjunction, we have isomorphisms $k_1\colon \Hom_R(L, DA\otimes_AL)\rightarrow D(D(DA\otimes_A L)\otimes_RL)$ and $k_2\colon \Hom_A(L, A)\rightarrow D(DA\otimes_A L)$ (see for example \citep[Proposition 2.1]{CRUZ2022410}). 
	
	Consider the right $A$-homomorphism $\zeta_{L}\colon DA\rightarrow \Hom_R(L, DA\otimes_A L)$, given by, $\zeta(f)(l)=f\otimes l, \ f\in DA, \ l\in L$. There is a commutative diagram 
	\begin{equation}
		\begin{tikzcd}
			DA \arrow[r, "\zeta_L"] \arrow[d, equal] & \Hom_R(L, DA\otimes_A L) \arrow[d, "D(Dk_2 \otimes_R L)\circ k_1"] \\
			DA \arrow[r, "D\tau_{L, A}"] & D(L\otimes_R \Hom_A(L, A))
		\end{tikzcd}. \label{eqqh6}
	\end{equation} 
	In fact, for $f\in DA, \l\in L$ and $g\in \Hom_A(L, A)$
	\begin{align}
		D(k_2\otimes_R L)\circ k_1\circ \zeta_L(f)(l\otimes g)&= k_1(\zeta_L(l))\circ k_2\otimes_R L(l\otimes g) = k_1(\zeta_L(f))(k_2(g)\otimes l) = \\
		k_2(g)(\zeta_L(f)(l)) &=k_2(g)(f\otimes l) = f(g(l)) = f \circ \tau_{L, A}(l\otimes g)=D\tau_{L, A}(f)(l\otimes g).
	\end{align}
	By assumption, $\tau_{L(\mi), A(\mi)}$ is a monomorphism for every maximal ideal $\mi$ in $R$. Denote by $D_{(\mi)}$ the standard duality in $R(\mi)$. Then, $D_{(\mi)}\tau_{L(\mi), A(\mi)}$ is surjective for every maximal ideal in $R$. By the diagram (\ref{eqqh6}), it follows that $\zeta_{L(\mi)}$ is surjective for every maximal ideal $\mi$ in $R$. Using the following commutative diagram
	\begin{equation}
		\begin{tikzcd}
			DA(\mi) \arrow[r, "\zeta_L(\mi)"] \arrow[d, "\simeq"] & \Hom_R(L, DA\otimes_A L)(\mi) \arrow[d, "\simeq"] \\
			D_{(\mi)}A(\mi) \arrow[r, "\zeta_{L(\mi)}"] & \Hom_{R(\mi)}(L(\mi), D_{(\mi)}A(\mi)\otimes_{A(\mi)} L(\mi))
		\end{tikzcd}
	\end{equation} we deduce that $\zeta_L(\mi)$ is surjective for every maximal ideal $\mi$ in $R$. By Nakayama's Lemma, $\zeta_L$ is surjective. By the commutativity of diagram (\ref{eqqh6}), $D\tau_{L, A}$ is surjective. As $L\in A\proj$, $D(L\otimes_R \Hom_A(L, A))\in R\proj$. Consequently, $D\tau_{L, A}$ is an $(A, R)$-epimorphism. Hence, $DD\tau_{L, A}$ is an $(A, R)$-monomorphism. Taking into account that $L\otimes_R \Hom_A(L, A)$ and $DA\in R\proj$ we conclude that $\tau_{L, A}$ is an $(A, R)$-monomorphism.
\end{proof}

The following result completes Lemma \ref{splitmoduleschangering} and it reduces the study of projective $R$-split $A$-modules to the study of maximal standard modules over finite-dimensional algebras over algebraically closed fields.

\begin{Lemma}\label{splitmodulesalgebraicclosure}
	Let $k$ be a field and let $A$ be a finite-dimensional $k$-algebra. Assume that $\overline{k}$ is the algebraic closure of $k$. Given $L\in A\m$, if $\overline{k}\otimes_k L\in \mathcal{M}(\overline{k}\otimes_k A)$ then $L\in \mathcal{M}(A)$.
\end{Lemma}
\begin{proof}
	It is immediate that $L$ is faithful over $k$. We will proceed to show that $L$ is projective over $A$. To see this observe that $\overline{k}$ is faithfully flat over $k$ and \begin{align}
		\overline{k}\otimes_k \Ext_A^1(L, N)=\Ext_{\overline{k}\otimes_k A}^1(\overline{k}\otimes_k L, \overline{k}\otimes_k N )=0, \quad \forall N\in A\m.
	\end{align}
	It remains to check that the map $\tau_{L, A}$ is injective. By assumption, $\tau_{\overline{k}\otimes_k L, \overline{k}\otimes_k A}$ is injective. Since $\overline{k}$ is faithfully flat over $k$ this implies that $\overline{k}\otimes_k \tau_{L, A}$ is injective and consequently $\tau_{L, A}$ is injective. Therefore, $L\in \mathcal{M}(A)$.
\end{proof}

\begin{Prop}\label{bijectionsplits}
	(see \citep[Proposition 4.7]{Rouquier2008}) There is a bijection from $\mathcal{M}(A)$ to the set of isomorphism classes of pairs $(J, P)$ where $J$ is a split heredity ideal of $A$ and $P$ is a progenerator for $B:=\End_A(J)^{op}$ such that $R\simeq \End_{\End_A(J)^{op}}(P)$. Here the equivalence is given in the following way: $(J, P)\sim (J', P')$ if and only if $J=J'$ and $P\simeq P'$ as $B$-modules. \label{bijection splits}
	Explicitly, \begin{align*}
		&\alpha\colon \mathcal{M}(A)\longrightarrow \{\text{isomorphism classes of pairs } (J, P)\}/_\sim \colon L\mapsto (\im \tau_L, \Hom_R(\Hom_A(L, A), R))\\
		&\beta\colon \{\text{isomorphism classes of pairs } (J, P)\}/_\sim \longrightarrow \mathcal{M}(A)\colon (J, P)\mapsto J\otimes_B P.
	\end{align*}
\end{Prop}
\begin{proof}
	Let $L\in \mathcal{M}(A)$. Let $J=\im \tau_L$ and $B=\End_A(J)^{op}$. By assumption, $\tau_L$ is $(A, R)$-monomorphism, so $J$ is an $R$-summand of $A$. Hence, $A/J$ is an $R$-summand of $A$. Since $L$ is projective $A$-module, by Lemma \ref{im tau is ideal}, $J^2=J$. Notice that $L\otimes_R \Hom_A(L, A)\simeq \im\tau_{L}=J$ as left $A$-modules. Since $L$ is faithful over $R$, it follows that $\Hom_A(L, A)$ is faithful over $R$. Since $L$ is an $R$-progenerator, $R$ is a summand of $L^s$ for some $s>0$. Hence, $\Hom_A(L, A)\simeq R\otimes_R\Hom_A(L, A)$ is an $R$-summand of $L^s\otimes_R \Hom_A(L, A)\simeq J^s$. Hence, $\Hom_A(L, A)$ is projective over $R$. As $R$ is commutative, $\Hom_A(L, A)$ is a progenerator for $R\m$. Now $J\simeq L\otimes_R \Hom_A(L, A)$ is projective over $A\otimes_R R\simeq A$. It remains to show that $B$ is Morita equivalent to $R$. 
	
	By Tensor-Hom adjunction, \begin{align*}
		\End_R(\Hom_A(L, A))\simeq \Hom_R(\Hom_A(L, A), \Hom_A(L, A))\simeq \Hom_A(L\otimes_R \Hom_A(L, A), A)\simeq \Hom_A(J, A).
	\end{align*} Consider the exact sequence
	\begin{align}
		0\rightarrow J\rightarrow A\rightarrow A/J\rightarrow 0. \label{eq6}
	\end{align}
	Applying $\Hom_A(J, -)$ to (\ref{eq6}) yields 
	\begin{align*}
		0\rightarrow \End_A(J)\rightarrow \Hom_A(J, A)\rightarrow \Hom_A(J, A/J)\rightarrow 0.
	\end{align*}
	Now since $J(A/J)=0$ and $J=J^2$, we get $\Hom_A(J, A/J)=0$ by Lemma \ref{A/Jmod}. It follows that \linebreak$B^{op}\simeq \Hom_A(J, A)\simeq \End_R(\Hom_A(L, A))$.
	On the other hand, $\Hom_A(L, A)$ is a progenerator of $R\m$, so the functor $\Hom_R(\Hom_A(L, A), -)\colon R\m\rightarrow B\m$ is an equivalence of categories. By Morita theorem, $P:=\Hom_R(\Hom_A(L, A), R)$ is a progenerator for $B\m$ and $R\simeq \End_B(P)$. Hence, $\alpha$ is well defined. 
	
	We claim now that $L\simeq \im \tau_L\otimes_B \Hom_R(\Hom_A(L, A), R)$ as $A$-modules.
	
	By the Morita theorem for progenerators (see e.g. \citep[Proposition 12.10]{Faith1973}), $\Hom_R(\Hom_A(L, A), R)\simeq \Hom_B(\Hom_A(L, A), B)$ as $(B, R)$-bimodules. Note the action of $A$ in $\im \tau_L\otimes_B \Hom_R(\Hom_A(L, A), R)$ is the one induced by $A$ in $L$. Hence, as left $A$-modules,
	\begin{align*}
		\im \tau_L\otimes_B \Hom_R(\Hom_A(L, A), R)&\simeq L\otimes_R \Hom_A(L, A)\otimes_B \Hom_B(\Hom_A(L, A), B)\\
		&\simeq L\otimes_R \Hom_B(\Hom_A(L, A), \Hom_A(L, A)), \text{ since } \Hom_A(L, A)\in \projr B\\
		&\simeq L\otimes_R R, \text{ since } \add \Hom_A(L, A)=R\proj \\
		&\simeq L.
	\end{align*}
For the third isomorphism see for example \citep[59.4]{Curtis2006}.
	
	Reciprocally, consider a pair $(J, P)$ such that $R\simeq \End_B(P)^{op}$, where $B=\End_A(J)^{op}$. Let $L=J\otimes_B P$.
	$J$ is projective over $A$ and $P$ is projective over $B$, hence $L$ is a projective $A\otimes_B B\simeq A$-module.
	Notice that for $M\in A\proj$ and $M'\in C\proj$ there exists a canonical isomorphism $$\Hom_A(M, N)\otimes_R \Hom_C(M', N')\rightarrow \Hom_{A\otimes_R C}(M\otimes_R M', N\otimes_R N').$$
	So, 
	\begin{align}
		\End_A(L)\simeq \End_{A\otimes_B B}(J\otimes_B P)\simeq \End_A(J)\otimes_B\End_B(P)\simeq B\otimes_B \End_B(P)\simeq \End_B(P)\simeq R.
	\end{align} 
	Consequently, $L$ is faithful over $R$. 
	
	Let $i\colon J\rightarrow A$ be the inclusion $A$-homomorphism. We can consider $f\in\Hom_B(B, \Hom_A(J, A))$ such that $f(1_B)=i$. Since $P$ is a progenerator of $B\m$, $B$ is a summand of some direct sum of copies of $P$. So, we can extend the map $f$ to $f\in \Hom_B(P^t, \Hom_A(J, A))$ such that there exists $x\in P^t$ with $f(x)=i$. Consider the canonical inclusions and projections $k_j\colon P\rightarrow P^t, \ \pi_j\colon P^t\rightarrow P$. Define $f_j=f\circ k_j\in \Hom_B(P, \Hom_A(J, A))$.
	We have $i=f(x)=\sum_j f\circ k_j\circ \pi_j(x)=\sum_j f\circ k_j(x_j)=\sum_j f_j(x_j)$ for some $x_j\in P$.
	
	Consider the adjoint map $\Hom_B(P, \Hom_A(J, A))\rightarrow \Hom_A(J\otimes_B P, A)$, which sends $f\in \Hom_B(P, \Hom_A(J, A))$ to the map \mbox{$((y\otimes x)\mapsto f(x)(y))$,} and let $g_j$ be the image of $f_j$ in $\Hom_A(J\otimes_B P, A)$. Then, for any $y\in J$,
	\begin{align*}
		\tau_{L=J\otimes_B P}(\sum_j y\otimes x_j\otimes g_j)=\sum_j \tau_L(y\otimes x_j\otimes g_j)=\sum_j g_j(y\otimes x_j)=\sum_j f_j(x_j)(y)=i(y)=y.
	\end{align*}
	Therefore, $J\subset \im \tau_L$. Note that $\Hom_A(L, A/J)\simeq \Hom_B(P, \Hom_A(J, A/J))=0$, by Tensor-Hom adjunction and by Lemma \ref{im tau is ideal}. The functor $\Hom_A(L, -)$ yields the exact sequence $0\rightarrow \Hom_A(L, J)\rightarrow \Hom_A(L, A)\rightarrow \Hom_A(L, A/J)=0$. Thus, we get $\tau_L(L\otimes_R \Hom_A(L, A))=\tau_L(L\otimes_R \Hom_A(L, J))\subset J$. We conclude that $\im \tau_L =J$. Since $P$ is a left $B$-progenerator and $\End_B(P)\simeq R$ then $\Hom_B(P, B)\simeq \Hom_R(P, R)$ as $(R, B)$-bimodules (see for example \citep[Chapter 12, pages 447-453]{Faith1973}). Now the functor $\Hom_A(J,-)$ yields the exact sequence \begin{align*}
		0\rightarrow B=\Hom_A(J, J)\rightarrow \Hom_A(J, A)\rightarrow \Hom_A(J, A/J)=0.
	\end{align*} Hence, $B\simeq \Hom_A(J, A)$ as left $B$-modules.
	Thus, as $R$-modules,
	\begin{align}
		\Hom_A(L, A)\simeq \Hom_B(P, \Hom_A(J, A))\simeq \Hom_B(P, B)\simeq \Hom_R(P, R).
	\end{align}
	Finally as $A$-modules, \begin{align*}
		J\simeq J\otimes_B B\simeq J\otimes_B \End_R(P)\simeq J\otimes_B P\otimes_R \Hom_R(P, R)\simeq J\otimes_B P\otimes_R \Hom_A(L, A)\simeq L\otimes_R \Hom_A(L, A).
	\end{align*}
	We conclude that the map $\tau_L\colon L\otimes_R \Hom_A(L, A)\rightarrow J$ is surjective between two isomorphic finitely generated $A$-modules. By Nakayama's Lemma, $\tau_L\colon L\otimes_R \Hom_A(L, A)\rightarrow J$ is an isomorphism. In particular, \linebreak$\tau_L\colon L\otimes_R \Hom_A(L, A)\rightarrow A$ is injective. Now since $A/J$ is projective over $R$, the exact sequence $0\rightarrow J\rightarrow A\rightarrow A/J\rightarrow 0$ splits over $R$. Hence, $\tau_L$ is an $(A, R)$-monomorphism, so $L\in \mathcal{M}(A)$. Thus, $\beta$ is well defined.
	
	We claim that $\alpha\circ \beta(J, P)=(J, P)$.
	
	In fact, $$\alpha\circ \beta(J, P)=\alpha(J\otimes_B P)=(\im \tau_{J\otimes_B P}, \Hom_R(\Hom_A(J\otimes_B P, A), R))=(J, \Hom_R(\Hom_A(J\otimes_B P, A), R)).$$ Since $L=J\otimes_B P\in \mathcal{M}(A)$ by the first direction we can regard $\Hom_A(L, A)$ as a right $B$-module. Recall that the functors $\Hom_B(P, -)$ and $\Hom_R(\Hom_B(P, B), -)$ form an equivalence. Hence, we obtain as left $B$-modules,
	\begin{align}
		\Hom_R(\Hom_A(L, A), R)&\simeq \Hom_R(\Hom_R(P,R), R)\simeq \Hom_R(\Hom_B(P, B), \Hom_B(P, P))\\&\simeq \Hom_R(\Hom_B(P, B), \Hom_B(P, -)) P\simeq P.
	\end{align}
	So, the claim follows.
	
	We have shown also that for any $L\in \mathcal{M}(A)$, \begin{align*}
		\beta\circ \alpha (L)=\beta(\im \tau_L, \Hom_R(\Hom_A(L, A), R))=\im \tau_L\otimes_{\End_A(\im \tau_L)^{op}} \Hom_R(\Hom_A(L, A), R)\simeq L
	\end{align*}as $A$-modules. Hence, $\beta\circ \alpha=\id$. Thus, $\alpha$ and $\beta$ are bijections.
\end{proof}

\begin{Cor}(see \citep[Second part of Proposition 4.10]{Rouquier2008})
	For any $L\in \mathcal{M}(A)$, the canonical functor $A/\im\tau_L\m\rightarrow A\m$ induces an equivalence between $A/\im\tau_L\m$ and the full subcategory of $A\m$ whose objects $M$ satisfy $\Hom_A(L, A)=0$. \label{full subcategory in terms of split module}
\end{Cor}
\begin{proof}
	For any $L\in \mathcal{M}(A),$ let $J=\im \tau_L$. $J$ is ideal and $J=J^2$. Hence, by Lemma \ref{im tau is ideal}, for $M\in A\m$, $M\in A/J\m$ if and only if $\Hom_A(J, M)=0$.
	
	But $L\simeq J\otimes_B P$ for some progenerator $P$ of $B$ by Proposition \ref{bijection splits}. By Tensor-Hom adjunction, $\Hom_A(L, M)\simeq \Hom_B(P, \Hom_A(J, M))$.
	
	We claim that $\Hom_A(L, M)\simeq \Hom_B(P, \Hom_A(J, M))=0$ if and only if $\Hom_A(J, M)=0$.
	
	Assume that $\Hom_A(J, M)=0$, then it is clear that $\Hom_A(L, M)\simeq \Hom_B(P, \Hom_A(J, M))=0$.
	Reciprocally, assume that $\Hom_A(L, M)\simeq \Hom_B(P, \Hom_A(J, M))=0$. If $\Hom_A(J, M)\neq 0$, then there exists a non-zero $B$-epimorphism $P^t\rightarrow \Hom_A(J, M)$. This would imply that $\Hom_B(P, \Hom_A(J, M))\neq 0$. The result follows.
\end{proof}

	A full subcategory $\mathcal{A}$ of an abelian category $\mathcal{B}$ is known as \textbf{Serre subcategory} if for any exact sequence in $\mathcal{B}$
	\begin{align*}
		0\rightarrow X\rightarrow M\rightarrow Y\rightarrow 0
	\end{align*}
	$M\in \mathcal{A}$ if and only if $X, Y\in \mathcal{A}$.

Hence, a Serre subcategory is a subcategory closed under extensions, submodules and quotients.
\begin{Cor}\label{serresubcategory}
	For any $L\in \mathcal{M}(A)$, let $J=\im \tau_L$. Then, $A/J\m$ is a Serre subcategory of $A\m$.
\end{Cor}
\begin{proof}
	Let $0\rightarrow X\rightarrow M\rightarrow Y\rightarrow 0$ be an exact sequence of $A$-modules.
	Applying the functor $\Hom_A(L, -)$ yields
	\begin{align*}
		0\rightarrow \Hom_A(L, X)\rightarrow \Hom_A(L, M)\rightarrow \Hom_A(L, Y)\rightarrow 0.
	\end{align*}
	Thus, $\Hom_A(L, M)=0$ if and only if $\Hom_A(L, X)=\Hom_A(L, Y)=0$. By Corollary \ref{full subcategory in terms of split module}, the result follows.
\end{proof}

\subsection{Picard Group and invertible modules}

To write this bijection in terms of split heredity ideals instead of pairs $(J, P)$ we need the notion of invertible module.
The theory of invertible modules can be studied with more detail, for example, in \citep{Faith1973}.

\begin{Def}
	Let $R$ be a commutative ring. \label{invertiblemodules}
	A module $M$ is called \textbf{invertible} if the functor $M\otimes_R -\colon R\m\rightarrow R\m$ is an equivalence of categories.
\end{Def}

\begin{Prop}Let $R$ be a commutative Noetherian ring. Let $M$ be a finitely generated $R$-module.
	The following assertions are equivalent.
	\begin{enumerate}[(i)]
		\item $M$ is invertible;
		\item There exists an $R$-module $N$ such that $M\otimes_R N\simeq R$;
		\item $M_\mathfrak{p}\simeq R_\mathfrak{p}$ for all prime ideals $\mathfrak{p}$ of $R$;
		\item $M_\mi\simeq R_\mi$ for all maximal ideals $\mi$ of $R$.
	\end{enumerate}
\end{Prop}\label{invertiblemodulescharacterization}
\begin{proof}
	See for example \citep[12.13]{Faith1973} and \citep[p.114, 115]{zbMATH01194716}.
\end{proof}
Note that for $L, L'$ invertible $R$-modules, exists $N, N'$ such that $L\otimes_R N\simeq R$ and $L'\otimes_R N'\simeq R$. So, \begin{align*}
	L\otimes_R L'\otimes N\otimes N'\simeq L\otimes_R N\otimes_R L'\otimes_R N'\simeq R\otimes_R R\simeq R.
\end{align*}Hence, $L\otimes_R L'$ is invertible.  
The isomorphism classes of invertible $R$-modules together with the tensor product form a group. This group is called the \textbf{Picard group} of the ring $R$. We denote it by $Pic(R)$. The unit is the equivalence class of the regular module $R$ and the inverse of $M$ is $\Hom_R(M, R)$.

\begin{Example}\label{picardgroupofalocalring}
	The Picard group of a local ring is trivial.
	
	Let $M\in Pic(R)$. Then, $M_\mi\simeq R_\mi$ for the maximal ideal $\mi$ of $R$, hence $M$ is projective. Since $R$ is local, $M$ is free, hence $M\simeq R^n$ for some $n$. On the other hand, $M_\mi\simeq R_\mi^n$ for the maximal ideal $\mi$. Therefore, we must have $n=1$, so $M$ is isomorphic to $R$.
\end{Example}
Now we can see that the Picard group $Pic(R)$ acts on $\mathcal{M}(A)$.

\begin{Lemma}\label{actionofpicardgroup}
	Let $F\in Pic(R),\ L\in\mathcal{M}(A)$. Then, $L\otimes_R F\in \mathcal{M}(A)$. Moreover, this gives an action of $Pic(R)$ on $\mathcal{M}(A)$. 
\end{Lemma}
\begin{proof}
	Let $F\in Pic(R),\ L\in\mathcal{M}(A)$. By Lemma \ref{split modules change ring}, $L_\mi\in \mathcal{M}(A_\mi)$ for each maximal ideal $\mi$ of $R$.
	Note that for each maximal ideal $\mi$ of $R$
	\begin{align*}
		(L\otimes_R F)_\mi\simeq L_\mi\otimes_{R_\mi} F_\mi\simeq L_\mi\otimes_{R_\mi} R_\mi\simeq L_\mi\in \mathcal{M}(A_\mi).
	\end{align*}
	Again, by Lemma \ref{split modules change ring}, $L\otimes_R F\in \mathcal{M}(A)$. Since $R\otimes_R L\simeq L$ and $(F_1\otimes_R F_2)\otimes_R L\simeq F_1\otimes_R(F_2\otimes_R L)$, the second claim follows.
\end{proof}

Note that two elements in $L, L'\in \mathcal{M}(A)$ are in the same orbit if and only if there exists $F\in Pic(R)$ such that $L'\simeq L\otimes_R F$ as $A$-modules.
We denote by $\mathcal{M}(A)/Pic(R)$ the set of orbits of $\mathcal{M}(A)$ under the action of $Pic(R)$.

\begin{Prop}\label{bijectionsplitheredity}
	There is a bijection from $\mathcal{M}(A)/Pic(R)$ to the set of split heredity ideals of $A$. More precisely, \begin{align*}
		\delta\colon \mathcal{M}(A)/Pic(R)\rightarrow \{\text{split heredity ideals of } A\}, \ L\mapsto \im \tau_L\\
		\vartheta\colon \{\text{split heredity ideals of } A\}\rightarrow \mathcal{M}(A)/Pic(R), \ J \mapsto J\otimes_B P
	\end{align*} where $B=\End_A(J)^{op}$ and $P$ is an arbitrary $B$-progenerator that satisfies $\End_B(P)^{op}\simeq R$. 
\end{Prop}
\begin{proof}
	Consider $L$ and $L\otimes_R F$, $F\in Pic(R)$. For every maximal ideal $\mi$ of $R$,
	\begin{align*}
		(L\otimes_R F\otimes_R \Hom_A(L\otimes_R F, A))_\mi\simeq L_\mi\otimes_{R_\mi} F_\mi\otimes_{R_\mi} \Hom_{A_\mi}(L_\mi\otimes_{R_\mi} F_\mi, A_\mi)\\\simeq L_\mi\otimes_{R_\mi} R_\mi\otimes_{R_\mi} \Hom_{A_\mi}(L_\mi\otimes_{R_\mi} R_\mi, A_\mi) \simeq L_\mi\otimes_{R_\mi} \Hom_{A_\mi}(L_\mi, A_\mi) \simeq (L\otimes_R \Hom_A(L, A))_\mi
	\end{align*}
	Hence, ${\tau_{L\otimes_R F}}_\mi={\tau_L}_\mi$ for every maximal ideal $\mi$ of $R$. Therefore, $(\im \tau_{L\otimes_R F})_\mi\simeq (\im \tau_L)_\mi$ for every maximal ideal $\mi$ of $R$. Since $(\im \tau_{L\otimes_R F}), \im \tau_L\subset A$ it follows that $\im \tau_{L\otimes_R F}=\im \tau_L$. So, $\delta$ is well defined. 
	
	Now we have to see that the image of $\vartheta$ is independent of the choice of $P$. Consider $P$ and $Q$ are $B$-progenerators such that $\End_B(P)^{op}\simeq R$ and $\End_B(Q)^{op}\simeq R$. Then, \begin{align}
		P\otimes_B \Hom_B(P, B)\simeq R \ \text{ as } R\text{-modules and }\\
		Q\otimes_B \Hom_B(Q, B)\simeq R \ \text{ as } R\text{-modules}.
	\end{align} Fix $P'=\Hom_B(P, B)$ and $Q'=\Hom_B(Q, B)$. Since $P$ and $Q$ are generators $Q'\otimes_R Q\simeq B$ and $P'\otimes_R P\simeq B$ as $(B, B)$-bimodules (see for example \citep[Chapter 12, pages 447-453]{Faith1973} ). It follows as left $A$-modules,
	\begin{align}
		J\otimes_B P\simeq J\otimes_B (B\otimes_B P)\simeq J\otimes_B (Q'\otimes_R Q)\otimes_B P\simeq (J\otimes_B Q)\otimes_R (Q'\otimes_B P)
	\end{align}
	Now $Q'\otimes_B P\in Pic(R)$. In fact, \begin{align*}
		(Q'\otimes_B P)\otimes_R (P'\otimes_B Q)\simeq Q'\otimes_B(P\otimes_R P')\otimes_B Q\simeq Q'\otimes_B B\otimes_B Q\simeq Q'\otimes_B Q\simeq R.
	\end{align*}
	Hence, $J\otimes_B P= J\otimes_B Q $ in $\mathcal{M}(A)$. Therefore, $\vartheta$ is well defined. Recall the maps $\alpha$ and $\beta$ from Proposition \ref{bijection splits}. Notice that $\delta$ is the projection onto the first coordinate of the map $\alpha$. Denote this projection by $\pi$. On the other hand, $\vartheta(J)=\beta(J, P)$ for some $B$-progenerator $P$. 
	
	Therefore, $\vartheta\circ \delta (L)= \vartheta (\im \tau_L)=\beta(\im \tau_L, \Hom_R(\Hom_A(L, A), R))=\beta\circ \alpha(L)=L$ for any $L\in \mathcal{M}(A)/Pic(R)$ and
	$\delta \circ \vartheta (J)=\delta \circ \beta(J, P)= \pi\circ \alpha\circ \beta(J, P)=\pi(J, P)=J$ for any split heredity $J$. Thus, both $\vartheta$ and $\delta$ are bijections.
\end{proof}

\section{Constructing filtrations in integral split highest weight categories}\label{Filtrations in split highest weight categories} 

We will now provide the details of how filtrations can be constructed in integral split highest weight categories.

\begin{Lemma}\label{Lemmaforfiltrationsone}
	Let $F$ be a free $R$-module of finite rank and let $L, Q\in A\m$ with $\End_A(L)\simeq R$. Let $f\colon F\rightarrow \Ext^1_A(Q, L)$ be surjective.
Then,	there is an isomorphism $$\Hom_R(F, \Ext_A^1(Q, L))\rightarrow \Ext_A^1(Q, L\otimes_R DF).$$ Moreover, the image of $f$ in $\Ext_A^1(Q, L\otimes_R DF)$\begin{align*}
		0\rightarrow L\otimes_R DF\rightarrow X\rightarrow Q\rightarrow 0
	\end{align*}satisfies $\Ext_A^1(X, L)=0$. \label{Lemma for filtrations1}
\end{Lemma}
\begin{proof}
	Note first that there is such isomorphism. Let $Q^{\bullet}$ be a projective $A$-resolution for $Q$. Then, \begin{align*}
		\Hom_R(F, \Ext_A^1(Q, L))&=\Hom_R(F, H^1(\Hom_A(Q^{\bullet}, L)))\\&\simeq H^1(\Hom_R(F, \Hom_A(Q^{\bullet}, L))), \ \text{since }\Hom_R(F, -) \text{ is exact}\\
		&\simeq H^1(\Hom_R(F, R)\otimes_R \Hom_A(Q^{\bullet}, L)), \ \text{since } F\in R\proj\\
		&\simeq H^1(\Hom_A(Q^{\bullet}, L\otimes_R DF))= \Ext_A^1(Q, L\otimes_R DF), \text{since } Q^{\bullet} \text{is a projective chain.} 
	\end{align*}
	Now we need to know how to obtain explicitly the image of $f$. Consider $F=R^n$, and $\{e_i, 1\leq i\leq n\}$ a basis. Then, $
	\left\lbrace f(e_i)\colon \quad 0\rightarrow L\rightarrow X_i\rightarrow Q\rightarrow 0 \quad|\quad 1\leq i\leq n\right\rbrace
	$ is an $R$-generator set for $\Ext_A^1(Q, L)$. Note that the previous isomorphism can be viewed as \begin{align*}
		\Hom_R(F, \Ext_A^1(Q, L))\rightarrow \Ext_A^1(Q, L)^n\rightarrow \Ext_A^1(Q, L^n)\rightarrow \Ext_A^1(Q, \Hom_R(F, L)).
	\end{align*}
	Consider a projective presentation for $Q$, $0\rightarrow N\xrightarrow{k} M\xrightarrow{\pi} Q\rightarrow 0 $. Apply the functors $\Hom_A(-, L)$ and $\Hom_A(-, L^n)$. We obtain a commutative diagram
	\begin{center}
		\begin{tikzcd}
			\Hom_A(N, L)^n\arrow[r, "\partial^n"]\arrow[d, "\simeq"]&\Ext_A^1(Q, L)^n\arrow[d, "\simeq"]\arrow[r]&0\\
			\Hom_A(N, L^n)\arrow[r, "\partial_n"]&\Ext_A^1(Q, L^n)\arrow[r]& 0
		\end{tikzcd}.
	\end{center} 
	For every $i$, since $\partial$ is surjective there exists a map $s_i\in \Hom_A(N, L)$ such that $\partial(s_i)=f(e_i)$. This map relates with $f(e_i)$ by the following pushout diagram:
	\begin{center}
		\begin{tikzcd}
			0\arrow[r]& N\arrow[r, "k"]\arrow[d, "s_i"]& M\arrow[r, "\pi"]\arrow[d, "p_i"]& Q\arrow[r] \arrow[d, equal]& 0\\
			0\arrow[r]& L\arrow[r]& X_i\arrow[r]& Q\arrow[r] & 0\\[-0.5cm]
			& &\arrow[u, equal]  _{pushout(s_i, k)} & &
		\end{tikzcd}.
	\end{center} 
	This description of the map $\partial$ can be found with more detail in any book of homological algebra (see e.g \citep[Theorem 2.4]{Hilton1997}).
	
	Now the image of $f$ in $\Ext_A^1(Q, L^n)$ is just $\partial_n(s_1, \ldots, s_n)$. Hence, it is given by the following diagram
	\begin{equation} 
		\begin{tikzcd}
			0\arrow[r]& N\arrow[r, "k"]\arrow[d, "({s_1}{,}\ldots{,}{s_n})"]& M\arrow[r, "\pi"]\arrow[d, "p"]& Q\arrow[r] \arrow[d, equal]& 0\\
			0\arrow[r]& L^n\arrow[r, "g"]& X\arrow[r]& Q\arrow[r] & 0\\[-0.5cm]
			& &\arrow[u, equal]  _{pushout(({s_1}{,}\ldots{,}{s_n}), k)} & &
		\end{tikzcd}. \label{qheqtikz21}
	\end{equation} 
	Now applying $\Hom_A(-, L)$ to the bottom row of (\ref{qheqtikz21}) yields 
	\begin{align}
		\Hom_A(L^n, L)\xrightarrow{\partial'}\Ext_A^1(Q, L)\rightarrow \Ext_A^1(X, L) \label{s1}
	\end{align}
	Note that $\Hom_A(L^n, L)\simeq \Hom_A(L, L)^n\simeq R^n=F$. Denote this isomorphism by $h\colon F\rightarrow \Hom_A(L^n, L)$. We claim that $f=\partial' \circ h$.
	
	Consider $\pi_j\colon L^n\rightarrow L$ the canonical epimorphism. $\partial'(\pi_j)$ is given by
	\begin{center}
		\begin{tikzcd}
			0\arrow[r]& L^n\arrow[r, "g"]\arrow[d, "\pi_j"]& X\arrow[r]\arrow[d]& Q\arrow[r] \arrow[d, equal]& 0\\
			0\arrow[r]& L\arrow[r]& Y_j\arrow[r]& Q\arrow[r] & 0\\[-0.5cm]
			& &\arrow[u, equal]  _{pushout(\pi_j, g)} & &
		\end{tikzcd}.
	\end{center} 
	Now consider the diagram
	\begin{center}
		\begin{tikzcd}[every arrow/.append style={shift left}]
			0\arrow[r]& N\arrow[r, "k"]\arrow[d, swap, "({s_1}{,}\ldots{,}{s_n})"]\arrow[dd, bend right=200, "s_j"{yshift=-15pt}]& M\arrow[r, "\pi"]\arrow[d, "p"]& Q\arrow[r] \arrow[d, equal]& 0\\
			0\arrow[r]& L^n\arrow[r, "g"]\arrow[d, "\pi_j", swap]& X\arrow[r]\arrow[d]& Q\arrow[r] \arrow[d, equal]& 0\\
			0\arrow[r]& L\arrow[r]& Y_j\arrow[r]& Q\arrow[r] & 0\\[-0.5cm]
		\end{tikzcd}.
	\end{center}
	
	So, the external diagram is a pushout as well. In fact, $Y_j$ is the pushout of $(s_j, k)$. By the universal property of pushouts, it follows that the exact sequences $0\rightarrow L\rightarrow Y_i\rightarrow Q\rightarrow 0$ and $0 \rightarrow L\rightarrow X_i\rightarrow Q \rightarrow 0$ are equivalent. Therefore, $\partial'(\pi_j)=f(e_j)$. So, the claim follows, and hence, $\partial'$ is surjective. By the exactness of (\ref{s1}), it follows that $\Ext_A^1(X, L)=0$.
\end{proof}

\begin{Lemma}\label{Lemmaforfiltrationstwo}
	Let $F$ be a free $R$-module and let $L, T\in A\m$ with $\End_A(L)\simeq R$. Let ${g\colon F\rightarrow \Ext^1_A(L, T)}$ be surjective.
	There is an isomorphism $\Hom_R(F, \Ext_A^1(L, T))\rightarrow \Ext_A^1(L\otimes_R F, T)$. Then, the image of $f$ in \mbox{$\Ext_A^1(L\otimes_R F, T)$}\begin{align*}
		0\rightarrow T\rightarrow Y\rightarrow L\otimes_R F\rightarrow 0
	\end{align*}satisfies $\Ext_A^1(L, Y)=0$.
\end{Lemma}
\begin{proof}
	The proof is the dual version of the previous one. For the sake of completeness, we will write a proof. The isomorphism exists:
	Let $T^{\bullet}$ be an injective resolution for $T$.
	\begin{align*}
		\Hom_R(F, \Ext_A^1(L, T))&=\Hom_R(F, H^1(\Hom_A(L, T^{\bullet})))\simeq H^1(\Hom_R(F, \Hom_A(L, T^{\bullet})))\\
		&\simeq H^1(\Hom_A(F\otimes_R L, T^{\bullet})), \ \text{by Tensor-Hom adjunction}\\
		&\simeq \Ext^1_A(F\otimes_R L, T).
	\end{align*}
	Consider $F=R^n$, and let $\{e_i, 1\leq i\leq n\}$ be a basis. Then,
	\begin{align*}
		\left\lbrace 0\rightarrow T\rightarrow Y_i\rightarrow L\rightarrow 0| 1\leq i\leq n\right\rbrace .
	\end{align*}is an $R$-generator set for $\Ext_A^1(L, T)$. Consider an injective presentation for $T$ \begin{align}
		0\rightarrow T\xrightarrow{k}M\xrightarrow{\pi}Q\rightarrow 0.
	\end{align} Applying $\Hom_A(L, -)^n$ and $\Hom_A(L^n, -)$ yields the commutative diagram
	\begin{center}
		\begin{tikzcd}
			\Hom_A(L, Q)^n\arrow[r, "\partial^n"]\arrow[d, "\simeq"]&\Ext_A^1(L, T)^n\arrow[d, "\simeq"]\arrow[r]&0\\
			\Hom_A(L^n, Q)\arrow[r, "\partial_n"]&\Ext_A^1(L^n, T)\arrow[r]& 0
		\end{tikzcd}.
	\end{center} 
	For every $i$, since $\partial$ is surjective there exists a map $s_i\in \Hom_A(L, Q)$ such that $\partial(s_i)=g(e_i)$,
	\begin{center}
		\begin{tikzcd}
			0\arrow[r]& T\arrow[r, "k"]\arrow[d, equal]& M\arrow[r, "\pi"]& Q\arrow[r] & 0\\
			0\arrow[r]& T\arrow[r]& Y_i\arrow[r]\arrow[u]& L\arrow[r]\arrow[u, "s_i"] & 0\\[-0.5cm]
			& &\arrow[u, equal]  _{pullback(s_i, \pi)} & &
		\end{tikzcd}.
	\end{center} 
	Now the image of $g$ in $\Ext_A^1(L^n, T)$ is just $\partial_n(s_1, \ldots, s_n)$. Hence, it is given by the following diagram
	\begin{center}
		\begin{tikzcd}
			0\arrow[r]& T\arrow[d, equal]\arrow[r, "k"]& M\arrow[r, "\pi"]& Q\arrow[r] & 0\\
			0\arrow[r]& T\arrow[r]& Y\arrow[r, "h"]\arrow[u]& L^n\arrow[u, "({s_1}{,}\ldots{,}{s_n})"]\arrow[r] & 0\\[-0.5cm]
			& &\arrow[u, equal]  _{pullback(({s_1}{,}\ldots{,}{s_n}), \pi)} & &
		\end{tikzcd}.
	\end{center} 
	Now applying $\Hom_A(L, -)$ to the image of $g$ yields
	\begin{align}
		\Hom_A(L, L^n)\xrightarrow{\partial'}\Ext_A^1(L, T)\rightarrow \Ext_A^1(L, Y)\label{s2}.
	\end{align}Let $w$ denote  the canonical isomorphism $w\colon F\rightarrow \Hom_A(L, L^n)$.
	Now computing $\partial'(k_j)$ for the canonical monomorphisms $k_j\colon L\rightarrow L^n$ and comparing with the pullback diagram that gives the image $g$, it follows again that the induced external diagram is again a pullback. By the universal property of pullbacks, it follows that $g=\partial'\circ w$. So, we conclude by the exactness of (\ref{s2}) that $\Ext_A^1(L, Y)=0$.
\end{proof}

Let $L\in \mathcal{M}(A)$. The following result shows how to reconstruct the subcategories $A\proj$ and $A/J\proj$ for $J=\im \tau_L$ from one another.

\begin{Lemma}(see \citep[Lemma 4.9]{Rouquier2008})
	Let $L\in \mathcal{M}(A)$ and let $J=\im \tau_L$.
	\begin{enumerate}[(i)]
		\item Given $P\in A\proj$, then $\im \tau_{L, P}=JP$ and $P/JP\in A/J\proj$.
		\item Let $Q\in A/J\proj$. Let $F$ be a free $R$-module and $f\colon F\rightarrow \Ext_A^1(Q, L)$ surjective. 
		
		Let $0\rightarrow L\otimes_R DF\xrightarrow{g} P\xrightarrow{h} Q\rightarrow 0$ be the extension in $\Ext_A^1(Q, L\otimes_R DF)$ corresponding to $f$ via the isomorphism $\Hom_R(F, \Ext_A^1(Q, L)) \rightarrow \Ext_A^1(Q, L\otimes_R DF)$. Then, $P\in A\proj$.
	\end{enumerate} \label{AandAmoduloJproj}
\end{Lemma}
\begin{proof}
	If $P=A$, then $\im \tau_{L, A}=\im \tau_L=J=JA$ and $P/JP=A/JA=A/J\in A/J\proj$.
	
	If $P=A^s$ for some $s>0$, then $\im \tau_{L, P}=(\im \tau_{L})^{\oplus s}=J^{\oplus s}=JP$. In such a case, $P/JP=A^{\oplus s}/J^{\oplus s}\simeq  (A/J)^{\oplus s}\in A/J\proj$.
	
	Finally, assume that $P$ is a summand of $A^s$, say $A^s\simeq P\oplus K$ for some $A$-module $K$ and some $s>0$. Then, $JP\oplus JK=JA^s=\im \tau_{L, A^s}=\im \tau_{L, P}\oplus \im \tau_{L, K}$. Since $\im \tau_{L, P}\subset P$, it follows $\im \tau_{L, P}=JP$.
	Moreover, $(P\oplus K)/(JP\oplus JK)\simeq P/JP\oplus K/JK$, hence $P/JP$ is a summand of $A^s/JA^s\in A/J\proj$, thus (i) follows.
	
	Assume $Q=(A/J)^n$ for some $n>0$. Consider the canonical epimorphism $\pi\colon A^n\rightarrow Q$. Since $A^n$ is projective over $A$, $\pi$ factors through $h$, that is, there exists $\phi\colon A^n\rightarrow P$ such that $\pi=h\circ \phi$. Let $\psi=\phi+ g\colon A^n\oplus L\otimes_R DF\rightarrow P$. Define $N=\ker \psi$.
	
	\textit{Claim 1.} $\psi$ is surjective.   
	
	Let $p\in P$. Then, $h(p)\in Q$. Since $\pi$ is surjective, there exists $x\in A^n$ such that $\pi(x)=h(p)$. Note that $h\circ \phi(x)=\pi(x)=h(p)$. Thus, $p-\phi(x)\in \ker h=\im g=L\otimes_R DF$. So, the claim follows.
	
	\textit{Claim 2.} $N\subset J^{\oplus n}\oplus L\otimes_R DF$.
	
	Notice that $(x, y)\in N$ if and only if $0=\psi(x, y)=\psi(x)+g(y)$ if and only if $\psi(x)=g(-y)$. In particular, $\pi(x)=h\circ \phi(x)= 0$, hence $x\in J^{\oplus n}$. So, the claim follows.
	
	Now note that for any $x\in J^{\oplus n}$, $h\circ \phi(x)=0.$ So, $\phi(x)\in\ker h=\im g$. Since $g$ is a monomorphism \linebreak[4]$\im g\simeq L\otimes_R DF\in A\proj$. 
	Therefore, the following sequence is $A$-split exact 
	\begin{align*}
		0\rightarrow N \xrightarrow{z} J^{\oplus n}\oplus L\otimes_R DF\xrightarrow{w} \im g\rightarrow 0 
	\end{align*} where $z(x, y)=(x, y)$ and $w(x, y)=\phi(x)+g(y)$. Since $J^{\oplus n}\oplus L\otimes_R DF$ is projective over $A$, we obtain that \linebreak[4]$N\oplus \im g\simeq J^{\oplus n}\oplus L\otimes_R DF $. Furthermore, $J^{\oplus n}\simeq (L\otimes_R \Hom_A(L, A))^n\simeq L\otimes_R V$, $V\in R\proj$. Hence,
	\begin{align*}
		L\otimes_R(V\oplus DF)\simeq L\otimes_R V \oplus L\otimes_R DF\simeq N\oplus \im g.
	\end{align*}
	
	By Lemma  \ref{Lemma for filtrations1}, $\Ext_A^1(P, L)=0$. Hence, $\Ext_A^1(P, L\otimes_R (V\oplus DF))=0$ as $V\oplus DF\in R\proj$. In particular, $\Ext_A^1(P, N)=0$. Thus, the exact sequence
	\begin{align*}
		0\rightarrow N\xrightarrow{\psi} A^n\oplus L\otimes_R DF\rightarrow P\rightarrow 0
	\end{align*}splits over $A$. Thus, it follows that $P$ is projective.
	
	Now assume that exists $n$ such that $Q_0=(A/J)^n\simeq Q\oplus Q_1$. Consider a free $R$-presentation for \linebreak[4]$\Ext_A^1(Q_1, L), g\colon R^s\rightarrow \Ext_A^1(Q_1, L)$.  Therefore, $f\oplus g\colon F\oplus R^s\rightarrow \Ext_A^1(Q_0, L)$ is surjective.
	The image of $g$ in $\Ext_A^1(Q_1, L\otimes_R D(R^s))$ is \begin{align}
		0\rightarrow L\otimes_R D(R^s)\rightarrow P_1\rightarrow Q_1\rightarrow 0.
	\end{align}So, the image of $f\oplus g$ in $\Ext_A^1(Q_0, L\otimes_R D(F\oplus R^s))$ is \begin{align}
		0\rightarrow L\otimes_R D(F\oplus R^s)\rightarrow P\oplus  P_1\rightarrow Q_0\rightarrow 0.
	\end{align}By the previous case, $P\oplus P_1\in A\proj$. So, we conclude that $P\in A\proj$, so (ii) follows.
\end{proof}

\begin{Lemma}(see \citep[Lemma 4.12]{Rouquier2008})
	Let $A$ be a projective Noetherian $R$-algebra. Let ${\{\St(\l)\colon \l\in\L \}}$ be a finite set of modules in $A\m$ together with a poset structure on $\L$. Let $\alpha$ be a maximal element in $\L$. Then, $(A\m, \{\Delta(\lambda)_{\lambda\in \Lambda}\})$ is a split highest weight category if and only if $\St(\alpha)\in \mathcal{M}(A)$ and $(A/J\m, \{\Delta(\lambda)_{\lambda\in \Lambda\backslash \{\alpha \} } \})$  is a split highest weight category, where $J=\im \tau_{\St(\alpha)}$.\label{splithwcinduction}
\end{Lemma}
\begin{proof}
	Let $(A\m, \{\Delta(\lambda)_{\lambda\in \Lambda}\})$ be a split highest weight category. Let $\alpha$ be a maximal element in $\L$. By (iv) of Definition \ref{splithwc}, $\ker \pi_\alpha=0$, so $\St(\alpha)\simeq P(\alpha)\in A\proj$. By \ref{splithwc}(v), $\End_A(\St(\alpha))\simeq R$. As $\St(\alpha)$ is faithful over $\End_A(\St(\alpha))^{op}$, it follows that $\St(\alpha)$ is faithful over $R$. Let $\l\in \L\backslash \{\alpha\}$. By (iv) of Definition \ref{splithwc}, $C(\l)$ has a finite filtration by modules of the form $\St(\mu)\otimes_R U_\mu$ with $U_\mu\in R\proj$ and $\mu>\l$. 
	In particular, $\St(\alpha)$ can appear. 
	Note that $\St(\alpha)$ is projective over $A$, so we can rearrange the filtration so that all modules of the form $\St(\alpha)\otimes_R U_\alpha$, $U_\alpha\in R\proj$, appear at the bottom of the filtration. In fact, consider the filtration \begin{align*}
		0\subset X_1\subset \cdots \subset X_i\subset \cdots\subset X_n=C(\l).
	\end{align*} Assume that $X_i/X_{i-1}\simeq \St(\alpha)\otimes_R U_\alpha\in A\proj$ for some $U_\alpha\in R\proj$. Thus, $X_i\simeq \St(\alpha)\otimes_R U_\alpha\bigoplus X_{i-1}$. So, $X_i/(\St(\alpha)\otimes_R U_\alpha)\simeq X_{i-1}$, and hence the filtration until $X_{i-1}$ can be written in the form
	\begin{align*}
		0\subset \tilde{X}_1/(\St(\alpha)\otimes_R U_\alpha)\subset \cdots\subset \tilde{X}_{i-1}/(\St(\alpha)\otimes_R U_\alpha)=X_{i-1}.
	\end{align*}
	Notice that $\tilde{X}_{j}/\tilde{X}_{j-1}\simeq \tilde{X}_{j}/(\St(\l)\otimes_R U_\alpha)/\tilde{X}_{j-1}/(\St(\l)\otimes_R U_\alpha)\simeq X_j/X_{j-1}$. Thus, we obtain a filtration\begin{align}
		0\subset \St(\alpha)\otimes_R U_\alpha\subset \tilde{X}_1\subset \cdots\subset \tilde{X}_{i-1}=X_i\subset X_{i+1}\subset \cdots \subset C(\l).
	\end{align}
	Hence, we have an exact sequence
	\begin{align}
		0\rightarrow \St(\alpha)\otimes_R U_\l\rightarrow C(\l)\rightarrow X(\l)\rightarrow 0,
		\label{eq28}\end{align} where the projective $R$-module $U_\l$ encodes all the occurrences of $\St(\alpha)$ in the filtration of $C(\l)$, and consequently in the filtration of $P(\l)$. Thus, $X(\l)$ has a filtration by modules of the form $\St(\mu)\otimes_R U_\mu$ with $\mu>\l$, $\mu\neq \alpha$.
	Applying $\Hom_A(\St(\alpha), -)$ to the filtration of $P(\l)$, we get the exact sequence
	\begin{align}
		0\rightarrow \Hom_A(\St(\alpha), C(\l))\rightarrow \Hom_A(\St(\alpha), P(\l))\rightarrow \Hom_A(\St(\alpha), \St(\l))\rightarrow 0.
	\end{align} By Condition (ii) of split highest weight category, we have $\Hom_A(\St(\alpha), \St(\l))=0$, since $\alpha$ is maximal. Hence, \begin{align*}
		\Hom_A(\St(\alpha), C(\l))\simeq \Hom_A(\St(\alpha), P(\l)).
	\end{align*} 
	Applying $\Hom_A(\St(\alpha), -)$ to (\ref{eq28}), we get the exact sequence
	\begin{align}
		0\rightarrow \Hom_A(\St(\alpha), \St(\alpha)\otimes_R U_\l)\rightarrow \Hom_A(\St(\alpha), C(\l))\rightarrow \Hom_A(\St(\alpha), X(\l))\rightarrow 0.
	\end{align}
	We have $\Hom_A(\St(\alpha), X(\l))=0$. In fact, if $\Hom_A(\St(\alpha), X(\l))\neq 0$, then by induction on the size of the filtration of $X(\l)$ it would exist $\St(\mu)$ with $\mu\neq\alpha$ such that $\Hom_A(\St(\alpha), \St(\mu))\neq 0$. Since $\alpha$ is maximal, this cannot happen.
	
	So, $
	U_\l\simeq\Hom_A(\St(\alpha), P(\l)).
	$
	
	As we have discussed $P=\sumSt P(\l)$ is a progenerator for $A\m$. 
	Put $P_0:= \St(\alpha)\otimes_R \sumSt U_\l=\im \tau_{\St(\alpha), P}\subset P$. Thus, $P/P_0$ is an extension of $\St(\l)$ by $\sumSt X(\l)$ with $U_\alpha=R$ and $X(\alpha)=0$. So, $\Hom_A(\St(\alpha), P/P_0)=0$. Since all standard modules are projective over $R$, we have that $P/P_0$  is projective over $R$. By Proposition \ref{split modules characterization}, it follows that $\tau_{\St(\alpha), P}$ is an $(A, R)$-monomorphism. Since $P$ is a progenerator, it follows by the proof of Proposition \ref{split modules characterization}, that $\tau_{\St(\alpha), A}$ is an $(A, R)$-monomorphism, thus $\St(\alpha)\in \mathcal{M}(A)$. 
	
	Fix $J=\im \tau_{\St(\alpha)}$. 
	Since $\Hom_A(\St(\alpha), \St(\l))=0$ for $\l\neq \alpha$ it follows that $\St(\l)\in A/J\m$ by Corollary \ref{full subcategory in terms of split module}.
	Now we will show that $(A/J\m, \{\Delta(\lambda)_{\lambda\in \Lambda\backslash \{\alpha \} } \})$ is a split highest weight category. Condition \ref{splithwc}(i) is clear.
	Since $A/J\m$ is a full subcategory of $A\m$, it follows that
	\begin{align}
		0\neq \Hom_{A/J}(\St(\l'), \St(\l''))=\Hom_A(\St(\l'), \St(\l'')).
	\end{align}By \ref{splithwc}, we get $\l'\leq \l''.$ So, Condition \ref{splithwc}(ii) for $A/J$ holds.
	In the same way,\begin{align}
		\End_{A/J}(\St(\l))=\End_A(\St(\l))\simeq R.
	\end{align}
	Let $N\in A/J\m$ satisfying $\Hom_{A/J}(\St(\l), N)=0$ for all $\l\in \Lambda\backslash \{\alpha \}$. Then, \begin{align}
		\Hom_A(\St(\l), N)=\Hom_{A/J}(\St(\l), N)=0.
	\end{align} By Corollary \ref{full subcategory in terms of split module}, $\Hom_A(\St(\alpha), N)=0$ since $N\in A/J\m$. Since $(A\m, \{\Delta(\lambda)_{\lambda\in \Lambda}\})$ is split highest weight category $N=0$, and thus Condition \ref{splithwc}(iii) holds.
	
	For any $\l\neq \alpha$, define $Q(\l)=\St(\alpha)\otimes_R U_\l$. We have that,\begin{align*}
		\im \tau_{\St(\alpha), P(\l)}\simeq \St(\alpha)\otimes_R \Hom_A(\St(\alpha), P(\l))\simeq \St(\alpha)\otimes_R U_{\l}=Q(\l).
	\end{align*}
	By Lemma \ref{AandAmoduloJproj} (i), $P(\l)/Q(\l)\in A/J\proj$. Since $Q(\l)\subset C(\l)$, it follows that the following exact sequence yields Condition \ref{splithwc}(iv)
	\begin{align}
		0\rightarrow X(\l)=C(\l)/Q(\l)\rightarrow P(\l)/Q(\l)\rightarrow \St(\l)\rightarrow 0.
	\end{align} 
	
	Conversely, assume now that $\St(\alpha)\in\mathcal{M}(A)$ and $(A/J\m, \{\Delta(\lambda)_{\lambda\in \Lambda\backslash \{\alpha \} } \})$ is a split highest weight category. 
	
	By Remark \ref{remendprojonma}, $\End_A(\St(\alpha))\simeq R$. Now by Condition \ref{splithwc}(v) of $(A/J\m, \{\Delta(\lambda)_{\lambda\in \Lambda\backslash \{\alpha \} } \})$ being a split highest weight category, $R\simeq \End_{A/J}(\St(\l))=\End_A(\St(\l))$ for $\l\neq \alpha$. Thus, condition \ref{splithwc}(v) holds for $A$. By Condition \ref{splithwc}(i) of $(A/J\m, \{\Delta(\lambda)_{\lambda\in \Lambda\backslash \{\alpha \} } \})$ being split highest weight category, each $\St(\l)$ is projective over $R$. By definition of $\St(\alpha)\in \mathcal{M}(A)$, $\St(\alpha)\in R\proj$. Thus, Condition \ref{splithwc}(i) for $A$ holds. 
	Now by Corollary \ref{full subcategory in terms of split module} and the fact that $A/J\m$ is full subcategory of $A\m$, it follows Condition \ref{splithwc}(ii) and (iii) for $A$.
	
	Since $\St(\alpha)$ is projective over $A$, we define $P(\alpha)=\St(\alpha)$. Now consider for $\l\neq \alpha$ the exact sequences provided by Condition \ref{splithwc}(iv) of $A/J$ being a split highest weight category
	\begin{align}
		0\rightarrow C'(\l)\xrightarrow{i_\l^{A/J}} P_{A/J}(\l)\xrightarrow{\pi_\l^{A/J}} \St(\l)\rightarrow 0.
	\end{align}
	Consider an $R$-free presentation for $\Ext_A^1(P_{A/J}(\l), \St(\alpha))$, say $f_\l\colon F_\l\twoheadrightarrow \Ext_A^1(P_{A/J}(\l), \St(\alpha))$.
	
	By Lemma \ref{AandAmoduloJproj} (ii), we have an exact sequence 
	\begin{align}
		0\rightarrow \St(\l)\otimes_R DF_\l\xrightarrow{k_\l} P(\l)\xrightarrow{h_\l} P_{A/J}(\l)\rightarrow 0,
	\end{align} where $P(\l)\in A\proj$. So, we have an exact sequence
	\begin{align}
		0\rightarrow C(\l)\xrightarrow{i_\l} P(\l)\xrightarrow{\pi_\l^{A/J}\circ h_\l} \St(\l)\rightarrow 0.
	\end{align} We define $\pi_\l=\pi_\l^{A/J}\circ h_\l$.
	We have the following commutative diagram
	\begin{center}
		\begin{tikzcd}
			0\arrow[r]&C''(\l)\arrow[r, "l_\l"] \arrow[d, "\exists w"]& C(\l)\arrow[r, "\exists g"]\arrow[d, hookrightarrow, "i_\l"]& C'(\l)\arrow[d, hookrightarrow, "i_\l^{A/J}"]\arrow[r]& 0\\
			0\arrow[r]&\St(\l)\otimes_R DF_\l \arrow[r, "k_\l"] & P(\l)\arrow[d, twoheadrightarrow, "\pi_\l"] \arrow[r, "h_\l"]& P_{A/J}(\l)\arrow[d, twoheadrightarrow, "\pi_\l^{A/J}"]\arrow[r]& 0\\
			&  & \St(\l)\arrow[r, equal] & \St(\l) &
		\end{tikzcd}.
	\end{center} Here some observations are in order. The existence of $g$ comes from the fact $\pi_\l^{A/J}\circ h_\l\circ i_\l=\pi_\l\circ i_\l =0$. So, $C''(\l)=\ker g$. The existence of $w$ comes from the fact $h_\l\circ i_\l\circ l_\l=i_\l^{A/J}\circ g\circ l_\l=0$. By Snake Lemma, $w$ is injective. 
	On the other hand, $\pi_\l\circ k_\l=\pi_\l^{A/J}\circ h_\l\circ k_\l=0$, so there exists $q_\l\colon \St(\l)\otimes_R DF_\l\rightarrow C(\l)$ such that $i_\l\circ q_\l=k_\l$. Now note that $
	i_\l^{A/J}\circ g\circ q_\l=h_\l\circ i_\l \circ q_\l=h_\l\circ k_\l=0.
	$ Since 	$i_\l^{A/J}$ is injective, $g\circ q_\l=0$. Thus, for every $x\in\St(\l)\otimes_R DF_\l$, $
	k_\l(x)=i_\l\circ q_\l(x)=i_\l\circ l_\l (y)=k_\l (w(y)) \text{ for some } y\in C''(\l).
	$ Thus, $w$ is an isomorphism. So, $C(\l)$ has a filtration by standard modules given by the one of $\St(\l)\otimes_R DF_\l$ on the bottom and the filtration of $C'(\l)$ on the top. So, it follows that $C(\l)$ has a  filtration by standard modules where only $\St(\mu)\otimes_R X$,  with $\mu>\l$ and $X\in R\proj$ can appear. So, we conclude that $(A\m, \{\Delta(\lambda)_{\lambda\in \Lambda}\})$ is a split highest weight category.
\end{proof}

\begin{Prop}(see \citep[Proposition 4.13]{Rouquier2008})\label{everyprojectivehasafiltration}
	Suppose $(A\m, \{\Delta(\lambda)_{\lambda\in \Lambda}\})$ is a split highest weight category over a commutative Noetherian ring $R$. Let $P\in A\proj$. Let $\St\rightarrow \{1, \ldots, n\}$, $\St_i\mapsto i$, be an increasing bijection. Then, there is a filtration \begin{align*}
		0=P_{n+1}\subset P_n\subset \cdots\subset P_1=P \quad \text{ with} \quad P_i/P_{i+1}\simeq \St_i\otimes_R U_i, \quad \text{for some} \quad U_i\in R\proj.
	\end{align*} %
\end{Prop}
\begin{proof}
	We shall proceed by induction on $|\L|=n$. Assume $n=1$. Consider $\St_1\in \mathcal{M}(A)$. Let $P\in A\proj$. By Proposition \ref{split modules characterization}, there exists $P_0=\im \tau_{\St_1, P}=\St_1\otimes_R U_1\subset P$, $U_1\in R\proj$ and $\Hom_A(\St_1, P/P_0)=0$. Thus, $P/P_0=0$. Hence, $0\subset \St_1\otimes_R U_1=P$ is a filtration with the desired properties.
	
	Assume now the result known for $|\L|=n-1$. Let $(A\m, \{\Delta(\lambda)_{\lambda\in \Lambda}\})$ be a split highest weight category with $|\L|=n$. By Lemma \ref{splithwcinduction}, $\St_n\in \mathcal{M}(A)$ and $(A/J\m, \{{\Delta_{j}}_{j=1, \ldots, n-1}  \})$ is a split highest weight category where $J=\im \tau_{\St_n}.$
	
	Let $P\in A\proj$. By Proposition \ref{split modules characterization}, there exists $U_n\in R\proj$ such that $\im \tau_{L, P}=\St_n\otimes_R U_n$. By Lemma \ref{AandAmoduloJproj} (i), $JP=\im \tau_{L, P}=\St_n\otimes_R U_n$ and $P/JP\in A/J\proj$. By induction, there is a filtration for $P/JP$:
	\begin{align*}
		0=P_n'\subset P_{n-1}'\subset \cdots\subset P_1'=P/JP, \text{ with } P_i/P_{i+1}\simeq \St_i\otimes_R U_i, \ i=1, \ldots, n-1.
	\end{align*}
	As the submodules of $P/JP$ are exactly the submodules of $P$ which contain $JP$, we get a filtration
	\begin{align*}
		0=P_{n+1}\subset P_n\subset P_{n-1}\subset \cdots\subset P_1=P,
	\end{align*} where $P_i'\simeq P_i/JP$ and $P_n=JP$. Note that $P_i/P_{i+1}\simeq (P_i'/JP)/(P_{i+1}'/JP)$ for $i=1, \ldots n$. Thus, the claim follows.
\end{proof}

\begin{Prop}\label{usualfiltrationisthisoneaswell}
	Let $(A\m, \{\Delta(\lambda)_{\lambda\in \Lambda}\})$ be a split highest weight category. Let $\St\rightarrow \{1, \ldots, n\}$, $\St_i\mapsto i$, be an increasing bijection. If $M\in \mathcal{F}(\Stsim)$, then there is a filtration
	\begin{align*}
	0=F_{n+1}\subset F_n\subset \cdots\subset F_1=M \quad \text{ with} \quad F_i/F_{i+1}\simeq \St_i\otimes_R U_i, \quad \text{for some} \quad U_i\in R\proj.
\end{align*} %
\end{Prop}
\begin{proof}
	First, we will prove that we can assume, without loss of generality, that the factors in the filtration of $M$ appear in non-increasing order due to Proposition \ref{extonstandards}. In fact,
	consider a filtration \begin{align}
		0=M_{m+1}\subset M_m\subset \cdots\subset M_1=M.
	\end{align} Consider $k$ such that $M_k/M_{k+1}\simeq \St_i\otimes_R U_k$, $M_{k+1}/M_{k+2}\simeq \St_j\otimes_R U_{k+1}$ and $i>j$.  
Since $M_{k+2}\subset M_{k+1}\subset M_{k}$, there is a canonical monomorphism $\St_j\otimes_R U_{k+1}\simeq M_{k+1}/M_{k+2}\rightarrow M_k/M_{k+2}$. As
\begin{align}
	(M_{k}/M_{k+2})/(M_{k+1}/M_{k+2})\simeq M_k/M_{k+1}\simeq \St_i\otimes_R U_k,
\end{align} there exists a short exact sequence $0\rightarrow\St_j\otimes_R U_{k+1} \rightarrow M_k/M_{k+2}\rightarrow \St_i\otimes_R U_k\rightarrow 0$. By Proposition \ref{extonstandards}, $\Ext_A^1(\St_i\otimes_R U_k, \St_j\otimes_R U_{k+1})=0$. Hence, this short exact sequence splits and we obtain a surjective map $h\colon M_k\rightarrow M_k/M_{k+2}\simeq \St_i\otimes_R U_k\oplus \St_j\otimes_R U_{k+1}\rightarrow \St_j\otimes_R U_{k+1}$. Define $\overline{M_{k+1}}:=\ker h$. Thus, $M_k/\overline{M_{k+1}}\simeq \im h=\St_j\otimes_R U_{k+1}$ and observe that $\overline{M_{k+1}}/M_{k+2}\simeq \St_i\otimes_R U_k$. In fact, the latter follows applying the Snake Lemma to the commutative diagram
	\begin{center} 
	\begin{tikzcd}
		0\arrow[r]& M_{k+2}\arrow[d, hookrightarrow] \arrow[r] & M_{k} \arrow[d, equal] \arrow[r]& \St_i\otimes_R U_k\bigoplus \St_j\otimes_R U_{k+1}\arrow[d]\arrow[r] &0\\
		0\arrow[r]& \overline{M_{k+1}}\arrow[r]&M_{k}\arrow[r]&\St_j\otimes_R U_{k+1}\arrow[r]& 0
	\end{tikzcd}.
\end{center} Therefore, $0=M_{m+1}\subset M_m\subset \cdots \subset M_{k+2}\subset \overline{M_{k+1}}\subset M_k\subset \cdots\subset M_1=M$ is still a filtration of $M$. Using, if necessary, this argument a finite number of times we can assume that $M$ has a filtration by objects of the form $\St_i\otimes_R U_i$ where the factors appear in a non-increasing order.

	Let \begin{align}
	0=M_{t+1}\subset M_t\subset M_{t-1}\subset \cdots \subset M_1=M \quad \text{where } M_j/M_{j+1}\simeq \St_{k_j}\otimes_R U_{j}, \ U_j\in R\proj
\end{align} be a filtration of $M$ where its factors appear in a non-increasing order.
Choose $\l\in \L$ maximal. So, there is a minimal index $i$ (possibly $t+1$) such that
\begin{align}
	0\subset M_t \subset M_{t-1} \subset \cdots \subset M_i \quad \text{ with } M_j/M_{j+1}\simeq \St(\l)\otimes_R U_j, \ t\geq j\geq i.
\end{align}
Using the fact that $\St(\l)$ is projective over $A$, also $\St(\l)\otimes_R U_j$ is projective and so the exact sequences \begin{align*}
	0\rightarrow M_{j+1}\rightarrow M_j\rightarrow \St(\l)\otimes_R U_j\rightarrow 0, \ i\leq j\leq t,
\end{align*} split over $A$. So, we deduce that, as $A$-modules, $M_i\simeq \St(\l)\otimes_R U_t\oplus \cdots\oplus \St(\l)\otimes_R U_i\simeq \St(\l)\otimes_R U_\l$ where $U_\l:=U_t\oplus \cdots \oplus U_i$.

We shall prove the claim by induction on $n=|\L|$. 	Assume that $n=1$. Then, $\St(1)$ is maximal, and by the previous discussion, the claim follows. Assume now that the result holds for split quasi-hereditary algebras with $|\L|=n-1$. Let $A$ be a split quasi-hereditary algebra with $|\L|=n$. Let $\l\in \L$ maximal. By the previous discussion, $M_i\simeq \St(\l)\otimes_R U_\l$, for some $U_\l\in R\proj$ and $M/M_i$ has a filtration  	\begin{align}
	0\subset M_{i-1}/M_i\subset M_{i-2}/M_i\subset \cdots \subset M/M_i.
\end{align} 
where factors in additive closure of $\St(\l)$ do not appear. 	It follows that $\Hom_A(\St(\l), M/M_i)=0$. By Corollary \ref{full subcategory in terms of split module}, $M/M_i\in A/J_\l\m$, where $J_\l$ denotes the image of $\tau_{\St(\l)}$. Since $(A/J_\l\m, \St(\mu)_{\mu\neq \l})$ is a split highest weight category with $|\L\backslash\{\l\}|=n-1$, we obtain, by induction, that 
$M/M_i$ has a filtration
\begin{align}
	0=F_n\subset \cdots \subset F_1=M/M_i \quad \text{with } F_j/F_{j+1}\simeq \St_j\otimes_R X_j, \ X_j\in R\proj.
\end{align} Here, $\L\backslash\{\l\}\rightarrow \{1, \ldots, n-1\}$ is an increasing bijection.
Put $\l\longleftrightarrow n$. So, the induced map $\L\rightarrow \{1, \ldots, n \}$ is an increasing bijection. Note that each $F_j$ is written on the form $F_j'/M_i$. Therefore,
\begin{align}
	0\subset F_n'=M_i\subset F_{n-1}'\subset \cdots\subset F_1'=M
\end{align}is a filtration of $M$ such that $F_n'\simeq \St_n\otimes_R U_\l$ and $F_j'/F_{j+1}'\simeq F_j/M_i/F_{j+1}/M_i\simeq F_j/F_{j+1}\simeq \St_j\otimes_R X_j$. 
\end{proof}

\end{appendices}

\section*{Acknowledgments}
Most of the results of this paper are contained in the author's PhD thesis \citep{thesis}, financially supported by \textit{Studienstiftung des Deutschen Volkes}. The author would like to thank Steffen Koenig for all the conversations on these topics, his comments and suggestions towards improving this manuscript.

\bibliographystyle{alphaurl}
\bibliography{bibliografia/bibarticle}

\begin{thebibliography}{{Don}87}

\bibitem[AB56]{zbMATH03116683}
M.~{Auslander} and D.~A. {Buchsbaum}.
\newblock {Homological dimension in Noetherian rings}.
\newblock {\em {Proc. Natl. Acad. Sci. USA}}, 42:36--38, 1956.
\newblock \href {https://doi.org/10.1073/pnas.42.1.36}
  {\path{doi:10.1073/pnas.42.1.36}}.

\bibitem[AB59]{zbMATH03151673}
M.~{Auslander} and D.~A. {Buchsbaum}.
\newblock {On ramification theory in Noetherian rings.}
\newblock {\em {Am. J. Math.}}, 81:749–765, 1959.
\newblock \href {https://doi.org/10.2307/2372926} {\path{doi:10.2307/2372926}}.

\bibitem[AG60]{zbMATH03190382}
M.~{Auslander} and O.~{Goldman}.
\newblock {Maximal orders.}
\newblock {\em {Trans. Am. Math. Soc.}}, 97:1–24, 1960.
\newblock \href {https://doi.org/10.2307/1993361} {\path{doi:10.2307/1993361}}.

\bibitem[{Ari}08]{zbMATH05707949}
S.~{Ariki}.
\newblock {Finite dimensional Hecke algebras}.
\newblock In {\em {Trends in representation theory of algebras and related
  topics. Proceedings of the 12th international conference on representations
  of algebras and workshop (ICRA XII), Toruń, Poland, August 15–24, 2007}},
  page 1–48. Zürich: European Mathematical Society (EMS), 2008.

\bibitem[{Bou}98]{zbMATH01194716}
N.~{Bourbaki}.
\newblock {\em {Elements of mathematics. Commutative algebra. Chapters 1--7.
  Transl. from the French.}}
\newblock Berlin: Springer, softcover edition of the 2nd printing 1989 edition,
  1998.

\bibitem[{Cou}20]{zbMATH07203140}
K.~{Coulembier}.
\newblock {The classification of blocks in BGG category \({\mathcal{O}} \)}.
\newblock {\em {Math. Z.}}, 295(1-2):821–837, 2020.
\newblock \href {https://doi.org/10.1007/s00209-019-02376-9}
  {\path{doi:10.1007/s00209-019-02376-9}}.

\bibitem[CPS88]{MR961165}
E.~Cline, B.~Parshall, and L.~Scott.
\newblock {Finite-dimensional algebras and highest weight categories}.
\newblock {\em J. Reine Angew. Math.}, 391:85–99, 1988.

\bibitem[CPS90]{CLINE1990126}
E.~Cline, B.~Parshall, and L.~Scott.
\newblock {Integral and graded quasi-hereditary algebras, {I}}.
\newblock {\em Journal of Algebra}, 131(1):126–160, 1990.
\newblock \href {https://doi.org/10.1016/0021-8693(90)90169-O}
  {\path{doi:10.1016/0021-8693(90)90169-O}}.

\bibitem[CR90]{zbMATH00046729}
C.~W. Curtis and I.~Reiner.
\newblock {\em {Methods of representation theory with applications to finite
  groups and orders. Volume 1. Paperback edition.}}
\newblock New York etc.: John Wiley \&| Sons, paperback edition edition, 1990.
\newblock \href {https://doi.org/10.1090/S0273-0979-1983-15099-1}
  {\path{doi:10.1090/S0273-0979-1983-15099-1}}.

\bibitem[CR06]{Curtis2006}
C.~W. Curtis and I.~Reiner.
\newblock {\em {Representation theory of finite groups and associative
  algebras}}.
\newblock AMS Chelsea Publishing, Providence, RI, 2006.
\newblock Reprint of the 1962 original.
\newblock \href {https://doi.org/10.1090/chel/356}
  {\path{doi:10.1090/chel/356}}.

\bibitem[Cru21]{thesis}
T.~Cruz.
\newblock {\em Algebraic analogues of resolution of singularities,
  quasi-hereditary covers and {Schur} algebras}.
\newblock PhD thesis, University of Stuttgart, 2021.
\newblock URL: \url{http://dx.doi.org/10.18419/opus-11835}.

\bibitem[Cru22]{CRUZ2022410}
T.~Cruz.
\newblock An integral theory of dominant dimension of {Noetherian} algebras.
\newblock {\em Journal of Algebra}, 591:410--499, 2022.
\newblock \href
  {https://doi.org/https://doi.org/10.1016/j.jalgebra.2021.09.029}
  {\path{doi:https://doi.org/10.1016/j.jalgebra.2021.09.029}}.

\bibitem[DK94]{MR1284468}
Y.~A. Drozd and V.~V. Kirichenko.
\newblock {\em {Finite-dimensional algebras}}.
\newblock Springer-Verlag, Berlin, 1994.
\newblock Translated from the 1980 Russian original and with an appendix by
  Vlastimil Dlab.
\newblock \href {https://doi.org/10.1007/978-3-642-76244-4}
  {\path{doi:10.1007/978-3-642-76244-4}}.

\bibitem[{Don}87]{zbMATH04031957}
S.~{Donkin}.
\newblock {On Schur algebras and related algebras. II}.
\newblock {\em {J. Algebra}}, 111:354–364, 1987.
\newblock \href {https://doi.org/10.1016/0021-8693(87)90222-5}
  {\path{doi:10.1016/0021-8693(87)90222-5}}.

\bibitem[DPS98]{zbMATH01139665}
J.~{Du}, B.~{Parshall}, and L.~{Scott}.
\newblock {Cells and $q$-Schur algebras.}
\newblock {\em {Transform. Groups}}, 3(1):33–49, 1998.
\newblock \href {https://doi.org/10.1007/BF01237838}
  {\path{doi:10.1007/BF01237838}}.

\bibitem[DR89a]{Dlab1989}
V.~{Dlab} and C.~M. {Ringel}.
\newblock {Every semiprimary ring is the endomorphism ring of a projective
  module over a quasi-hereditary ring}.
\newblock {\em {Proc. Am. Math. Soc.}}, 107(1):1--5, 1989.

\bibitem[DR89b]{Dlab1989d}
V.~Dlab and C.~M. Ringel.
\newblock {Quasi-hereditary algebras}.
\newblock {\em Illinois Journal of Mathematics}, 33(2):280–291, 1989.
\newblock \href {https://doi.org/10.1215/ijm/1255988725}
  {\path{doi:10.1215/ijm/1255988725}}.

\bibitem[DR92]{zbMATH00140218}
V.~{Dlab} and C.~M. {Ringel}.
\newblock {The module theoretical approach to quasi-hereditary algebras}.
\newblock In {\em {Representations of algebras and related topics. Proceedings
  of the Tsukuba international conference, held in Kyoto, Japan, 1990}}, page
  200–224. Cambridge: Cambridge University Press, 1992.
\newblock \href {https://doi.org/10.1017/cbo9780511661853.007}
  {\path{doi:10.1017/cbo9780511661853.007}}.

\bibitem[DR98]{zbMATH01198520}
J.~{Du} and H.~{Rui}.
\newblock {Based algebras and standard bases for quasi-hereditary algebras.}
\newblock {\em {Trans. Am. Math. Soc.}}, 350(8):3207–3235, 1998.
\newblock \href {https://doi.org/10.1090/s0002-9947-98-02305-8}
  {\path{doi:10.1090/s0002-9947-98-02305-8}}.

\bibitem[{Du}03]{zbMATH02037023}
J.~{Du}.
\newblock {Finite-dimensional algebras and standard systems.}
\newblock {\em {Algebr. Represent. Theory}}, 6(5):461–474, 2003.
\newblock \href {https://doi.org/10.1023/b:alge.0000006491.28617.d8}
  {\path{doi:10.1023/b:alge.0000006491.28617.d8}}.

\bibitem[Fac98]{MR3025306}
A.~Facchini.
\newblock {\em {Module theory}}.
\newblock {Modern Birkhäuser Classics}. Birkhäuser/Springer Basel AG, Basel,
  1998.
\newblock Endomorphism rings and direct sum decompositions in some classes of
  modules.
\newblock \href {https://doi.org/10.1007/978-3-0348-0303-8}
  {\path{doi:10.1007/978-3-0348-0303-8}}.

\bibitem[Fai73]{Faith1973}
C.~Faith.
\newblock {\em {Algebra: rings, modules and categories. {I}}}.
\newblock Springer-Verlag, New York-Heidelberg, 1973.
\newblock Die Grundlehren der mathematischen Wissenschaften, Band 190.
\newblock \href {https://doi.org/10.1007/978-3-642-80634-6}
  {\path{doi:10.1007/978-3-642-80634-6}}.

\bibitem[GL96]{zbMATH00871761}
J.~J. {Graham} and G.~I. {Lehrer}.
\newblock {Cellular algebras.}
\newblock {\em {Invent. Math.}}, 123(1):1–34, 1996.
\newblock \href {https://doi.org/10.1007/BF01232365}
  {\path{doi:10.1007/BF01232365}}.

\bibitem[{Gre}81]{zbMATH03708660}
J.~A. {Green}.
\newblock {Polynomial representations of \(GL_ n\).}
\newblock {Algebra, Proc. Conf., Carbondale 1980, Lect. Notes Math. 848,
  124-140}, 1981.

\bibitem[{Gre}93]{zbMATH00427660}
J.~A. {Green}.
\newblock {Combinatorics and the Schur algebra}.
\newblock {\em {J. Pure Appl. Algebra}}, 88(1-3):89–106, 1993.
\newblock \href {https://doi.org/10.1016/0022-4049(93)90015-L}
  {\path{doi:10.1016/0022-4049(93)90015-L}}.

\bibitem[Gre07]{zbMATH05080041}
J.~A. Green.
\newblock {\em {Polynomial representations of \(\text{GL}_n\). With an appendix
  on Schensted correspondence and Littelmann paths by K. Erdmann, J. A. Green
  and M. Schocker. 2nd corrected and augmented edition.}}, volume 830.
\newblock Berlin: Springer, 2nd corrected and augmented edition edition, 2007.
\newblock \href {https://doi.org/10.1007/3-540-46944-3}
  {\path{doi:10.1007/3-540-46944-3}}.

\bibitem[HS97]{Hilton1997}
P.~J. Hilton and U.~Stammbach.
\newblock {\em {A course in homological algebra}}, volume~4 of {\em {Graduate
  Texts in Mathematics}}.
\newblock Springer-Verlag, New York, second edition, 1997.
\newblock \href {https://doi.org/10.1007/978-1-4419-8566-8}
  {\path{doi:10.1007/978-1-4419-8566-8}}.

\bibitem[KKM05]{zbMATH02164791}
O.~{Khomenko}, S.~{Koenig}, and V.~{Mazorchuk}.
\newblock {Finitistic dimension and tilting modules for stratified algebras.}
\newblock {\em {J. Algebra}}, 286(2):456--475, 2005.
\newblock \href {https://doi.org/10.1016/j.jalgebra.2005.01.017}
  {\path{doi:10.1016/j.jalgebra.2005.01.017}}.

\bibitem[{Kle}15]{zbMATH06434447}
A.~S. {Kleshchev}.
\newblock {Affine highest weight categories and affine quasihereditary
  algebras}.
\newblock {\em {Proc. Lond. Math. Soc. (3)}}, 110(4):841--882, 2015.
\newblock \href {https://doi.org/10.1112/plms/pdv004}
  {\path{doi:10.1112/plms/pdv004}}.

\bibitem[KM20]{zbMATH07203067}
A.~{Kleshchev} and R.~{Muth}.
\newblock {Based quasi-hereditary algebras}.
\newblock {\em {J. Algebra}}, 558:504--522, 2020.
\newblock \href {https://doi.org/10.1016/j.jalgebra.2019.04.034}
  {\path{doi:10.1016/j.jalgebra.2019.04.034}}.

\bibitem[{Kra}17]{zbMATH06751586}
H.~{Krause}.
\newblock {Highest weight categories and strict polynomial functors. With an
  appendix by Cosima Aquilino}.
\newblock In {\em Representation theory -- Current trends and perspectives. In
  part based on talks given at the last joint meeting of the priority program
  in Bad Honnef, Germany, in March 2015}, pages 331--373. Z\"urich: European
  Mathematical Society (EMS), 2017.

\bibitem[KX98]{zbMATH01218863}
S.~Koenig and C.~Xi.
\newblock {On the structure of cellular algebras.}
\newblock In {\em {Algebras and modules II. Eighth international conference on
  representations of algebras, Geiranger, Norway, August 4–10, 1996}}, page
  365–386. Providence, RI: American Mathematical Society, 1998.

\bibitem[KX99a]{zbMATH01463504}
S.~Koenig and C.~Xi.
\newblock {Cellular algebras: Inflations and Morita equivalences.}
\newblock {\em {J. Lond. Math. Soc., II. Ser.}}, 60(3):700–722, 1999.
\newblock \href {https://doi.org/10.1112/S0024610799008212}
  {\path{doi:10.1112/S0024610799008212}}.

\bibitem[KX99b]{zbMATH01384521}
S.~Koenig and C.~Xi.
\newblock {When is a cellular algebra quasi-hereditary?}
\newblock {\em {Math. Ann.}}, 315(2):281–293, 1999.
\newblock \href {https://doi.org/10.1007/s002080050368}
  {\path{doi:10.1007/s002080050368}}.

\bibitem[KX00]{zbMATH01475237}
S.~Koenig and C.~Xi.
\newblock {A self-injective cellular algebra is weakly symmetric.}
\newblock {\em {J. Algebra}}, 228(1):51–59, 2000.
\newblock \href {https://doi.org/10.1006/jabr.1999.8037}
  {\path{doi:10.1006/jabr.1999.8037}}.

\bibitem[KX12]{zbMATH05994238}
S.~{Koenig} and C.~{Xi}.
\newblock {Affine cellular algebras.}
\newblock {\em {Adv. Math.}}, 229(1):139--182, 2012.
\newblock \href {https://doi.org/10.1016/j.aim.2011.08.010}
  {\path{doi:10.1016/j.aim.2011.08.010}}.

\bibitem[Lam06]{zbMATH02166983}
T.~Y. Lam.
\newblock {\em {Serre's problem on projective modules}}.
\newblock {Springer Monographs in Mathematics}. Springer-Verlag, Berlin, 2006.
\newblock \href {https://doi.org/10.1007/978-3-540-34575-6}
  {\path{doi:10.1007/978-3-540-34575-6}}.

\bibitem[Par89]{zbMATH04116809}
B.~J. Parshall.
\newblock {Finite dimensional algebras and algebraic groups.}
\newblock {Classical groups and related topics, Proc. Conf., Beijing/China
  1987, Contemp. Math. 82, 97-114}, 1989.
\newblock \href {https://doi.org/10.1090/conm/082/982281}
  {\path{doi:10.1090/conm/082/982281}}.

\bibitem[PS88]{PS88}
B.~Parshall and L.~Scott.
\newblock {Derived categories, quasi-hereditary algebras, and algebraic
  groups}.
\newblock {\em Carlton Univ. Lecture Notes in Math.}, 3(3):1–104, 1988.

\bibitem[Rin91]{MR1128706}
C.~M. Ringel.
\newblock {The category of modules with good filtrations over a
  quasi-hereditary algebra has almost split sequences}.
\newblock {\em Math. Z.}, 208(2):209–223, 1991.
\newblock \href {https://doi.org/10.1007/BF02571521}
  {\path{doi:10.1007/BF02571521}}.

\bibitem[{Rot}09]{Rotman2009a}
J.~J. {Rotman}.
\newblock {\em {An introduction to homological algebra. 2nd ed.}}
\newblock Berlin: Springer, 2nd ed. edition, 2009.
\newblock \href {https://doi.org/10.1007/b98977} {\path{doi:10.1007/b98977}}.

\bibitem[Rou08]{Rouquier2008}
R.~Rouquier.
\newblock {{$q$}-{S}chur algebras and complex reflection groups}.
\newblock {\em Moscow Mathematical Journal}, 8(1):119–158, 184, 2008.
\newblock \href {https://doi.org/10.17323/1609-4514-2008-8-1-119-158}
  {\path{doi:10.17323/1609-4514-2008-8-1-119-158}}.

\bibitem[{Sch}01]{zbMATH02662157}
I.~{Schur}.
\newblock {Über eine Klasse von Matrizen, die sich einer gegebenen Matrix
  zuordnen lassen.}
\newblock {Diss. Berlin. 76 S.}, 1901.
\newblock URL: \url{http://eudml.org/doc/203316}.

\bibitem[Swa78]{zbMATH03627311}
R.~G. Swan.
\newblock {Projective modules over Laurent polynomial rings.}
\newblock {\em {Trans. Am. Math. Soc.}}, 237:111–120, 1978.
\newblock \href {https://doi.org/10.2307/1997613} {\path{doi:10.2307/1997613}}.

\bibitem[{Tot}97]{zbMATH00966941}
B.~{Totaro}.
\newblock {Projective resolutions of representations of \(\text{GL}(n)\)}.
\newblock {\em {J. Reine Angew. Math.}}, 482:1--13, 1997.
\newblock \href {https://doi.org/10.1515/crll.1997.482.1}
  {\path{doi:10.1515/crll.1997.482.1}}.

\end{thebibliography}

\Address
\end{document}